\documentclass[11pt]{amsart}

\usepackage{amsmath}
\usepackage{amssymb}
\usepackage{amsthm}
\usepackage{mathtools}
\usepackage[dvipsnames]{xcolor}
\usepackage[pdfusetitle,colorlinks=true,linkcolor=MidnightBlue,citecolor=ForestGreen,urlcolor=BrickRed]{hyperref}

\allowdisplaybreaks
\theoremstyle{plain}
\newtheorem{thm}{Theorem}[section]
\newtheorem{lem}[thm]{Lemma}
\newtheorem{prop}[thm]{Proposition}
\newtheorem{cor}[thm]{Corollary}

\theoremstyle{definition}

\theoremstyle{remark}

\newcommand{\C}{\mathbb{C}}
\newcommand{\R}{\mathbb{R}}

\newcommand{\E}{\mathbb{E}}
\newcommand{\Prob}{\mathbb{P}}
\newcommand{\ind}{\mathbf{1}}
\newcommand{\dd}{\,\mathrm{d}}

\newcommand{\ee}{\mathrm{e}}
\newcommand{\abs}[1]{\lvert #1\rvert}
\newcommand{\norm}[1]{\lVert #1\rVert}

\newcommand{\op}{\mathrm{op}}
\newcommand{\HS}{\mathrm{HS}}

\DeclareMathOperator{\Tr}{Tr}
\DeclareMathOperator{\rank}{rank}
\DeclareMathOperator{\diag}{diag}

\DeclareMathOperator{\Ai}{Ai}

\begin{document}

\title[One-Sided Product-Scale Upper Bounds]
{One-Sided Product-Scale Upper Bounds for Nested Complex Wishart Extremes}
\author{Xiufan Yang}
\address{School of Science, Nanjing University of Posts and Telecommunications, Nanjing 210023, China}
\email{B25100020@njupt.edu.cn}

\hypersetup{
  pdftitle={One-Sided Product-Scale Upper Bounds for Nested Complex Wishart Extremes},
  pdfauthor={Xiufan Yang},
  pdfsubject={Finite-size soft-edge decorrelation for nested complex Wishart matrices},
  pdfkeywords={complex Wishart matrices, largest eigenvalue, determinantal kernels, finite-size decorrelation, Gaussian extremes}
}

\date{}

\begin{abstract}
I study the largest eigenvalue \(\Lambda_j\) of \(X^{(j)}(X^{(j)})^*\), where \(X^{(j)}\) is the \(M_j\times j\) northwest rectangle of one infinite complex Gaussian array. For two upper-tail events and, separately, for two lower-tail events at the largest-eigenvalue soft edge, I prove finite-\(N\) one-sided product-scale upper bounds for their joint probabilities. The logarithmic walls have polynomially small marginals, so an additive \(o(1)\) covariance estimate may be much larger than the vanishing product that must be controlled. Uniformly for levels in a short macroscopic window, bounded deterministic shifts, and separations at least \(N^{2/3+\epsilon}\), each joint probability is at most \((1+o(1))\) times the product of its marginals. The classical \(N^{2/3}\) correlation window corresponds to order-one extended-Airy time; the theorem works at supercritical separation. Exact Laguerre operators along admissible row--column paths provide the finite-dimensional input. The upper-tail proof uses occupancy counts and cross-block trace estimates, whereas the lower-tail proof uses gap determinants, diagonal resolvents, and a Schur complement. I transfer the result between the half-integer Laguerre and physical rectangular normalizations and deduce intrinsic first- and second-moment estimates for separated sparse grids under the rarity, count-growth, spacing, and macroscopic-window budget.
\end{abstract}

\maketitle

\section{Introduction and main results}\label{sec:gaussian-intro}

Most soft-edge universality results concern one random matrix at one level. Questions about repeated or almost-sure extremes along a nested family are different: they depend on joint rare-event probabilities across many strongly coupled levels. In the Wishart path considered here, consecutive matrices share almost all of their raw entries. Fixed-level Tracy--Widom convergence therefore gives no control of the required two-level probabilities, and an additive \(o(1)\) covariance estimate may be much larger than the product of the polynomially small marginals. The purpose of this paper is to obtain a two-level estimate at that product scale in the complex Gaussian model.

\subsection*{The model and the question}
Let \((\xi_{rs})_{r,s\geq1}\) be one infinite array of independent centered circular complex Gaussian variables with \(\E\abs{\xi_{rs}}^2=1\). For a nondecreasing integer sequence \((M_j)\), the common-array nested rectangular process is
\[
X^{(j)}=(\xi_{rs})_{1\leq r\leq M_j,\ 1\leq s\leq j},
\qquad
\Lambda_j=\lambda_{\max}\!\left(X^{(j)}(X^{(j)})^*\right).
\]
Thus the matrices at different levels share all entries in the smaller northwest rectangle. The observable \(\Lambda_j\) is the largest squared singular value of \(X^{(j)}\). Both physical dimensions can grow, and the identity of the smaller side can change between two levels. I call that endpoint configuration orientation reversing when both dimensions increase while the smaller physical side at the earlier endpoint becomes the larger physical side at the later endpoint. It is more specific than merely passing through a square rectangle.

The principal question is whether two logarithmically rare upper-tail events, and separately two logarithmically rare lower-tail events, for this shared-array soft edge satisfy a one-sided joint-probability upper bound on the scale of their marginal product. I assume
\[
\frac{M_j}{j}\longrightarrow\gamma>0,
\qquad
0\leq M_{j+1}-M_j\leq K,
\]
so the path has a limiting rectangular aspect ratio and a bounded number of row additions per column addition. Theorem~\ref{thm:nested-wishart-decorrelation} treats pairs of levels in a short macroscopic window \([N,(1+\eta_0)N]\), with bounded deterministic shifts of the logarithmic thresholds. The short window keeps all endpoint dimensions uniformly comparable to \(N\), while still allowing a polynomial number of separated tests.

\subsection*{Probability scale and correlation length}
The relevant marginals vanish polynomially. The one-time Laguerre tail asymptotics used here convert the upper threshold \(a(\log j)^{2/3}\) and the lower threshold \(-b(\log j)^{1/3}\) into the rarity exponents
\[
\kappa_+=\frac43a^{3/2},
\qquad
\kappa_-=\frac{b^3}{12};
\]
more precisely, throughout the short window the two marginal probabilities are \(N^{-\kappa_++o(1)}\) and \(N^{-\kappa_-+o(1)}\), uniformly under bounded deterministic shifts. These constants come from the complex Laguerre moderate-deviation tails in \cite{BBBK24}. The shift uniformity matters because later counting arguments compare many nearby deterministic walls, but an \(O(1)\) shift does not change the leading power of \(N\).

For a sign \(\mathfrak s\in\{+,-\}\), write \(H_j^+=E_j^+\), \(H_j^-=G_j^-\), and \(u_j^{\mathfrak s}=\Prob(H_j^{\mathfrak s})\). Since
\[
u_n^{\mathfrak s}u_m^{\mathfrak s}
=
N^{-2\kappa_{\mathfrak s}+o(1)}\longrightarrow0,
\]
an additive estimate of the form
\(\Prob(H_n^{\mathfrak s}\cap H_m^{\mathfrak s})-u_n^{\mathfrak s}u_m^{\mathfrak s}=o(1)\)
can be too coarse: its error may dominate the quantity being estimated. The required precision is instead the one-sided product-scale upper bound
\[
\Prob(H_n^{\mathfrak s}\cap H_m^{\mathfrak s})
\leq
(1+o(1))u_n^{\mathfrak s}u_m^{\mathfrak s}.
\]
Only this upper direction is asserted. It is the direction needed to control the off-diagonal contribution to sparse-grid second moments.

The separation scale is dictated by the classical soft-edge correlation window. For sorted endpoint dimensions \((p,q)\) and \((P,Q)\), put \(A=P-p\), \(B=Q-q\), and
\[
\delta_{p,q}
=
\frac{(\sqrt p+\sqrt q)^{2/3}}{(pq)^{1/3}},
\qquad
\tau_{p,q;A,B}
=
\frac{A/p+B/q}{2\delta_{p,q}}.
\]
Here \(\tau\) is the extended-Airy time associated with a discrete row--column displacement in the normalization of this paper. When all dimensions are comparable to \(N\), one has \(\delta_{p,q}\asymp N^{-1/3}\) and \(\tau\asymp(A+B)/N^{2/3}\). Consequently, displacement of order \(N^{2/3}\) corresponds to order-one Airy time, whereas the theorem's separation \(m-n\geq N^{2/3+\epsilon}\) forces \(\tau\) to grow at least as a positive power of \(N\). The \(N^{2/3}\) window and the associated correlation transition are classical in Laguerre and related soft-edge processes; see \cite[Proposition~5]{ForresterNagao2011} and, for a recent Wigner-minor transition theorem, \cite[Theorems~2.2--2.4]{BaoCipolloniErdosHenheikKolupaiev2025Transition}.

Fixed-orientation Wishart/Laguerre determinantal structure and the general coupled-chain mechanism are prior. Rank-one LUE growth and consecutive Wishart-minor kernels are treated in \cite[Sections~3.2--3.3]{ForresterNagao2011} and \cite[Theorems~1 and~3]{AdlerVanMoerbekeWang2013}; related time--level and space-like path models appear in \cite[Theorem~1.3]{FerrariFrings2010} and \cite[Theorem~1.3]{BenigniHungWu2026}. The Eynard--Mehta, L-ensemble, and Schur-process frameworks provide the general determinant-chain mechanism \cite{EynardMehta1998,BorodinRains2005,OkounkovReshetikhin2003}. The closest product-scale probability comparator is the recent two-event estimate for logarithmic upper- and lower-tail events of nested Wigner minors in \cite[Proposition~2.2]{BaoCipolloniErdosHenheikKolupaiev2025LFL}.

\subsection*{Relation to previous work and proof strategy}
Against this background, the headline output is the Wishart-specific one-sided product-scale upper estimate in Theorem~\ref{thm:nested-wishart-decorrelation}, with its common-array row--column geometry, short-window endpoint range, supercritical separation, and bounded-shift uniformity. The determinantal law along an admissible monotone row--column path and the endpoint formula for orientation reversal provide supporting finite-dimensional structure. The fixed-orientation cases are classical in substance; direct priority is not asserted for the reversing formula. I do not use or claim a distinction between the full admissible-path law and possible continuous Schur-process specializations after reparametrization or determinantal gauge changes. The Airy-time coefficient is used only as a model- and coordinate-specific conversion to the classical extended-Airy framework and its \(N^{2/3}\) correlation window.

The proof uses two soft-edge coordinates with different roles. The half-integer Laguerre normalization
\(\overline\chi_j=(\Lambda_j-\mu_j)/\sigma_j\)
uses \(p_j+\tfrac12\) and \(q_j+\tfrac12\) in its center and scale; it aligns the finite Laguerre functions with the Airy approximation and the exact two-level operators. The project normalization
\[
\chi_j^{\mathrm{proj}}
=
\frac{\Lambda_j-(\sqrt{M_j}+\sqrt j)^2}
{(\sqrt{M_j}+\sqrt j)(M_j^{-1/2}+j^{-1/2})^{1/3}}
\]
uses the physical rectangle dimensions and is the coordinate in the separated-grid corollary. On the macroscopic window, the two centers differ by \(O(1)\) and the two scales have ratio \(1+O(N^{-1})\). Hence the logarithmic thresholds differ by \(o(1)\) after standardization, and the theorem transfers its marginal and pair-probability conclusions to \(\chi_j^{\mathrm{proj}}\). The lower-wall regularity clause remains stated in the half-integer coordinate.

The upper and lower events require different determinantal reductions. For the upper event, let \(Z_j^+\) count particles above the upper threshold and let \(T_{10}^+,T_{01}^+\) be the restricted cross blocks. The exact mixed-count identity is
\[
\E[Z_n^+Z_m^+]
=
\E Z_n^+\,\E Z_m^+
-
\Tr(T_{01}^+T_{10}^+).
\]
The cross-block trace estimates carry the upper rarity exponent, so the trace correction is negligible compared with \(u_n^+u_m^+\); the inequality \(\ind_{E_n^+}\ind_{E_m^+}\leq Z_n^+Z_m^+\) then gives the required upper bound. The lower event is instead a gap event. After writing the diagonal gap blocks as \(A_n^-=I-\widehat K_{00}\) and \(D_m^-=I-\widehat K_{11}\), the central interaction is controlled by
\[
\norm{(D_m^-)^{-1}T_{10}^-(A_n^-)^{-1}T_{01}^-}_1
\leq
N^{-2\epsilon+o(1)}.
\]
The resolvent bounds and this Schur-complement estimate compare the two-level gap determinant with the product of its one-level gap determinants. This distinction explains why upper occupancy estimates alone do not treat the lower tail.

\subsection*{Main result}
The next theorem is the main one-sided probability result. Its opening block states the marginal powers and the product-scale upper estimates for upper/upper pairs and, separately, lower/lower pairs. The remaining clauses record the two-level Laguerre operators, their norm decay, the lower-gap resolvent control, and the relative stability of lower walls used in the proof and in later counting arguments.

\begin{thm}[Finite-\(N\) product-scale decorrelation]\label{thm:nested-wishart-decorrelation}
Let \(\gamma>0\). Let \((M_j)_{j\geq1}\) be a nondecreasing sequence of positive integers such that
\[
\frac{M_j}{j}\longrightarrow\gamma,
\qquad
0\leq M_{j+1}-M_j\leq K
\]
for a fixed finite \(K\). Let \(X^{(j)}\) be the \(M_j\times j\) northwest rectangle of one infinite array of independent centered circular complex Gaussian variables of second absolute moment one. Put
\[
\Lambda_j=\lambda_{\max}(X^{(j)}(X^{(j)})^*),
\qquad
p_j=\min(M_j,j),
\qquad
q_j=\max(M_j,j),
\]
\[
\mu_j
=
\left(
\sqrt{p_j+\frac12}
+
\sqrt{q_j+\frac12}
\right)^2,
\]
\[
\sigma_j
=
\left(
\sqrt{p_j+\frac12}
+
\sqrt{q_j+\frac12}
\right)
\left[
\left(p_j+\frac12\right)^{-1/2}
+
\left(q_j+\frac12\right)^{-1/2}
\right]^{1/3},
\]
and
\[
\overline\chi_j=\frac{\Lambda_j-\mu_j}{\sigma_j}.
\]
Fix \(a,b>0\), put
\[
\kappa_+=\frac43a^{3/2},
\qquad
\kappa_-=\frac{b^3}{12},
\]
and assume
\[
\max(\kappa_+,\kappa_-)<\frac13.
\]
Fix
\[
0<\epsilon<
\frac13-\max(\kappa_+,\kappa_-)
\]
and \(H_{\mathrm{shift}}<\infty\). There are constants \(\eta_0>0\) and \(c>0\), depending only on \(\gamma\) and \(K\), such that the following assertions hold as \(N\to\infty\), uniformly over integers \(n,m\) satisfying
\[
N\leq n<m\leq(1+\eta_0)N,
\qquad
m-n\geq N^{2/3+\epsilon},
\]
and deterministic shifts
\[
h_{n,+},h_{m,+},h_{n,-},h_{m,-}
\]
with absolute values at most \(H_{\mathrm{shift}}\). Define
\[
E_j^+
=
\left\{
\overline\chi_j
\geq
a(\log j)^{2/3}+h_{j,+}
\right\},
\]
\[
G_j^-
=
\left\{
\overline\chi_j
\leq
-b(\log j)^{1/3}+h_{j,-}
\right\},
\]
and
\[
u_j^+=\Prob(E_j^+),
\qquad
u_j^-=\Prob(G_j^-).
\]
Then
\[
u_j^+=N^{-\kappa_++o(1)},
\qquad
u_j^-=N^{-\kappa_-+o(1)}
\]
uniformly for \(j\) in the displayed window, and
\[
\Prob(E_n^+\cap E_m^+)
\leq
(1+o(1))u_n^+u_m^+,
\]
\[
\Prob(G_n^-\cap G_m^-)
\leq
(1+o(1))u_n^-u_m^-.
\]
All deterministic \(o(1)\) terms are uniform over the displayed endpoint and shift choices. The preceding marginal and pair-probability conclusions also hold after replacing \(\overline\chi_j\) by
\[
\chi_j^{\mathrm{proj}}
=
\frac{
\Lambda_j-(\sqrt{M_j}+\sqrt j)^2
}{
(\sqrt{M_j}+\sqrt j)
(M_j^{-1/2}+j^{-1/2})^{1/3}
}.
\]

The exact two-time geometry is as follows. Put
\[
(p,q)=(p_n,q_n),
\qquad
(P,Q)=(p_m,q_m),
\qquad
A=P-p,
\qquad
B=Q-q,
\]
\[
H=A+B,
\qquad
\rho=\frac12\left(\frac Ap+\frac Bq\right).
\]
Then
\[
H=(m-n)+(M_m-M_n),
\qquad
H\geq N^{2/3+\epsilon},
\qquad
\rho\geq cN^{-1/3+\epsilon}.
\]
After a levelwise determinantal gauge, the equal-level kernels are the Laguerre orthogonal projections, and the cross blocks are
\[
\widehat K_{10}(x,y)
=
\sum_{\ell=1}^p
d_\ell
\phi_{P-\ell,Q-\ell}(x)
\phi_{p-\ell,q-\ell}(y),
\]
\[
\widehat K_{01}(x,y)
=
-\sum_{v=0}^{\infty}
d_{-v}^{-1}
\phi_{p+v,q+v}(x)
\phi_{P+v,Q+v}(y),
\]
where
\[
\phi_{r,s}(x)
=
\sqrt{\frac{r!}{s!}}\,
\exp(-x/2)x^{(s-r)/2}L_r^{(s-r)}(x)
\]
and
\[
d_\ell
=
\sqrt{
\frac{
(P-\ell)!(Q-\ell)!p!q!
}{
(p-\ell)!(q-\ell)!P!Q!
}
}.
\]
These formulas cover fixed orientation, orientation reversal between the two rectangles, square endpoints, and simultaneous row-and-column growth.

Let \(I_j^+\) be the upper half-line defining \(E_j^+\), and let
\(T_{10}^+,T_{01}^+\) be the cross blocks restricted to
\(I_m^+\times I_n^+\). Then
\[
\max\left(
\norm{T_{10}^+}_{\op},
\norm{T_{10}^+}_{\HS},
\norm{T_{10}^+}_1,
\norm{T_{01}^+}_{\op},
\norm{T_{01}^+}_{\HS},
\norm{T_{01}^+}_1
\right)
\leq
N^{-\kappa_+-\epsilon+o(1)}.
\]
Let \(B_j^-\) be the forbidden upper half-line whose gap event is \(G_j^-\), and let
\(T_{10}^-,T_{01}^-\) be the correspondingly restricted cross blocks. Then
\[
\max\left(
\norm{T_{10}^-}_{\op},
\norm{T_{10}^-}_{\HS},
\norm{T_{10}^-}_1,
\norm{T_{01}^-}_{\op},
\norm{T_{01}^-}_{\HS},
\norm{T_{01}^-}_1
\right)
\leq
N^{-\epsilon+o(1)}.
\]
Writing
\[
A_n^-=I-\widehat K_{00}\big|_{B_n^-},
\qquad
D_m^-=I-\widehat K_{11}\big|_{B_m^-},
\]
one also has
\[
\norm{(A_n^-)^{-1}}_{\op}
+
\norm{(D_m^-)^{-1}}_{\op}
\leq
N^{o(1)}
\]
and
\[
\norm{
(D_m^-)^{-1}T_{10}^-
(A_n^-)^{-1}T_{01}^-
}_1
\leq
N^{-2\epsilon+o(1)}.
\]

Finally, fix \(D_{\mathrm{wall}}<\infty\). There is \(R_*>1\) such that the following relative lower-wall estimate holds uniformly for every \(j\) in the displayed macroscopic window. Let
\[
F_j(t)=\Prob(\overline\chi_j\leq t).
\]
If \(0<\ell_{\mathrm{wall}}\leq1\) and
\[
[s,s+\ell_{\mathrm{wall}}]
\subseteq
[-D_{\mathrm{wall}}(\log N)^{1/3},-R_*],
\]
then
\[
F_j(s+\ell_{\mathrm{wall}})-F_j(s)
\leq
\ell_{\mathrm{wall}}N^{o(1)}
F_j(s+\ell_{\mathrm{wall}}).
\]
Consequently, if
\[
\ell_{\mathrm{wall}}\leq N^{-\alpha_{\mathrm{wall}}}
\]
for a fixed \(\alpha_{\mathrm{wall}}>0\), then
\[
1
\leq
\frac{F_j(s+\ell_{\mathrm{wall}})}{F_j(s)}
\leq
1+N^{-\alpha_{\mathrm{wall}}+o(1)}.
\]
Moreover, for every fixed \(C_{\mathrm{lad}}<\infty\), both estimates hold simultaneously and uniformly over any deterministic family of at most
\(C_{\mathrm{lad}}\log N\) wall intervals satisfying the displayed assumptions.
\end{thm}

The event inequalities above are uniform only in the displayed short macroscopic window, at the displayed supercritical separations, and under the displayed bounded shifts. The phrase finite-\(N\) refers to estimates formulated for sufficiently large finite matrix sizes with uniform deterministic errors; no complementary lower pair bound is asserted. The theorem treats the upper and lower tails of the largest-eigenvalue soft edge. It does not concern the smallest eigenvalue.

An intrinsic consequence is a separated-grid moment estimate. For a sign \(\mathfrak s\), the rarity exponent \(\kappa_{\mathfrak s}\) is the marginal cost: one event has probability \(N^{-\kappa_{\mathfrak s}+o(1)}\). Choosing the explicit arithmetic grid with
\[
\abs{\mathcal T_N^{\mathfrak s}}
=
\left\lfloor N^{\kappa_{\mathfrak s}+\theta_{\mathfrak s}}\right\rfloor
\]
points offsets that rarity and targets a first moment of order \(N^{\theta_{\mathfrak s}}\). The step
\(\lceil N^{2/3+\epsilon_{\mathfrak s}}\rceil\)
pays the cost of placing every distinct pair beyond the classical correlation window. The resulting arithmetic grid has span at most
\[
N^{\kappa_{\mathfrak s}+\theta_{\mathfrak s}+2/3+\epsilon_{\mathfrak s}+o(1)},
\]
so it fits inside a short macroscopic window exactly under
\(\kappa_{\mathfrak s}+\theta_{\mathfrak s}+\epsilon_{\mathfrak s}<1/3\).
In the second moment, the diagonal terms reproduce the first moment, while the theorem's one-sided pair estimate controls the off-diagonal sum uniformly under bounded deterministic shifts. Corollary~\ref{cor:gaussian-grid-moments}, stated after the proof of Theorem~\ref{thm:nested-wishart-decorrelation}, makes this intrinsic count statement precise. The companion paper \cite{YangLFL} later uses its conditional form as a Gaussian input.

Section~\ref{sec:gaussian-growth} derives the one-step transitions, the determinantal law along admissible monotone paths, the conditional kernel, and the fixed- and reversing-orientation two-level formulas. Section~\ref{sec:airy-tails} identifies the Airy-time conversion, proves the uniform Laguerre-to-Airy estimates and one-time logarithmic tails, and then proves Theorem~\ref{thm:nested-wishart-decorrelation} and Corollary~\ref{cor:gaussian-grid-moments}. Corollary~\ref{cor:gaussian-supercritical-margins} records separate one-time supercritical margins. Appendix~\ref{app:gaussian-finite-particle-algebra} contains the finite-particle determinant calculations, and Appendix~\ref{app:gaussian-airy-envelopes} records the Airy envelope estimates.

\section{Gaussian row-column growth and determinantal kernels}\label{sec:gaussian-growth}

\emph{Goal.} This section derives the exact finite-$N$ spectral law of the common Gaussian row--column growth and the two-level Laguerre blocks used later in the probability estimates. \emph{Obstacle.} Along a rectangular path the number of positive squared singular values can change, and the identity of the smaller physical side can switch between endpoints. \emph{Mechanism.} I combine the Gaussian rank-one update law with fixed-particle Eynard--Mehta after padding each level by deterministic virtual states, then evaluate the resulting finite-rank term in Laguerre coordinates. \emph{Output.} The section supplies unconditional and forward conditional kernels on the admissible finite elementary up-right paths in the statements, together with fixed- and reversing-orientation two-level formulas that feed the subsequent Airy-time and decorrelation analysis.

\subsection{One-step Gaussian growth and shell commutation}

For an $m_t$-by-$n_t$ rectangle, $p_t=\min(m_t,n_t)$ is the number of positive squared singular values; any remaining eigenvalues of the larger Gram matrix are zero. An increment of a physical side that is already weakly larger leaves $p_t$ unchanged and is a same-size step, whereas an increment of the strictly smaller side increases $p_t$ by one and is a birth step. The fixed-particle determinant mechanism requires one cardinality at every level, so the later construction pads the $p_t$ real particles to the terminal cardinality by deterministic virtual states. These labels are neither eigenvalues nor random particles.

The determinant-chain mechanism used below is the Eynard--Mehta construction \cite[pp.~4451--4455]{EynardMehta1998}; a theorem-level finite-chain formulation is given in \cite[Theorem~1.4]{BorodinRains2005}, and broader interlacing-chain geometry appears in \cite[Theorem~1]{OkounkovReshetikhin2003}. These precedents supply the general mechanism. The work here is its finite-dimensional specialization to the common northwest Gaussian row--column field and the normalization required by the later Laguerre analysis.

\begin{lem}[Fixed-particle Eynard--Mehta lemma]\label{lem:fixed-particle-em}
Let $T$ be a nonnegative integer and $P$ a positive integer. For each $t\in\{0,\ldots,T\}$, let $(E_t,\mu_t)$ be a sigma-finite measure space. Let
\[
 \phi_i\colon E_0\to\C,\qquad
 \psi_j\colon E_T\to\C,\qquad 1\leq i,j\leq P,
\]
and let $W_t\colon E_t\times E_{t+1}\to\C$ for $0\leq t<T$. For $s<t$, define $W_{s,t}$ by convolution of $W_s,\ldots,W_{t-1}$. Define
\[
 F_i^{(t)}=\phi_i*W_{0,t},
 \qquad
 G_j^{(t)}=W_{t,T}*\psi_j,
\]
with $F_i^{(0)}=\phi_i$ and $G_j^{(T)}=\psi_j$.

Suppose that all integrals below converge absolutely and that
\[
 \det[\phi_i(x_j^{(0)})]
 \prod_{t=0}^{T-1}\det[W_t(x_i^{(t)},x_j^{(t+1)})]
 \det[\psi_i(x_j^{(T)})]
\]
is nonnegative almost everywhere and has positive finite integral $Z$ with respect to
\[
 \prod_{t=0}^T\prod_{a=1}^P\mu_t(\dd x_a^{(t)}).
\]
Define
\[
 M_{ij}
 =
 \int_{E_t}F_i^{(t)}(x)G_j^{(t)}(x)\,\mu_t(\dd x).
\]
This value is independent of $t$. Then $M$ is invertible,
\[
 Z=(P!)^{T+1}\det M,
\]
and the point process on the disjoint union of $E_0,\ldots,E_T$ formed by the $P$ coordinates at every time is determinantal relative to the direct-sum measure, with correlation kernel
\[
 K(s,x;t,y)
 =
 -\ind_{\{s<t\}}W_{s,t}(x,y)
 +
 \sum_{i,j=1}^P
 G_j^{(s)}(x)(M^{-1})_{ji}F_i^{(t)}(y).
\]
\end{lem}

\begin{proof}
The proof is recorded in Appendix~\ref{app:proof-fixed-particle-em}.
\end{proof}

\begin{prop}[Rank-one Markov description of complex Gaussian growth]\label{prop:rank-one-growth}
Let $(\xi_{ij})_{i,j\geq1}$ be independent centered circular complex Gaussian variables satisfying $\E\abs{\xi_{ij}}^2=1$. Let $((m_t,n_t))_{t\geq0}$ be a deterministic sequence of positive integer pairs such that exactly one coordinate increases by one at each step. Put
\[
 A_t=(\xi_{ij})_{\substack{1\leq i\leq m_t\\1\leq j\leq n_t}},
\]
let $\Lambda_t$ be the decreasing list of positive squared singular values of $A_t$, and let $\mathcal E_t$ be the sigma-algebra generated by the entries of $A_t$.

For $d\geq1$ and a decreasing positive list
\[
 \ell=(\ell_1,\ldots,\ell_l),\qquad l\leq d,
\]
define $R_d(\ell)$ as the law of the decreasing list of positive eigenvalues of
\[
 \diag(\ell_1,\ldots,\ell_l,0,\ldots,0)+gg^*,
\]
where $g$ is a standard complex Gaussian vector in $\C^d$. Then $(\Lambda_t)_{t\geq0}$ is a time-inhomogeneous Markov chain relative to $(\mathcal E_t)_{t\geq0}$. Conditionally on $\mathcal E_t$, a column addition applies $R_{m_t}$ and a row addition applies $R_{n_t}$.

If $b_1,\ldots,b_d$ are the diagonal entries $\ell_1,\ldots,\ell_l,0,\ldots,0$ and $q\notin\{b_1,\ldots,b_d\}$, then
\[
 \begin{aligned}
 &\det(qI_d-\diag(b_1,\ldots,b_d)-gg^*)\\
 &\qquad=\det(qI_d-\diag(b_1,\ldots,b_d))
 \left(1-\sum_{i=1}^d\frac{W_i}{q-b_i}\right),
 \end{aligned}
\]
where $W_i=\abs{g_i}^2$ are independent exponential random variables with density $\ee^{-w}$ on $w>0$. The eigenvalues of
\[
 \diag(b_1,\ldots,b_d)+gg^*
\]
interlace those of $\diag(b_1,\ldots,b_d)$.

Consequently, let $(M_N)_{N\geq1}$ be nondecreasing and define
\[
 X^{(N)}=(\xi_{ij})_{\substack{1\leq i\leq M_N\\1\leq j\leq N}},
 \qquad
 r_N=M_{N+1}-M_N.
\]
Conditionally on $X^{(N)}$, the law of $\Lambda_{N+1}$ is obtained from $\Lambda_N$ by applying $R_{M_N}$ once and then $R_{N+1}$ exactly $r_N$ times. If $r_N\leq K$ for all $N$, at most $K+1$ rank-one kernels occur in each transition.
\end{prop}

\begin{proof}
For a positive-semidefinite Hermitian matrix $B$, write $\operatorname{Spec}_+(B)$ for its positive eigenvalues, listed in decreasing order with multiplicity. Let $\mu_d$ denote the standard complex Gaussian probability law on $\C^d$.

\noindent\textbf{Step 1.} We prove the conditional transition for a column addition. Suppose
\[
 (m_{t+1},n_{t+1})=(m_t,n_t+1),
\]
and abbreviate $m=m_t$, $n=n_t$, and $A=A_t$. The newly revealed entries form a vector $x\in\C^m$ with law $\mu_m$, independent of $\mathcal E_t$. Direct multiplication gives
\[
 A_{t+1}A_{t+1}^*=AA^*+xx^*.
\]
For every bounded Borel function $f$ of finite positive eigenvalue lists,
\begin{equation}\label{eq:column-transition}
 \E[f(\Lambda_{t+1})\mid\mathcal E_t]
 =
 \int f(\operatorname{Spec}_+(AA^*+xx^*))\,\mu_m(\dd x).
\end{equation}
The right side depends on $AA^*$ only through its eigenvalues. If $V$ is unitary, then $V^*x$ has law $\mu_m$, and
\[
 \operatorname{Spec}_+(V^*AA^*V+xx^*)
 =
 \operatorname{Spec}_+(AA^*+(Vx)(Vx)^*).
\]
Changing variables proves unitary conjugation invariance. The eigenvalues of $AA^*$ are the entries of $\Lambda_t$ followed by zeros. Hence
\begin{equation}\label{eq:column-kernel}
 \E[f(\Lambda_{t+1})\mid\mathcal E_t]
 =
 \int f(y)\,R_{m_t}(\Lambda_t,\dd y).
\end{equation}
No choice of eigenvectors is required.

\noindent\textbf{Step 2.} We prove the row transition. Suppose
\[
 (m_{t+1},n_{t+1})=(m_t+1,n_t).
\]
The fresh row is $x^*$, where $x\in\C^{n_t}$ has law $\mu_{n_t}$ and is independent of $\mathcal E_t$. Then
\[
 A_{t+1}^*A_{t+1}=A_t^*A_t+xx^*.
\]
The positive eigenvalues of $A_t^*A_t$ and $A_tA_t^*$ agree with multiplicity. If $A_t^*A_tv=\lambda v$ with $\lambda>0$, then $A_tv\neq0$ and
\[
 A_tA_t^*(A_tv)=\lambda A_tv.
\]
The converse follows by applying $A_t^*$, and the two maps are inverse up to multiplication by $\lambda$ on the corresponding eigenspaces. Repeating the invariant integral argument gives
\begin{equation}\label{eq:row-kernel}
 \E[f(\Lambda_{t+1})\mid\mathcal E_t]
 =
 \int f(y)\,R_{n_t}(\Lambda_t,\dd y).
\end{equation}
Equations \eqref{eq:column-kernel} and \eqref{eq:row-kernel} prove the Markov property.

\noindent\textbf{Step 3.} We establish the determinant formula. Put
\[
 D=\diag(b_1,\ldots,b_d),
 \qquad
 A=qI_d-D.
\]
The assumption on $q$ makes $A$ invertible. Factoring $A$ gives
\begin{equation}\label{eq:rank-one-factor}
 \det(A-gg^*)=\det(A)\det(I_d-A^{-1}gg^*).
\end{equation}
For column vectors $u,v$, multilinearity of the determinant gives
\begin{equation}\label{eq:matrix-det-lemma}
 \det(I_d-uv^*)=1-v^*u.
\end{equation}
Indeed, the term with no replaced column is one, every term with at least two replaced columns vanishes because those columns are proportional to $u$, and the term with only column $j$ replaced is $-\overline{v_j}u_j$. Applying \eqref{eq:matrix-det-lemma} to \eqref{eq:rank-one-factor} yields
\[
 \det(qI_d-D-gg^*)
 =
 \det(qI_d-D)
 \left(1-g^*(qI_d-D)^{-1}g\right),
\]
which is the asserted formula.

The coordinates $g_i$ are independent. Circularity and $\E\abs{g_i}^2=1$ imply that their real and imaginary parts are independent $N(0,1/2)$ variables. Polar coordinates therefore give
\[
 \Prob(\abs{g_i}^2\in\dd w)=\ee^{-w}\dd w,\qquad w>0.
\]

\noindent\textbf{Step 4.} We prove interlacing from spectral subspaces. Order the eigenvalues of $D$ as
\[
 b_1\geq\cdots\geq b_d
\]
and those of $C=D+gg^*$ as
\[
 c_1\geq\cdots\geq c_d.
\]
We use the dimension inequality
\begin{equation}\label{eq:subspace-dimension}
 \dim(U\cap V)\geq\dim U+\dim V-d.
\end{equation}
It follows from
\[
 \dim(U+V)=\dim U+\dim V-\dim(U\cap V)
\]
and $\dim(U+V)\leq d$.

Fix $1\leq i\leq d$. Let $U$ be the span of eigenvectors of $D$ corresponding to $b_1,\ldots,b_i$, and let $V$ be the span of eigenvectors of $C$ corresponding to $c_i,\ldots,c_d$. By \eqref{eq:subspace-dimension}, $U\cap V$ contains a unit vector $x$. Spectral expansion gives
\[
 x^*Dx\geq b_i,
 \qquad
 x^*Cx\leq c_i.
\]
Since
\[
 x^*(C-D)x=\abs{g^*x}^2\geq0,
\]
we obtain
\begin{equation}\label{eq:interlace-upper}
 c_i\geq x^*Cx\geq x^*Dx\geq b_i.
\end{equation}

Now fix $1\leq i<d$. Let $U$ be the span of the eigenvectors of $D$ corresponding to $b_i,\ldots,b_d$, let $V$ be the span of the eigenvectors of $C$ corresponding to $c_1,\ldots,c_{i+1}$, and let $H=g^\perp$. Applying \eqref{eq:subspace-dimension} twice gives
\[
 \begin{split}
 \dim(U\cap V\cap H)
 &\geq\dim U+\dim V+\dim H-2d\\
 &\geq(d-i+1)+(i+1)+(d-1)-2d=1.
 \end{split}
\]
Choose a unit vector $y\in U\cap V\cap H$. Since $g^*y=0$,
\[
 y^*Cy=y^*Dy.
\]
Spectral expansion gives
\[
 y^*Dy\leq b_i,
 \qquad
 y^*Cy\geq c_{i+1}.
\]
Thus
\begin{equation}\label{eq:interlace-lower}
 c_{i+1}\leq y^*Cy=y^*Dy\leq b_i.
\end{equation}
Equations \eqref{eq:interlace-upper} and \eqref{eq:interlace-lower} give
\[
 c_i\geq b_i\geq c_{i+1},
 \qquad 1\leq i<d,
\]
together with $c_d\geq b_d$.

\noindent\textbf{Step 5.} Fix $N$. Starting from $X^{(N)}$, first reveal the entries
\[
 \xi_{i,N+1},\qquad 1\leq i\leq M_N.
\]
This produces one column transition $R_{M_N}$. Next, for $s=1,\ldots,r_N$, reveal row $M_N+s$ through columns $1,\ldots,N+1$. Each row is independent of the preceding matrix and applies $R_{N+1}$. Iterating the conditional identities and using the tower property gives
\[
 R_{M_N}R_{N+1}^{\circ r_N}.
\]
If $r_N\leq K$, this composition contains at most $K+1$ kernels.
\end{proof}

A column addition updates $X_tX_t^*$ by an isotropic rank-one term in ambient dimension $m_t$, while a row addition updates $X_t^*X_t$ in ambient dimension $n_t$; the two Gram matrices have the same positive spectrum. The probability kernel $R_d$ records this rank-one spectral update after the necessary zero eigenvalues have been appended. The next proposition writes its normalized spectral densities as $s_p$ for a same-size step and $b_{p,r}$ for a birth, where $r$ is the dimension of the zero eigenspace before the birth. By contrast, the later kernels $B_t$ and $W_t$ are one-particle factors in a determinant chain: $B_t$ records real interlacing and virtual-state transfer, while $W_t$ also carries the next-level one-body weight. They are not normalized spectral transition densities.

\begin{prop}[Exact interlacing transition densities]\label{prop:gaussian-transition-densities}
Let $(z_{ij})_{i,j\geq1}$ be independent centered circular complex Gaussian variables with $\E\abs{z_{ij}}^2=1$. Let $((m_t,n_t))_{t=0}^H$ be an elementary up-right path of positive integer pairs. Put
\[
 \begin{aligned}
 X_t&=(z_{ij})_{\substack{1\leq i\leq m_t\\1\leq j\leq n_t}},
 &p_t&=\min(m_t,n_t),\\
 q_t&=\max(m_t,n_t),
 &r_t&=q_t-p_t.
 \end{aligned}
\]
Almost surely, $X_t$ has $p_t$ distinct positive squared singular values
\[
 \lambda_1^{(t)}>\cdots>\lambda_{p_t}^{(t)}>0.
\]
For $x_1>\cdots>x_k>0$, put
\[
 \Delta_k(x)=\prod_{1\leq i<j\leq k}(x_i-x_j).
\]

For $x_1>\cdots>x_p>0$ and $y_1>\cdots>y_p>0$, define
\[
 s_p(x,y)
 =
 \ee^{-\sum_{j=1}^py_j+\sum_{i=1}^px_i}
 \frac{\Delta_p(y)}{\Delta_p(x)}
 \ind_{\{y_1>x_1>y_2>x_2>\cdots>y_p>x_p\}}.
\]
For $r\geq1$, $x_1>\cdots>x_p>0$, and $y_1>\cdots>y_{p+1}>0$, define
\[
 \begin{split}
 b_{p,r}(x,y)
 &=
 \frac1{(r-1)!}
 \ee^{-\sum_{j=1}^{p+1}y_j+\sum_{i=1}^px_i}
 \frac{\Delta_{p+1}(y)}{\Delta_p(x)}\\
 &\quad\times
 \frac{\prod_{j=1}^{p+1}y_j^{r-1}}
      {\prod_{i=1}^px_i^r}
 \ind_{\{y_1>x_1>\cdots>y_p>x_p>y_{p+1}>0\}}.
 \end{split}
\]
Conditionally on $\lambda^{(t)}=x$, the density of $\lambda^{(t+1)}$ is $s_{p_t}(x,y)$ when $p_{t+1}=p_t$, and it is $b_{p_t,r_t}(x,y)$ when $p_{t+1}=p_t+1$. In the second case $r_t\geq1$. Both functions integrate to one over their decreasing chambers. Consequently, conditionally on $\lambda^{(0)}$, the multi-time density of
\[
 (\lambda^{(1)},\ldots,\lambda^{(H)})
\]
is the product of these transition densities.

If $(M_N)$ is nondecreasing and
\[
 \sup_N(M_{N+1}-M_N)\leq K,
\]
the path $(M_N,N)$ has the same exact product representation after refining each transition into one column addition followed by $M_{N+1}-M_N$ row additions. At most $K+1$ elementary kernels are used per sample-size transition.
\end{prop}

\begin{proof}
We first prove almost-sure full rank and simplicity. Fix $t$, abbreviate
\[
 m=m_t,\qquad n=n_t,\qquad p=p_t,\qquad X=X_t,
\]
and form the smaller Gram matrix
\[
 G=
 \begin{cases}
 XX^*,&m\leq n,\\
 X^*X,&n<m.
 \end{cases}
\]
Its eigenvalues are exactly the $p$ squared singular values of $X$.

Express every complex entry of $X$ through its real and imaginary parts. Every entry of $G$ is a polynomial in these real coordinates, so $\det G$ is a real polynomial. Moreover,
\[
 \rank X<p
 \quad\Longleftrightarrow\quad
 \det G=0.
\]
This polynomial is not identically zero: the rectangular matrix with diagonal entries $1,2,\ldots,p$ and all other entries zero has smaller Gram matrix
\[
 \diag(1^2,2^2,\ldots,p^2).
\]

Let $\chi_G(u)=\det(uI_p-G)$. Its coefficients are real polynomials in the coordinates of $X$. If $p\geq2$, the discriminant of $\chi_G$ is a polynomial in those coefficients and equals
\[
 \prod_{1\leq i<j\leq p}(\rho_i-\rho_j)^2,
\]
where $\rho_1,\ldots,\rho_p$ are the eigenvalues of $G$. For $p=1$, take the discriminant to be one. The discriminant is not identically zero because it is positive at the rectangular diagonal matrix above.

The zero set of a nonzero real polynomial in finitely many variables has Lebesgue measure zero. This follows by induction on the number of variables. In one variable, a nonzero polynomial has finitely many roots. For the induction step, write
\[
 P(x_1,\ldots,x_d)
 =
 \sum_{k=0}^sa_k(x_1,\ldots,x_{d-1})x_d^k
\]
and choose a coefficient which is not identically zero. By induction, the set of first $d-1$ coordinates at which every coefficient vanishes is null. At every other point, the resulting one-variable polynomial has finitely many roots. Fubini's theorem completes the induction.

The joint law of the real and imaginary parts of the entries of $X_t$ is absolutely continuous. Thus rank deficiency and repeated squared singular values each have probability zero. Taking a union over $t=0,\ldots,H$ proves the simultaneous assertion.

By Proposition~\ref{prop:rank-one-growth}, every elementary transition is an isotropic complex-Gaussian rank-one update. We compute its density on the strict-spectrum event.

First let
\[
 D=\diag(x_1,\ldots,x_p),
 \qquad
 x_1>\cdots>x_p>0,
\]
and let $g$ be a standard complex Gaussian vector in $\C^p$. Put $w_i=\abs{g_i}^2$. The $w_i$ are independent mean-one exponentials. If $y_1>\cdots>y_p$ are the eigenvalues of $D+gg^*$, the determinant identity in Proposition~\ref{prop:rank-one-growth} gives
\begin{equation}\label{eq:same-size-secular}
 1-\sum_{i=1}^p\frac{w_i}{z-x_i}
 =
 \frac{\prod_{j=1}^p(z-y_j)}
      {\prod_{i=1}^p(z-x_i)}.
\end{equation}
The left side is strictly increasing between consecutive poles. Its limits at the poles and at infinity give
\[
 y_1>x_1>y_2>x_2>\cdots>y_p>x_p.
\]
Taking residues at $z=x_i$ gives
\begin{equation}\label{eq:same-size-residue}
 w_i
 =
 -\frac{\prod_{j=1}^p(x_i-y_j)}
        {\prod_{k\neq i}(x_i-x_k)}.
\end{equation}
Comparison of coefficients at infinity gives
\begin{equation}\label{eq:same-size-trace}
 \sum_iw_i=\sum_jy_j-\sum_ix_i.
\end{equation}
Conversely, every $y$ in the displayed chamber gives positive $w_i$ through \eqref{eq:same-size-residue} and recovers the same rational function. Thus this is a smooth bijection with $(0,\infty)^p$.

Differentiating \eqref{eq:same-size-residue} gives
\[
 \frac{\partial w_i}{\partial y_j}
 =
 \frac{w_i}{y_j-x_i}.
\]
Therefore the absolute Jacobian is
\[
 \prod_iw_i
 \left|\det\left[\frac1{y_j-x_i}\right]_{i,j=1}^p\right|.
\]
The Cauchy determinant formula and \eqref{eq:same-size-residue} reduce this expression to
\[
 \frac{\Delta_p(y)}{\Delta_p(x)}.
\]
Indeed, the cross-differences cancel, while the two copies of $\Delta_p(x)$ arising from
\[
 \prod_i\prod_{k\neq i}\abs{x_i-x_k}
\]
leave one $\Delta_p(x)$ in the denominator. Multiplying by the density
\[
 \ee^{-\sum_iw_i}
\]
and using \eqref{eq:same-size-trace} gives $s_p(x,y)$. Since the change of variables is a bijection from a probability density, $s_p$ integrates to one.

For the birth transition, begin with the diagonal matrix having entries
\[
 x_1,\ldots,x_p,\underbrace{0,\ldots,0}_{r\text{ times}},
 \qquad r\geq1,
\]
and add $gg^*$ with $g\in\C^{p+r}$ standard complex Gaussian. Let $w_i$ be the squared projection onto the $x_i$ eigendirection, and let $w_0$ be the squared norm of the projection onto the $r$-dimensional zero eigenspace. Then $w_1,\ldots,w_p$ are independent mean-one exponentials, while $w_0$ is independent with density
\[
 \frac{\ee^{-w_0}w_0^{r-1}}{(r-1)!},
 \qquad w_0>0.
\]
The updated matrix has $r-1$ zero eigenvalues and $p+1$ positive eigenvalues $y_1>\cdots>y_{p+1}$. After canceling the common zero factors, the secular equation is
\begin{equation}\label{eq:birth-secular}
 1-\sum_{i=1}^p\frac{w_i}{z-x_i}-\frac{w_0}{z}
 =
 \frac{\prod_{j=1}^{p+1}(z-y_j)}
      {z\prod_{i=1}^p(z-x_i)}.
\end{equation}
It gives
\[
 y_1>x_1>\cdots>y_p>x_p>y_{p+1}>0
\]
and
\[
 \sum_iw_i+w_0=\sum_jy_j-\sum_ix_i.
\]
Taking the residue at zero yields
\[
 w_0=\frac{\prod_{j=1}^{p+1}y_j}{\prod_{i=1}^px_i}.
\]

Apply the preceding Jacobian computation to the $p+1$ old spectral points
\[
 (x_1,\ldots,x_p,0)
\]
and the weights $(w_1,\ldots,w_p,w_0)$. Since
\[
 \Delta_{p+1}(x_1,\ldots,x_p,0)
 =
 \Delta_p(x)\prod_{i=1}^px_i,
\]
the absolute Jacobian from $y$ to $(w_1,\ldots,w_p,w_0)$ is
\[
 \frac{\Delta_{p+1}(y)}
      {\Delta_p(x)\prod_ix_i}.
\]
Multiplying by
\[
 \ee^{-\sum_iw_i-w_0}\frac{w_0^{r-1}}{(r-1)!}
\]
and substituting the trace and residue identities gives $b_{p,r}(x,y)$. The bijection argument proves that this density integrates to one.

For a rectangular matrix, use the full-rank $p$-by-$p$ Gram matrix when the smaller dimension does not change. When the smaller dimension grows from $p$ to $p+1$, the larger dimension $q$ remains fixed, so $r=q-p\geq1$ and the $q$-by-$q$ Gram matrix has $p$ positive eigenvalues and $r$ zero eigenvalues. The second calculation applies. Proposition~\ref{prop:rank-one-growth} then yields the product formula. The final assertion follows from the refined exposure in that proposition.
\end{proof}

Thus $s_p$ is the normalized rank-one interlacing law on the strict alternating chamber when the positive-particle count is fixed. In the birth density, the additional particle below the previous spectrum and the factor $(r-1)!^{-1}$ arise from the gamma-distributed squared projection onto the $r$-dimensional zero eigenspace. These two densities are the elementary factors of the exact spectral path law.

\begin{prop}[Endpoint-shell path independence]\label{prop:shell-path-independence}
Fix positive integers $m_0\leq m_1$ and $n_0\leq n_1$. Given a deterministic matrix $A\in\C^{m_0\times n_0}$, form its random shell completion
\[
 \mathcal C(A)\in\C^{m_1\times n_1}
\]
by retaining $A$ as the upper-left block and filling every other entry with independent centered circular complex Gaussian variables of second absolute moment one. Let $\Lambda(B)$ denote the decreasing list of positive eigenvalues of $BB^*$. Then the law of $\Lambda(\mathcal C(A))$ depends on $A$ only through $\Lambda(A)$.

Consequently, the composition of the one-step kernels from Proposition~\ref{prop:rank-one-growth} along any elementary up-right path from $(m_0,n_0)$ to $(m_1,n_1)$ is independent of the path. If
\[
 a=m_1-m_0,\qquad b=n_1-n_0,
\]
then
\begin{equation}\label{eq:kernel-shell-commutation}
 R_{n_0}^{\circ a}R_{m_1}^{\circ b}
 =
 R_{m_0}^{\circ b}R_{n_1}^{\circ a}.
\end{equation}
In particular,
\[
 R_nR_{m+1}=R_mR_{n+1}
\]
on every admissible initial positive spectral list.
\end{prop}

\begin{proof}
Let $A,A'\in\C^{m_0\times n_0}$ satisfy $\Lambda(A)=\Lambda(A')$. The singular value decomposition gives unitary matrices $U\in\C^{m_0\times m_0}$ and $V\in\C^{n_0\times n_0}$ such that
\[
 A'=UAV^*.
\]
Write
\[
 \mathcal C(A)
 =
 \begin{bmatrix}
  A&G\\
  H&J
 \end{bmatrix},
\]
where the three added blocks are mutually independent standard complex Gaussian matrices. Let
\[
 U_1=\diag(U,I_{m_1-m_0}),
 \qquad
 V_1=\diag(V,I_{n_1-n_0}).
\]
Then
\[
 U_1\mathcal C(A)V_1^*
 =
 \begin{bmatrix}
  UAV^*&UG\\
  HV^*&J
 \end{bmatrix}
 =
 \begin{bmatrix}
  A'&UG\\
  HV^*&J
 \end{bmatrix}.
\]
Unitary invariance of complex Gaussian matrices implies
\[
 (UG,HV^*,J)\overset{\mathrm{law}}=(G,H,J).
\]
Left and right unitary multiplication preserve singular values. Hence
\[
 \Lambda(\mathcal C(A))
 \overset{\mathrm{law}}=
 \Lambda(\mathcal C(A')).
\]

We next verify the one-step transition for an arbitrary current matrix. Let $B$ be a deterministic $m$-by-$n$ complex matrix and let $\ell=\Lambda(B)$. If an independent standard complex Gaussian column $g\in\C^m$ is adjoined, then
\[
 [B,g][B,g]^*=BB^*+gg^*.
\]
Choose a unitary matrix $S$ such that
\[
 S^*BB^*S=\diag(\ell_1,\ldots,\ell_l,0,\ldots,0).
\]
Conjugation and unitary invariance of $S^*g$ show that the new spectrum has law $R_m(\ell)$. If instead an independent standard complex Gaussian row $h$ is adjoined, then
\[
 \begin{bmatrix}B\\h\end{bmatrix}^{\!*}
 \begin{bmatrix}B\\h\end{bmatrix}
 =
 B^*B+h^*h.
\]
Diagonalizing $B^*B$ proves that the new positive squared singular values have law $R_n(\ell)$. These statements remain valid conditionally when $B$ is random and the new row or column is independent of it.

Expose the shell entries in the order prescribed by any elementary up-right path. Every step applies the corresponding $R_d$, and the final matrix is always $\mathcal C(A)$. Since its spectral law depends on the initial matrix only through $\ell=\Lambda(A)$, all pathwise compositions agree.

If all $a$ rows are exposed first, the kernel is
\[
 R_{n_0}^{\circ a}R_{m_1}^{\circ b}.
\]
If all $b$ columns are exposed first, it is
\[
 R_{m_0}^{\circ b}R_{n_1}^{\circ a}.
\]
Their equality proves \eqref{eq:kernel-shell-commutation}. Taking $a=b=1$, $m_0=m$, and $n_0=n$ gives the last assertion.
\end{proof}

Proposition~\ref{prop:shell-path-independence} permits the entries of a completed Gaussian shell to be exposed in a canonical order when only the endpoint spectral kernel or endpoint law is being evaluated. It does not assert samplewise equality of different exposure paths, equality of their intermediate processes, or equality of every determinantal representative after a gauge change.

\subsection{Determinantal kernels along monotone paths}

The next theorem treats precisely the admissible finite elementary up-right paths in its statement. The terminal cardinality $P$ supplies a common number of determinant coordinates, and deterministic virtual states occupy the missing coordinates before their birth times. The exponents $\kappa_t$ and weights $w_t$ encode the one-body powers left after the exponential and Vandermonde factors telescope. The transfers $B_t$ implement interlacing and virtual-state propagation, while $W_t$ incorporates the level weight used by Eynard--Mehta. In the resulting kernel, the term involving $W_{s,t}$ is the free transfer and the inverse-moment-matrix sum is the finite-rank correction. General interlacing-chain determinantal geometry is prior \cite[Theorem~1]{OkounkovReshetikhin2003}; the theorem gives the exact common-array Gaussian specialization in the present normalization, without a priority claim beyond that specialization.

\begin{thm}[Extended determinantal kernel for arbitrary complex Gaussian growth]\label{thm:arbitrary-path-kernel}
Let $(z_{ij})_{i,j\geq1}$ be independent centered circular complex Gaussian variables with $\E\abs{z_{ij}}^2=1$. Let $((m_t,n_t))_{t=0}^T$ be an elementary up-right path with $T\geq1$ and $(m_0,n_0)=(1,1)$. Put
\[
 \begin{aligned}
 X_t&=(z_{ij})_{\substack{1\leq i\leq m_t\\1\leq j\leq n_t}},
 &p_t&=\min(m_t,n_t),\\
 q_t&=\max(m_t,n_t),
 &r_t&=q_t-p_t,
 &\epsilon_t&=p_{t+1}-p_t,
 \end{aligned}
\]
and $P=p_T$. Let
\[
 \lambda_1^{(t)}>\cdots>\lambda_{p_t}^{(t)}>0
\]
be the positive squared singular values of $X_t$. Define
\[
 \kappa_0=0,
 \qquad
 \kappa_t=\epsilon_{t-1}(r_{t-1}-1)-\epsilon_tr_t
 \quad(1\leq t<T),
\]
and
\[
 \kappa_T=\epsilon_{T-1}(r_{T-1}-1).
\]
Put
\[
 w_t(x)=x^{\kappa_t}\quad(t<T),
 \qquad
 w_T(x)=\ee^{-x}x^{\kappa_T}.
\]

Let $v_1,\ldots,v_P$ be distinct formal virtual states and define
\[
 E_t=(0,\infty)\sqcup\{v_{p_t+1},\ldots,v_P\},
\]
with Lebesgue measure on $(0,\infty)$ and counting measure on the virtual states. Define
\[
 \phi_1(x)=1\quad(x>0),\qquad
 \phi_i(v_i)=1\quad(2\leq i\leq P),
\]
with all other values zero, and define
\[
 \psi_i(y)=y^{P-i},\qquad y>0.
\]

For $0\leq t<T$, define $B_t\colon E_t\times E_{t+1}\to\R$ as follows. If $\epsilon_t=0$, let
\[
 B_t(x,y)=\ind_{\{y>x\}}
\]
for real $x,y$, let $B_t(v_k,v_l)=\ind_{\{k=l\}}$, and set all real-virtual cross terms to zero. If $\epsilon_t=1$, use the same real-real entry, set
\[
 B_t(v_{p_t+1},y)=1\quad(y>0),
 \qquad
 B_t(v_k,v_l)=\ind_{\{k=l\}}\quad(k,l\geq p_t+2),
\]
and set all remaining entries to zero. Define
\[
 W_t(z,y)=
 \begin{cases}
  B_t(z,y)w_{t+1}(y),&y>0,\\
  B_t(z,v_k),&y=v_k.
 \end{cases}
\]
Let $W_{s,t}$, $F_i^{(t)}$, $G_j^{(t)}$, and $M$ be as in Lemma~\ref{lem:fixed-particle-em}.

At level $t$, augment the positive squared singular values by the deterministic points
\[
 v_{p_t+1},\ldots,v_P.
\]
The resulting point process is determinantal. The matrix $M$ is finite and invertible, and an exact correlation kernel is
\[
 K(s,z;t,y)
 =
 -\ind_{\{s<t\}}W_{s,t}(z,y)
 +
 \sum_{i,j=1}^P
 G_j^{(s)}(z)(M^{-1})_{ji}F_i^{(t)}(y).
\]
Restricting this kernel to real arguments gives the complete multi-time process of positive squared singular values. The same conclusion holds for a finite path starting from any positive rectangle: prepend a path from $(1,1)$ and restrict the resulting determinantal process to the prescribed levels.
\end{thm}

\begin{proof}
By Proposition~\ref{prop:gaussian-transition-densities}, the squared singular values are almost surely positive and distinct and have the exact transition densities stated there. The initial value
\[
 \lambda_1^{(0)}=\abs{z_{11}}^2
\]
has density $\ee^{-x}$ on $x>0$.

At level $t$, place the augmented coordinates in canonical order:
\[
 \lambda_1^{(t)},\ldots,\lambda_{p_t}^{(t)},
 v_{p_t+1},\ldots,v_P.
\]
If $\epsilon_t=0$, the virtual block of $[B_t]$ is the identity and the real block has entries
\[
 \ind_{\{\lambda_j^{(t+1)}>\lambda_i^{(t)}\}}.
\]
On the interlacing chamber this block is lower triangular with diagonal entries one. Outside its weak closure, its determinant vanishes. If $\epsilon_t=1$, the active real block has the same indicator rows and a final all-one row arising from $v_{p_t+1}$. On the birth interlacing chamber, it is again lower triangular with diagonal one and vanishes outside the weak closure. Therefore, in both cases,
\begin{equation}\label{eq:transfer-determinant}
 \det[W_t]
 =
 \ind_{\{\text{interlacing at }t\}}
 \prod_{j=1}^{p_{t+1}}
 w_{t+1}(\lambda_j^{(t+1)}).
\end{equation}
Boundary sets have Lebesgue measure zero.

For the canonical augmented coordinates at time zero,
\[
 \det[\phi_i(x_j^{(0)})]=1.
\]
At time $T$ there are no virtual states, and
\[
 \det[\psi_i(\lambda_j^{(T)})]_{i,j=1}^P
 =
 \Delta_P(\lambda^{(T)}).
\]
Thus the determinant product in Lemma~\ref{lem:fixed-particle-em} equals
\begin{equation}\label{eq:path-determinant-density}
 D
 =
 \Delta_P(\lambda^{(T)})
 \prod_{u=1}^T\prod_{j=1}^{p_u}
 w_u(\lambda_j^{(u)})
 \prod_{t=0}^{T-1}
 \ind_{\{\text{interlacing at }t\}}
\end{equation}
on the product of decreasing chambers.

Multiply the initial density by the transition densities in Proposition~\ref{prop:gaussian-transition-densities}. The exponential factors telescope to
\[
 \exp\left(-\sum_j\lambda_j^{(T)}\right),
\]
and the Vandermonde ratios telescope to $\Delta_P(\lambda^{(T)})$. At an intermediate level $1\leq u<T$, the exponent contributed to each variable is
\[
 \epsilon_{u-1}(r_{u-1}-1)-\epsilon_ur_u=\kappa_u.
\]
At the terminal level, it is
\[
 \epsilon_{T-1}(r_{T-1}-1)=\kappa_T.
\]
There is no residual exponent at level zero because $r_0=0$, while a birth step requires $r_0\geq1$, so $\epsilon_0=0$. Hence, with
\[
 C_{\mathrm{path}}
 =
 \prod_{\{t:\epsilon_t=1\}}(r_t-1)!,
\]
the quantity in \eqref{eq:path-determinant-density} is $C_{\mathrm{path}}$ times the normalized joint density.

Permuting the $P$ coordinate labels at any fixed level changes the signs of the two incident determinants, or of the boundary determinant and its adjacent determinant at an endpoint. Their product is unchanged. A nonzero augmented determinant product contains each prescribed virtual state exactly once and $p_t$ real coordinates at level $t$. For a fixed decreasing real vector, there are $P!$ assignments to the labeled slots at each level. Therefore
\[
 Z=(P!)^{T+1}C_{\mathrm{path}},
\]
which is positive and finite.

It remains to verify absolute convergence. For $1\leq i\leq P$, let $\tau_i$ be the first time for which $p_{\tau_i}=i$, with $\tau_1=0$, and put
\[
 \rho_i=r_{\tau_i}.
\]
For $t<T$, define
\[
 d_t=
 \begin{cases}
 q_t,&\epsilon_t=0,\\
 p_t,&\epsilon_t=1.
 \end{cases}
\]
A direct forward induction gives, for real $x$ and $i\leq p_t$,
\begin{equation}\label{eq:preliminary-forward-function}
 F_i^{(t)}(x)
 =
 \frac{\rho_i!}{(p_t-i)!(q_t-i)!}x^{d_t-i},
\end{equation}
while at the terminal level
\begin{equation}\label{eq:preliminary-terminal-forward}
 F_i^{(T)}(x)
 =
 \frac{\rho_i!}{(P-i)!(q_T-i)!}\ee^{-x}x^{q_T-i}.
\end{equation}
At the birth time, the formula follows from the all-one virtual row and the definition of $\kappa$. At each later step, convolution integrates $x^a$ over $0<x<y$, producing $y^{a+1}/(a+1)$, and multiplication by $y^{\kappa_{t+1}}$ gives exponent $d_{t+1}-i$. The denominator $a+1$ is precisely the next factorial factor.

Since $G_j^{(T)}=\psi_j$,
\[
 M_{ij}
 =
 \int_0^\infty
 F_i^{(T)}(x)x^{P-j}\,\dd x
\]
is a finite gamma integral. All kernels and boundary functions are nonnegative on their real supports. Tonelli's theorem identifies every iterated one-particle integral with this finite terminal integral. Backward convolution from the terminal exponential is finite for each real $x>0$, and a virtual-source integral is a sub-integral of the same finite path integral. Thus all convolutions and moments converge absolutely.

Lemma~\ref{lem:fixed-particle-em} now applies and gives the asserted kernel. Restricting a determinantal process to real arguments removes the deterministic virtual particles. Finally, prepending a path from $(1,1)$ and restricting the resulting process proves the arbitrary-starting-rectangle assertion.
\end{proof}

The transfer determinant now has a direct interpretation: at each level it contributes one strict interlacing indicator and the next-level one-body weight. Multiplication over the path makes the exponential and Vandermonde factors telescope to the determinant-chain density. The negative free-transfer term records unconstrained one-particle propagation between ordered levels, while the finite-rank term enforces the boundary data; restricting to real arguments removes only the deterministic virtual states.

The next three lemmas form one algebraic mechanism. The boundary calculation isolates the detailed birth schedule in row factors of the moment matrix, the factorial Hankel inversion identifies the terminal Laguerre block, and the forward calculation shows that those birth-time row factors cancel from every real finite-rank summand. The combined purpose is to replace path-dependent determinant bookkeeping by endpoint-adapted Laguerre formulas without asserting that every component of the multi-time kernel is path-independent.

\begin{lem}[Closed boundary moment matrix]\label{lem:boundary-moment}
Use the notation of Theorem~\ref{thm:arbitrary-path-kernel}, and put $Q=q_T$. For $1\leq i\leq P$, let $\tau_i$ be the first time for which $p_{\tau_i}=i$, with $\tau_1=0$, and define
\[
 a_i=q_{\tau_i},
 \qquad
 \rho_i=a_i-i=r_{\tau_i}.
\]
Then
\begin{equation}\label{eq:closed-moment-matrix}
 M_{ij}
 =
 \frac{\rho_i!}{(P-i)!(Q-i)!}(P+Q-i-j)!,
 \qquad 1\leq i,j\leq P.
\end{equation}
Thus every dependence of $M$ on the chosen interleaving is confined to the row factors $\rho_i!$.
\end{lem}

\begin{proof}
The proof is recorded in Appendix~\ref{app:proof-boundary-moment}.
\end{proof}

\begin{lem}[Laguerre inversion of the factorial Hankel matrix]\label{lem:laguerre-hankel}
Let $P\leq Q$ be positive integers, put $r=Q-P$, and define
\[
 H_{ij}=(P+Q-i-j)!,
 \qquad 1\leq i,j\leq P.
\]
Define
\[
 L_k^{(r)}(x)
 =
 \sum_{\ell=0}^k
 (-1)^\ell
 \frac{(k+r)!}{(k-\ell)!(r+\ell)!\ell!}x^\ell.
\]
Then $H$ is invertible and
\begin{equation}\label{eq:hankel-laguerre-inversion}
 \sum_{i,j=1}^P
 x^{P-j}(H^{-1})_{ji}y^{Q-i}
 =
 y^r
 \sum_{k=0}^{P-1}
 \frac{k!}{(k+r)!}
 L_k^{(r)}(x)L_k^{(r)}(y).
\end{equation}
Moreover,
\begin{equation}\label{eq:christoffel-darboux}
 \begin{aligned}
 \sum_{k=0}^{P-1}\frac{k!}{(k+r)!}
 L_k^{(r)}(x)L_k^{(r)}(y)
 &=\frac{P!}{(P+r-1)!}\\
 &\quad{}\times
 \frac{L_{P-1}^{(r)}(x)L_P^{(r)}(y)
 -L_P^{(r)}(x)L_{P-1}^{(r)}(y)}{x-y},
 \end{aligned}
\end{equation}
with the diagonal value defined by continuity. After multiplication by $\ee^{-y}$, the left side of \eqref{eq:hankel-laguerre-inversion} is the terminal equal-time finite-rank block in Theorem~\ref{thm:arbitrary-path-kernel}.
\end{lem}

\begin{proof}
The proof is recorded in Appendix~\ref{app:proof-laguerre-hankel}.
\end{proof}

\begin{lem}[Forward functions and cancellation of birth-time gauges]\label{lem:forward-gauge-cancellation}
Use the notation of Theorem~\ref{thm:arbitrary-path-kernel} and Lemma~\ref{lem:boundary-moment}. Put
\[
 d_T=Q,\qquad h_T(x)=\ee^{-x}.
\]
For $t<T$, put $h_t(x)=1$ and
\[
 d_t=
 \begin{cases}
 q_t,&\epsilon_t=0,\\
 p_t,&\epsilon_t=1.
 \end{cases}
\]
For real $x$ and $1\leq i\leq p_t$,
\begin{equation}\label{eq:explicit-forward}
 F_i^{(t)}(x)
 =
 \frac{\rho_i!}{(p_t-i)!(q_t-i)!}
 h_t(x)x^{d_t-i}.
\end{equation}
For $i>p_t$, $F_i^{(t)}$ vanishes on the real part of $E_t$ and
\[
 F_i^{(t)}(v_k)=\ind_{\{i=k\}}.
\]
If
\[
 H_{ij}=(P+Q-i-j)!,
\]
then, for real $y$, $1\leq i\leq p_t$, and $1\leq j\leq P$,
\begin{equation}\label{eq:gauge-cancellation}
 (M^{-1})_{ji}F_i^{(t)}(y)
 =
 (H^{-1})_{ji}
 \frac{(P-i)!(Q-i)!}{(p_t-i)!(q_t-i)!}
 h_t(y)y^{d_t-i}.
\end{equation}
Thus the birth-time factors $\rho_i!$ cancel from every real finite-rank summand.
\end{lem}

\begin{proof}
The proof is recorded in Appendix~\ref{app:proof-forward-gauge}.
\end{proof}

Taken together, Lemmas~\ref{lem:boundary-moment}--\ref{lem:forward-gauge-cancellation} confine the interleaving dependence relevant to the real finite-rank summands to row factors and then cancel those factors after inversion. The surviving blocks are expressed through terminal dimensions, endpoint dimensions, and the selected outgoing-step gauge.

The next theorem restarts the determinant chain from the actual positive spectrum of one fixed starting block. It is used as a forward conditional law for the strictly future levels after conditioning on all entries of that block. It does not condition on a terminal boundary and is therefore neither a two-boundary bridge nor an interior resampling theorem. The real starting eigenvalues replace the corresponding source states, while particles not yet born remain represented by deterministic virtual states.

\begin{thm}[Conditional determinantal resampling kernel]\label{thm:conditional-growth-kernel}
Let $(z_{ij})$ be the complex Gaussian array above, and let $((m_t,n_t))_{t=0}^T$ be an elementary up-right path with $T\geq1$. Put
\[
 \begin{aligned}
 p_t&=\min(m_t,n_t),
 &q_t&=\max(m_t,n_t),
 &r_t&=q_t-p_t,\\
 \epsilon_t&=p_{t+1}-p_t,
 &P&=p_T.
 \end{aligned}
\]
Condition on all entries of
\[
 X_0=(z_{ij})_{\substack{1\leq i\leq m_0\\1\leq j\leq n_0}},
\]
and let
\[
 x_1>\cdots>x_{p_0}>0
\]
be its positive squared singular values. Define
\[
 \kappa_t=\epsilon_{t-1}(r_{t-1}-1)-\epsilon_tr_t
 \quad(1\leq t<T),
 \qquad
 \kappa_T=\epsilon_{T-1}(r_{T-1}-1),
\]
and
\[
 w_t(u)=u^{\kappa_t}\quad(t<T),
 \qquad
 w_T(u)=\ee^{-u}u^{\kappa_T}.
\]

Let
\[
 E_t=(0,\infty)\sqcup\{v_{p_t+1},\ldots,v_P\},
 \qquad 1\leq t\leq T.
\]
Use the augmented transfers $B_t$ from Theorem~\ref{thm:arbitrary-path-kernel}. Define source states
\[
 \zeta_i=
 \begin{cases}
 x_i,&i\leq p_0,\\
 v_i,&i>p_0,
 \end{cases}
\]
and define on $E_1$
\[
 \phi_i(y)=
 \begin{cases}
 B_0(\zeta_i,y)w_1(y),&y>0,\\
 B_0(\zeta_i,v_k),&y=v_k.
 \end{cases}
\]
For $1\leq t<T$, define $W_t$ as in Theorem~\ref{thm:arbitrary-path-kernel}, and put
\[
 \psi_i(y)=y^{P-i}.
\]
Let $F_i^{(t)}$, $G_j^{(t)}$, and $M$ be the corresponding forward functions, backward functions, and moment matrix.

Conditionally on $X_0$, augment each strictly future configuration by its deterministic virtual points. This conditional process is determinantal, $M$ is finite and invertible, and its kernel is
\[
 \begin{aligned}
 K_x(s,z;t,y)
 &=-\ind_{\{s<t\}}W_{s,t}(z,y)\\
 &\quad{}+\sum_{i,j=1}^P
 G_j^{(s)}(z)(M^{-1})_{ji}F_i^{(t)}(y),
 \qquad 1\leq s,t\leq T.
 \end{aligned}
\]
Restriction to real arguments gives the conditional multi-time process of positive squared singular values. The kernel depends on $X_0$ only through $x$.
\end{thm}

\begin{proof}
Work on the probability-one event from Proposition~\ref{prop:gaussian-transition-densities} on which every spectrum along the path is positive and simple. Write
\[
 \lambda^{(t)}
 =
 (\lambda_1^{(t)},\ldots,\lambda_{p_t}^{(t)}),
 \qquad
 \lambda^{(0)}=x.
\]
Proposition~\ref{prop:rank-one-growth} shows that, conditionally on $X_t$, the next spectral law depends on $X_t$ only through $\lambda^{(t)}$. Proposition~\ref{prop:gaussian-transition-densities} identifies its normalized density. Iterating conditional expectations yields
\begin{equation}\label{eq:conditional-path-density}
 Q_x(\lambda^{(1)},\ldots,\lambda^{(T)})
 =
 \prod_{t=0}^{T-1}
 h_t(\lambda^{(t)},\lambda^{(t+1)}),
\end{equation}
where $h_t$ is the corresponding same-size or birth density.

For each $t$, define the canonical augmentation
\[
 a^{(t)}
 =
 (\lambda_1^{(t)},\ldots,\lambda_{p_t}^{(t)},
 v_{p_t+1},\ldots,v_P),
\]
with $a^{(0)}=(\zeta_1,\ldots,\zeta_P)$. A direct determinant calculation gives
\begin{equation}\label{eq:conditional-B-determinant}
 \det[B_t(a_i^{(t)},a_j^{(t+1)})]_{i,j=1}^P
 =
 I_t,
\end{equation}
where $I_t$ is the strict interlacing indicator. For a same-size step, the virtual block is the identity and the real block is
\[
 [\ind_{\{\lambda_j^{(t+1)}>\lambda_i^{(t)}\}}].
\]
If $n_i$ is the number of leading ones in row $i$, its determinant is nonzero exactly when $n_i=i$ for every $i$, which is strict interlacing. The matrix is then lower triangular with diagonal one. At a birth step, the active block has the same real rows and a final all-one row from $v_{p_t+1}$. The same leading-one argument gives the birth interlacing indicator with positive sign.

Factoring the one-body weights from the columns and using the terminal Vandermonde gives
\begin{equation}\label{eq:conditional-determinant-product}
 \begin{split}
 &\det[\phi_i(a_j^{(1)})]
 \prod_{t=1}^{T-1}
 \det[W_t(a_i^{(t)},a_j^{(t+1)})]
 \det[\psi_i(\lambda_j^{(T)})]\\
 &\qquad=
 \Delta_P(\lambda^{(T)})
 \prod_{t=0}^{T-1}I_t
 \prod_{s=1}^T\prod_{j=1}^{p_s}
 w_s(\lambda_j^{(s)}).
 \end{split}
\end{equation}

Multiplication of the transition densities in Proposition~\ref{prop:gaussian-transition-densities} yields
\[
 \begin{split}
 Q_x
 &=
 \prod_{t=0}^{T-1}I_t
 \exp\left(\sum_ix_i-\sum_j\lambda_j^{(T)}\right)
 \frac{\Delta_P(\lambda^{(T)})}{\Delta_{p_0}(x)}\\
 &\quad\times
 \prod_{\{t:\epsilon_t=1\}}
 \frac{\prod_{j=1}^{p_{t+1}}
       (\lambda_j^{(t+1)})^{r_t-1}}
 {(r_t-1)!\prod_{i=1}^{p_t}
       (\lambda_i^{(t)})^{r_t}}.
 \end{split}
\]
At an intermediate level $1\leq s<T$, the exponent of each $\lambda_j^{(s)}$ is $\kappa_s$, and its terminal exponent is $\kappa_T$. Thus \eqref{eq:conditional-path-density} equals $C_x$ times \eqref{eq:conditional-determinant-product}, where
\[
 C_x
 =
 \frac{\exp(\sum_ix_i)}{\Delta_{p_0}(x)}
 \prod_{\{t:\epsilon_t=1\}}\frac1{(r_t-1)!}
 \prod_{i=1}^{p_0}x_i^{-\epsilon_0r_0}>0.
\]
Hence the determinant product is a positive constant multiple of the exact conditional density.

On the full spaces $E_t^P$, reordering the coordinates at one level changes the signs of the two adjacent determinants, so their product is invariant. A nonzero determinant forces the configuration at level $t$ to contain exactly $p_t$ real points and each prescribed virtual point once. Duplicate states annihilate a determinant, every retained virtual target requires its matching virtual source, and a birth has exactly one additional virtual source capable of supporting the additional real target. Therefore the determinant product vanishes away from permutations of the canonical augmentations and is nonnegative elsewhere. Its total integral is
\[
 \frac{(P!)^T}{C_x},
\]
which is positive and finite.

We verify the remaining absolute-convergence requirement. All entries of $\phi_i$, $W_t$, and $\psi_j$ are nonnegative. A path issuing from a real source $x_i$ remains above $x_i>0$, and its polynomial factors are integrable because the terminal weight contains an exponential. A path issuing from $v_i$ remains virtual until its unique birth level $a$, and thereafter has
\[
 0<u_a<\cdots<u_T.
\]
Its monomial factor is
\[
 \prod_{t=a}^Tu_t^{\kappa_t}u_T^{P-j}.
\]
For $a\leq s<T$, telescoping gives
\[
 \sum_{t=a}^s\kappa_t
 =
 r_{a-1}-1-b_{a,s}-\epsilon_sr_s,
\]
where $b_{a,s}$ counts birth steps among $a,\ldots,s-1$. If $c_{a,s}$ counts same-size steps there, then
\[
 r_s=r_{a-1}-1+c_{a,s}-b_{a,s}.
\]
The prefix exponent sum is therefore $-c_{a,s}$ when $\epsilon_s=1$ and at least $-c_{a,s}$ when $\epsilon_s=0$. The analogous terminal sum is at least $-c_{a,T}$. Since
\[
 s-a+1=b_{a,s}+c_{a,s}+1,
\]
every initial partial exponent sum is strictly greater than $-(s-a+1)$. Successive integration on
\[
 0<u_a<\cdots<u_T<1
\]
is finite, and the terminal exponential controls the complementary region.

Apply Lemma~\ref{lem:fixed-particle-em} to the future spaces $E_1,\ldots,E_T$. It yields the displayed conditional kernel. The only source-dependent quantities in the construction are the spectral source states $x_i$, so the kernel depends on $X_0$ only through $x$.
\end{proof}

Consequently, the strictly future conditional process is determinantal almost surely. Although the conditioning sigma-algebra contains every entry of $X_0$, unitary invariance reduces the displayed kernel's dependence on the starting block to its positive spectrum. The conditional determinant construction is one-sided in time, and any later application of operator identities must verify their hypotheses on almost every starting-spectrum fiber.

\subsection{Two-level reductions and endpoint kernels}

The later upper and lower probability arguments require different conversions from kernel estimates. For an upper event, the relevant variable is the occupancy count of an upper set, and the cross-block trace integral controls the mixed second moment. For a lower event, the relevant event is a gap, and the block Fredholm determinant is separated by a Schur complement involving the diagonal gap resolvents. The following lemma records these two reductions, including the trace-class, invertibility, and absolute-convergence assumptions needed in the gap and conditional settings. It does not by itself yield the main $(1+o(1))$ product-scale upper estimate; that conclusion also requires the later cross-block asymptotics and one-time probability comparisons.

\begin{lem}[Finite-intensity two-level reductions]\label{lem:two-level-reductions}
Let a determinantal point process be defined on the disjoint union of two sigma-finite measure spaces $(E_s,\nu_s)$ and $(E_t,\nu_t)$, with kernel blocks $K_{uv}$. Let $A_s\subseteq E_s$ and $A_t\subseteq E_t$ be measurable, put
\[
 N_u=\#(\text{points in }A_u),
 \qquad
 \mu_u=\int_{A_u}K_{uu}(x,x)\,\nu_u(\dd x),
\]
and assume that $\mu_s,\mu_t<\infty$. Suppose the equal-level restricted kernels have Hermitian representatives and that
\[
 C_{st}
 =
 \int_{A_s\times A_t}
 K_{st}(x,y)K_{ts}(y,x)
 \,\nu_s(\dd x)\nu_t(\dd y)
\]
converges absolutely. Then
\[
 \E[N_sN_t]=\mu_s\mu_t-C_{st}
\]
and
\[
 \mu_u-\frac{\mu_u^2}{2}
 \leq
 \Prob(N_u>0)
 \leq
 \mu_u.
\]
Consequently, if $0\leq\mu_s,\mu_t\leq1$, $c\geq0$, and
\[
 \abs{C_{st}}\leq c\mu_s\mu_t,
\]
then
\[
 \Prob(N_s>0,N_t>0)
 \leq
 4(1+c)\Prob(N_s>0)\Prob(N_t>0).
\]

For measurable $B_s\subseteq E_s$ and $B_t\subseteq E_t$, put
\[
 H_u=L^2(B_u,\nu_u)
\]
and suppose the restricted block kernel
\[
 \mathcal K_B=
 \begin{bmatrix}
 K_{ss}&K_{st}\\
 K_{ts}&K_{tt}
 \end{bmatrix}
\]
is trace class on $H_s\oplus H_t$. Put
\[
 A=I-K_{ss},
 \qquad
 D=I-K_{tt},
\]
and suppose that $A$ and $D$ are boundedly invertible. Let $G_u$ be the event that $B_u$ contains no process point, and define
\[
 R=D^{-1}K_{ts}A^{-1}K_{st}.
\]
Then $R$ is trace class and
\[
 \Prob(G_s\cap G_t)
 =
 \Prob(G_s)\Prob(G_t)\det(I-R),
\]
so
\[
 \Prob(G_s\cap G_t)
 \leq
 \ee^{\norm{R}_1}
 \Prob(G_s)\Prob(G_t).
\]
The same conclusions hold almost surely for conditional determinantal laws. In particular, they apply to the conditional kernel in Theorem~\ref{thm:conditional-growth-kernel} for almost every starting spectrum whenever the stated fiberwise hypotheses hold.
\end{lem}

\begin{proof}
The proof is recorded in Appendix~\ref{app:proof-two-level-reductions}.
\end{proof}

The first clause of Lemma~\ref{lem:two-level-reductions} is the occupancy-count input for upper-tail events; its conclusion is a constant-factor event bound under the displayed cross-trace hypothesis. The second clause is the gap-determinant input for lower-tail events; its interaction operator is the Schur-complement correction between the two gaps. Neither clause identifies the process as independent, and the lemma alone does not supply the later $(1+o(1))$ product-scale upper bound.

For the explicit blocks, $K_{uv}(x,y)$ means that the first variable $x$ lies at level $u$ and the second variable $y$ lies at level $v$. Thus $K_{10}$ runs from the later level to the earlier level and has no negative free-transfer term, whereas $K_{01}$ runs from the earlier level to the later level and contains $-W$. The equal-level formulas are written in an outgoing-step gauge: changing the outgoing step changes the displayed one-sided weight by a levelwise determinantal conjugation, not the marginal point process. The later projection gauge makes the equal-level representatives Hermitian Laguerre projections. In the following algebra, $Q$ continues to denote the terminal larger dimension $q_T$.

\begin{prop}[Terminal-to-earlier differential blocks]\label{prop:terminal-earlier-differential}
Use the notation of Theorem~\ref{thm:arbitrary-path-kernel} and put
\[
 r=Q-P,
 \qquad
 R_P(x,y)
 =
 \sum_{k=0}^{P-1}
 \frac{k!}{(k+r)!}
 L_k^{(r)}(x)L_k^{(r)}(y).
\]
For $h\geq0$ and an operator $C$, define
\[
 C^{\mathrm{fall}\,0}=I,
 \qquad
 C^{\mathrm{fall}\,h}
 =
 C(C-I)\cdots(C-(h-1)I)
 \quad(h\geq1).
\]
Let $E_y=y\,\frac{\dd}{\dd y}$. Fix $t<T$ and put
\[
 A=P-p_t,
 \qquad
 B=Q-q_t.
\]
If $\epsilon_t=1$, then
\begin{equation}\label{eq:differential-birth}
 K(T,x;t,y)
 =
 \left(\frac{\dd}{\dd y}\right)^A
 (rI+E_y)^{\mathrm{fall}\,B}R_P(x,y).
\end{equation}
If $\epsilon_t=0$, then
\begin{equation}\label{eq:differential-same}
 K(T,x;t,y)
 =
 \left(\frac{\dd}{\dd y}\right)^B
 (E_y-rI)^{\mathrm{fall}\,A}
 [y^rR_P(x,y)].
\end{equation}
At the terminal level,
\begin{equation}\label{eq:terminal-laguerre-block}
 K(T,x;T,y)=\ee^{-y}y^rR_P(x,y).
\end{equation}
Apart from the outgoing-step gauge, these blocks depend only on $(p_t,q_t)$ and $(P,Q)$.
\end{prop}

\begin{proof}
Since $T\geq t$, the negative free-transfer term in Theorem~\ref{thm:arbitrary-path-kernel} is absent. At time $T$,
\[
 G_j^{(T)}(x)=x^{P-j}.
\]
Lemma~\ref{lem:forward-gauge-cancellation} gives, for $t<T$,
\begin{equation}\label{eq:differential-inverse-sum}
 K(T,x;t,y)
 =
 \sum_{j=1}^P\sum_{i=1}^{p_t}
 x^{P-j}(H^{-1})_{ji}
 \frac{(P-i)!(Q-i)!}{(p_t-i)!(q_t-i)!}
 y^{d_t-i}.
\end{equation}

Set $a=P-i$. Then $i\leq p_t$ is equivalent to $a\geq A$, and
\[
 \frac{(P-i)!}{(p_t-i)!}
 =
 a^{\mathrm{fall}\,A},
 \qquad
 \frac{(Q-i)!}{(q_t-i)!}
 =
 (r+a)^{\mathrm{fall}\,B}.
\]
An integer falling factorial vanishes when its order exceeds its argument, so the sum may be extended over $a=0,\ldots,P-1$.

If $\epsilon_t=1$, then $d_t=p_t=P-A$ and
\[
 y^{d_t-i}=y^{a-A}.
\]
On $y^a$, the operator $(rI+E_y)^{\mathrm{fall}\,B}$ multiplies by $(r+a)^{\mathrm{fall}\,B}$, while $\left(\frac{\dd}{\dd y}\right)^A$ multiplies by $a^{\mathrm{fall}\,A}$ and changes $y^a$ to $y^{a-A}$. Thus each coefficient and monomial in \eqref{eq:differential-inverse-sum} equals
\[
 \left(\frac{\dd}{\dd y}\right)^A
 (rI+E_y)^{\mathrm{fall}\,B}y^a.
\]
Lemma~\ref{lem:laguerre-hankel} then proves \eqref{eq:differential-birth}.

If $\epsilon_t=0$, then $d_t=q_t=Q-B$ and
\[
 y^{d_t-i}=y^{r+a-B}.
\]
Starting from $y^{r+a}$, the operator $(E_y-rI)^{\mathrm{fall}\,A}$ multiplies by $a^{\mathrm{fall}\,A}$, and the $B$th derivative multiplies by $(r+a)^{\mathrm{fall}\,B}$ and changes the exponent to $r+a-B$. Lemma~\ref{lem:laguerre-hankel} gives \eqref{eq:differential-same}.

At $t=T$, Lemma~\ref{lem:forward-gauge-cancellation} has $A=B=0$, $d_T=Q$, and $h_T(y)=\ee^{-y}$. Hence
\[
 K(T,x;T,y)
 =
 \ee^{-y}
 \sum_{i,j=1}^P
 x^{P-j}(H^{-1})_{ji}y^{Q-i},
\]
and \eqref{eq:terminal-laguerre-block} follows from Lemma~\ref{lem:laguerre-hankel}.
\end{proof}

Proposition~\ref{prop:terminal-earlier-differential} converts the factorial mode weights in the inverse moment sum into falling Euler operators and ordinary derivatives. The outgoing-step distinction changes the Laguerre gauge and exchanges the roles of $A$ and $B$ between the derivative order and the falling-Euler order; in both formulas the falling-Euler operator acts before the outer derivative, while the later-to-earlier direction remains free of the negative transfer term.

For fixed orientation, $p$ and $q$ are the earlier smaller and larger dimensions, and $P$ and $Q$ are their later counterparts on the same physical sides. Hence $A=P-p$ counts increments of the smaller side, which are birth steps in the canonical order, while $B=Q-q$ counts increments of the larger side, which are same-size steps; $r_0=q-p$ is the initial rectangularity. Shell path independence permits the canonical exposure of the $B$ larger-side increments followed by the $A$ smaller-side increments. The cases $A=0$, $B=0$, and $A,B\geq1$ therefore represent respectively a pure same-size segment, a pure birth segment, and a path with one turn. Here $W$ is the Eynard--Mehta free transfer rather than a normalized spectral transition density. Fixed-orientation Wishart/Laguerre minor structure is prior in substance \cite{AdlerVanMoerbekeWang2013}; the theorem records the formula in the normalization and gauge used in this paper.

\begin{thm}[Two-level kernel for fixed-orientation nested rectangles]\label{thm:fixed-orientation-kernel}
Let $(m_0,n_0)$ and $(m_1,n_1)$ satisfy
\[
 m_0\leq m_1,\qquad n_0\leq n_1,
\]
and assume the same physical side is weakly larger at both endpoints. If
\[
 m_0\geq n_0,\qquad m_1\geq n_1,
\]
put
\[
 (p,q,P,Q)=(n_0,m_0,n_1,m_1).
\]
If
\[
 n_0\geq m_0,\qquad n_1\geq m_1,
\]
put
\[
 (p,q,P,Q)=(m_0,n_0,m_1,n_1).
\]
Set
\[
 A=P-p,\quad B=Q-q,\quad r_0=q-p,\quad r=Q-P,
\]
and assume $A+B\geq1$. Put
\[
 \mathcal L_{d,\alpha}(x,y)
 =
 \sum_{k=0}^{d-1}
 \frac{k!}{(k+\alpha)!}
 L_k^{(\alpha)}(x)L_k^{(\alpha)}(y).
\]

Define $W(x,z)=0$ unless $z>x$. If $A=0$, set
\begin{equation}\label{eq:fixed-W-A-zero}
 W(x,z)=\ee^{-z}\frac{(z-x)^{B-1}}{(B-1)!}.
\end{equation}
If $B=0$, set
\begin{equation}\label{eq:fixed-W-B-zero}
 W(x,z)=
 \ee^{-z}z^r\frac{(z-x)^{A-1}}{(A-1)!}.
\end{equation}
If $A,B\geq1$, set
\begin{equation}\label{eq:fixed-W-positive}
 W(x,z)
 =
 \ee^{-z}z^r
 \int_x^z
 \frac{(u-x)^{B-1}u^{-(r+A)}(z-u)^{A-1}}
 {(B-1)!(A-1)!}\,\dd u.
\end{equation}

The two-level positive squared-singular-value process is determinantal. With the earlier level indexed by zero and the later level by one, one exact kernel has blocks
\begin{equation}\label{eq:fixed-K11}
 K_{11}(x,y)
 =
 \ee^{-y}y^r\mathcal L_{P,r}(x,y),
\end{equation}
\begin{equation}\label{eq:fixed-K00-same}
 K_{00}(x,y)
 =
 \ee^{-x}y^{r_0}\mathcal L_{p,r_0}(x,y)
 \qquad(B\geq1),
\end{equation}
\begin{equation}\label{eq:fixed-K00-birth}
 K_{00}(x,y)
 =
 \ee^{-x}x^{r_0}\mathcal L_{p,r_0}(x,y)
 \qquad(B=0),
\end{equation}
\begin{equation}\label{eq:fixed-K10-same}
 K_{10}(x,y)
 =
 \left(\frac{\dd}{\dd y}\right)^B
 (E_y-rI)^{\mathrm{fall}\,A}
 [y^r\mathcal L_{P,r}(x,y)]
 \qquad(B\geq1),
\end{equation}
\begin{equation}\label{eq:fixed-K10-birth}
 K_{10}(x,y)
 =
 \left(\frac{\dd}{\dd y}\right)^A
 \mathcal L_{P,r}(x,y)
 \qquad(B=0),
\end{equation}
and
\begin{equation}\label{eq:fixed-K01}
 K_{01}(x,y)
 =
 -W(x,y)
 +
 \int_0^\infty W(x,z)K_{11}(z,y)\,\dd z.
\end{equation}
\end{thm}

\begin{proof}
Choose the physical side which is weakly larger at both endpoints. There are $B$ increments of that side and $A$ increments of the other side. Since
\[
 q+B=Q\geq P=p+A,
\]
the path that first performs all $B$ larger-side increments and then all $A$ smaller-side increments never crosses the diagonal. Proposition~\ref{prop:shell-path-independence} permits this ordering, and Theorem~\ref{thm:arbitrary-path-kernel} supplies the extended kernel.

We first compute the real one-particle transfer. If $A=0$, every step is same-size, every intermediate power is zero, and the terminal weight is $\ee^{-z}$. The $B-1$ ordered intermediate variables contribute the simplex volume
\[
 \frac{(z-x)^{B-1}}{(B-1)!},
\]
which gives \eqref{eq:fixed-W-A-zero}.

If $B=0$, every step is a birth. All interior powers vanish, while the terminal power is $r$. The $A-1$ intermediate variables contribute
\[
 \frac{(z-x)^{A-1}}{(A-1)!},
\]
giving \eqref{eq:fixed-W-B-zero}.

Suppose $A,B\geq1$. Inside each constant-type segment all powers vanish. At the unique turn, after the $B$ same-size steps and before the first birth step, the rectangularity is
\[
 r_0+B=Q-p=r+A.
\]
The corresponding $\kappa$ exponent is $-(r+A)$, while the terminal exponent is $r$. If $u$ is the coordinate at the turn, the variables between $x$ and $u$ contribute
\[
 \frac{(u-x)^{B-1}}{(B-1)!},
\]
and those between $u$ and $z$ contribute
\[
 \frac{(z-u)^{A-1}}{(A-1)!}.
\]
Multiplication by $\ee^{-z}z^ru^{-(r+A)}$ and integration over $x<u<z$ proves \eqref{eq:fixed-W-positive}.

Lemma~\ref{lem:laguerre-hankel} identifies the terminal block as \eqref{eq:fixed-K11}. At the earlier endpoint, the marginal squared-singular-value law is the complex Laguerre law for a $p$-by-$q$ rectangle. In the outgoing-step gauge of Theorem~\ref{thm:arbitrary-path-kernel}, the forward functions have common factor $y^{r_0}$ when the outgoing step is same-size and common factor one when it is a birth. The dual functions have respectively the common factors $\ee^{-x}$ and $\ee^{-x}x^{r_0}$. Convolution with $\ind_{\{y>x\}}$ raises a monomial degree by one, and multiplication by the level weight gives degree $q-i$ in the same-size case and $p-i$ in the birth case. Lemma~\ref{lem:laguerre-hankel}, applied with $(P,Q)=(p,q)$, then gives \eqref{eq:fixed-K00-same} and \eqref{eq:fixed-K00-birth}.

Proposition~\ref{prop:terminal-earlier-differential} gives \eqref{eq:fixed-K10-same} when $B\geq1$. When $B=0$, it gives
\[
 K_{10}(x,y)
 =
 \left(\frac{\dd}{\dd y}\right)^A
 (rI+E_y)^{\mathrm{fall}\,0}
 \mathcal L_{P,r}(x,y),
\]
which is \eqref{eq:fixed-K10-birth}.

Finally, write $C_{ab}$ for the finite-rank term in the Eynard--Mehta kernel. Backward propagation gives
\[
 C_{01}(x,y)
 =
 \int_0^\infty W(x,z)C_{11}(z,y)\,\dd z.
\]
At equal terminal time, $C_{11}=K_{11}$. From the earlier level to the later level, the full kernel subtracts the free transfer. This proves \eqref{eq:fixed-K01}.
\end{proof}

The three displayed transfers in Theorem~\ref{thm:fixed-orientation-kernel} encode the zero-turn and one-turn geometries separately. In every case $K_{10}$ is the later-to-earlier finite-rank block, while $K_{01}$ contains the negative free transfer and its propagated finite-rank correction. The two displayed $K_{00}$ forms are outgoing-step gauges of the same earlier Laguerre marginal law.

For orientation reversal, $p<q$ are the earlier smaller and larger dimensions and $P<Q$ are the later smaller and larger dimensions, but the physical side of size $p$ grows to $Q$ while the other side has size $q$ earlier and $P$ later. Thus $r_0=q-p$ counts births needed to reach the square, $B=Q-q$ counts the subsequent same-size increments of the first side, and $C=P-q$ counts the final births of the other side; $A=P-p=r_0+C$ is the ordered smaller-dimension displacement entering the differential block. The canonical exposure order is therefore $r_0$ births, then $B$ same-size steps, then $C$ births. The case $C=0$ has no second turn, while $C\geq1$ produces the second turning integral. Passing through a square alone is not the defining geometry: the physical side that is smaller at the earlier endpoint becomes the larger side at the later endpoint. I use \emph{strict orientation reversal} for $C\geq1$, when both physical dimensions grow; the theorem also includes the boundary case $C=0$, in which only the side that was initially smaller grows past the other. Abstract chain composition is prior \cite[Theorem~1.4]{BorodinRains2005} and \cite[Theorem~1]{OkounkovReshetikhin2003}; the theorem gives an explicit orientation-reversing formula in the present normalization, with direct priority unresolved after transposition, composition, and gauge comparison.

\begin{thm}[Two-level kernel for orientation-reversing nested rectangles]\label{thm:orientation-reversing-kernel}
Suppose one physical side has size $p$ at the earlier rectangle and $Q$ at the later rectangle, while the other has size $q$ earlier and $P$ later, where
\[
 p<q\leq P<Q.
\]
Put
\[
 \begin{aligned}
 r_0&=q-p,
 &C&=P-q,
 &B&=Q-q,\\
 A&=P-p=r_0+C,
 &r&=Q-P=B-C.
 \end{aligned}
\]
Define $\mathcal L_{d,\alpha}$ as in Theorem~\ref{thm:fixed-orientation-kernel}.

Set $W(x,z)=0$ for $z\leq x$. If $C=0$, put
\begin{equation}\label{eq:reverse-W-C-zero}
 W(x,z)
 =
 \ee^{-z}\frac{(z-x)^{r_0+B-1}}{(r_0+B-1)!}.
\end{equation}
If $C\geq1$, put
\begin{equation}\label{eq:reverse-W-positive}
 W(x,z)
 =
 \ee^{-z}z^r
 \int_x^z
 \frac{(u-x)^{r_0+B-1}u^{-B}(z-u)^{C-1}}
 {(r_0+B-1)!(C-1)!}\,\dd u.
\end{equation}

The two-level positive squared-singular-value process is determinantal, with one exact kernel
\begin{align}
 K_{00}(x,y)
 &=
 \ee^{-x}x^{r_0}\mathcal L_{p,r_0}(x,y),
 \label{eq:reverse-K00}\\
 K_{11}(x,y)
 &=
 \ee^{-y}y^r\mathcal L_{P,r}(x,y),
 \label{eq:reverse-K11}\\
 K_{10}(x,y)
 &=
 \left(\frac{\dd}{\dd y}\right)^A
 (rI+E_y)^{\mathrm{fall}\,B}
 \mathcal L_{P,r}(x,y),
 \label{eq:reverse-K10}\\
 K_{01}(x,y)
 &=
 -W(x,y)+
 \int_0^\infty W(x,z)K_{11}(z,y)\,\dd z.
 \label{eq:reverse-K01}
\end{align}
The same statement holds with rows and columns interchanged.
\end{thm}

\begin{proof}
Consider the physical side whose size changes from $p$ to $Q$. Reorder the shell increments by first increasing this side $r_0=q-p$ times until the rectangle is square, then increasing it $B=Q-q$ more times, and finally increasing the other side $C=P-q$ times. The step sequence is
\[
 r_0\text{ births},\qquad
 B\text{ same-size steps},\qquad
 C\text{ births}.
\]
Proposition~\ref{prop:shell-path-independence} permits this reordering.

Consecutive steps of the same type have zero interior $\kappa$ exponent. At the first turn, the rectangle is square. The incoming birth has preceding rectangularity one and contributes zero, while the outgoing same-size step contributes no negative power. Thus the first turning exponent is zero.

If $C=0$, there is no second turn, the terminal exponent is zero, and every intermediate weight is one. There are $r_0+B-1$ ordered intermediate coordinates, which gives \eqref{eq:reverse-W-C-zero}.

Suppose $C\geq1$. Immediately before the second turn, the rectangularity is $B$. The unique nonzero interior exponent is therefore $-B$, and the terminal exponent is $r=B-C$. There are $r_0+B-1$ ordered coordinates between $x$ and the turning coordinate $u$, and $C-1$ between $u$ and $z$. Multiplying the two simplex volumes by $u^{-B}\ee^{-z}z^r$ gives \eqref{eq:reverse-W-positive}.

At the earlier level, the first canonical step is a birth. The forward functions therefore span the polynomials of degree less than $p$, while their duals span $\ee^{-x}x^{r_0}$ times those polynomials. Lemma~\ref{lem:laguerre-hankel} gives \eqref{eq:reverse-K00}. The terminal formula in that lemma gives \eqref{eq:reverse-K11}. Proposition~\ref{prop:terminal-earlier-differential}, applied in the outgoing birth gauge, gives \eqref{eq:reverse-K10}. Backward propagation of the finite-rank term and subtraction of the free transfer give \eqref{eq:reverse-K01}. Transposition preserves squared singular values, proving the row-column symmetric case.
\end{proof}

In the reversing geometry, the two formulas for $W$ distinguish the absence or presence of the final birth segment. The block direction remains unchanged: $K_{01}$ contains $-W$, and $K_{10}$ is obtained from the terminal Laguerre block by the displayed differential operators. Transposition covers the row--column symmetric endpoint configuration without merging the separate $C=0$ and $C\geq1$ cases.

The last algebraic step aligns the two endpoint Laguerre bases. The terminal degree is shifted by $A$ so that the surviving terminal modes match the earlier degree range, and the displayed factorial ratio is the finite-dimensional coefficient carried by each matched mode. The proposition keeps the cases with and without the earlier power separate.

\begin{prop}[Shifted-Laguerre terminal-to-earlier block]\label{prop:shifted-laguerre-block}
Use the notation of Theorems~\ref{thm:fixed-orientation-kernel} and~\ref{thm:orientation-reversing-kernel}. In the fixed-orientation case, if $B\geq1$, then
\begin{equation}\label{eq:shifted-laguerre-with-power}
 K_{10}(x,y)
 =
 (-1)^Ay^{r_0}
 \sum_{j=0}^{p-1}
 \frac{(j+A)!}{(j+r_0)!}
 L_{j+A}^{(r)}(x)L_j^{(r_0)}(y).
\end{equation}
If $B=0$, then
\begin{equation}\label{eq:shifted-laguerre-without-power}
 K_{10}(x,y)
 =
 (-1)^A
 \sum_{j=0}^{p-1}
 \frac{(j+A)!}{(j+r_0)!}
 L_{j+A}^{(r)}(x)L_j^{(r_0)}(y).
\end{equation}
In the orientation-reversing case, $K_{10}$ is also given by \eqref{eq:shifted-laguerre-without-power}. All Laguerre parameters in these formulas are nonnegative integers.
\end{prop}

\begin{proof}
Write $E_y=y\,\frac{\dd}{\dd y}$. We first record three identities. Termwise application of $E_y$ to the finite Laguerre series gives
\begin{equation}\label{eq:laguerre-euler-shift}
 (E_y-\alpha I)^{\mathrm{fall}\,h}
 [y^\alpha L_k^{(\alpha)}(y)]
 =
 (-1)^hy^{\alpha+h}L_{k-h}^{(\alpha+h)}(y)
\end{equation}
for $k\geq h$, with zero left side for $k<h$. Termwise differentiation gives
\begin{equation}\label{eq:laguerre-weighted-derivative}
 \left(\frac{\dd}{\dd y}\right)^h
 [y^\alpha L_k^{(\alpha)}(y)]
 =
 \frac{(k+\alpha)!}{(k+\alpha-h)!}
 y^{\alpha-h}L_k^{(\alpha-h)}(y)
\end{equation}
when $\alpha\geq h$. The ordinary derivative identity is
\begin{equation}\label{eq:laguerre-derivative}
 \left(\frac{\dd}{\dd y}\right)^hL_k^{(\alpha)}(y)
 =
 (-1)^hL_{k-h}^{(\alpha+h)}(y)
\end{equation}
for $k\geq h$, with zero for $k<h$.

Suppose first that the orientation is fixed and $B\geq1$. By Theorem~\ref{thm:fixed-orientation-kernel},
\[
 K_{10}(x,y)
 =
 \left(\frac{\dd}{\dd y}\right)^B
 (E_y-rI)^{\mathrm{fall}\,A}
 [y^r\mathcal L_{P,r}(x,y)].
\]
Apply \eqref{eq:laguerre-euler-shift} with $\alpha=r$ and $h=A$, followed by \eqref{eq:laguerre-weighted-derivative} with $\alpha=r+A$ and $h=B$. Terms with $k<A$ vanish, while the term with $k\geq A$ becomes
\[
 (-1)^A
 \frac{k!}{(k+r-B)!}
 y^{r+A-B}
 L_k^{(r)}(x)L_{k-A}^{(r+A-B)}(y).
\]
The dimension identities give
\[
 r+A-B=r_0.
\]
Set $k=j+A$. Since $P-A=p$, the range becomes $0\leq j\leq p-1$, proving \eqref{eq:shifted-laguerre-with-power}.

If $B=0$, Theorem~\ref{thm:fixed-orientation-kernel} gives
\[
 K_{10}=
 \left(\frac{\dd}{\dd y}\right)^A\mathcal L_{P,r}.
\]
Apply \eqref{eq:laguerre-derivative}, discard $k<A$, use $r+A=r_0$, and set $k=j+A$. This gives \eqref{eq:shifted-laguerre-without-power}.

For orientation reversal,
\[
 K_{10}(x,y)
 =
 \left(\frac{\dd}{\dd y}\right)^A
 (rI+E_y)^{\mathrm{fall}\,B}
 \mathcal L_{P,r}(x,y).
\]
Here
\[
 r=B-C,\qquad A=r_0+C.
\]
We work coefficientwise. On $y^\ell$, the falling operator contributes
\[
 (r+\ell)^{\mathrm{fall}\,B}.
\]
If $\ell<C$, this product contains zero. If $\ell\geq C$, it equals
\[
 \frac{(r+\ell)!}{(\ell-C)!}.
\]
Therefore
\[
 (rI+E_y)^{\mathrm{fall}\,B}L_k^{(r)}(y)
 =
 \sum_{\ell=C}^k
 (-1)^\ell
 \frac{(k+r)!}{(k-\ell)!(\ell-C)!\ell!}y^\ell.
\]
Applying the $A$th derivative annihilates the sum when $k<A$. For $k\geq A$, set $\ell=A+t$ to obtain
\[
 \begin{split}
 &\left(\frac{\dd}{\dd y}\right)^A
 (rI+E_y)^{\mathrm{fall}\,B}L_k^{(r)}(y)\\
 &\qquad=
 (-1)^A
 \sum_{t=0}^{k-A}
 (-1)^t
 \frac{(k+r)!}
 {(k-A-t)!(r_0+t)!t!}y^t\\
 &\qquad=
 (-1)^A
 \frac{(k+r)!}{(k-C)!}
 L_{k-A}^{(r_0)}(y).
 \end{split}
\]
Multiplication by $k!/(k+r)!$ and the substitution $k=j+A$ give \eqref{eq:shifted-laguerre-without-power}. The dimension assumptions imply $r,r_0\geq0$.
\end{proof}

The shifted factorial weights in Proposition~\ref{prop:shifted-laguerre-block} become, after the levelwise normalization in the following Airy-time result, the normalized mode-survival coefficient. That result analyzes the coefficient within the extended-Airy scaling framework; the present section supplies only the exact finite-dimensional Laguerre input and does not introduce a different Airy object or correlation scale. Separately, Lemma~\ref{lem:two-level-reductions} later converts the resulting cross-block estimates into upper-occupancy bounds and lower-gap Schur-complement bounds. These are the two handoffs from the exact-kernel module to the proof of the one-sided product-scale probability theorem.

The exact terminal-to-earlier Laguerre block from Proposition~\ref{prop:shifted-laguerre-block} leaves one coefficient to be placed in soft-edge coordinates. Passing from the polynomial basis to the orthonormal Laguerre wavefunctions contributes one square-root factorial normalization at each level. Division by the levelwise constant $c_0$ then exposes the quotient $c_\ell/c_0$, later denoted $d_\ell$: this is the normalized mode-survival factor joining the earlier mode indexed by $p-\ell$ to the terminal mode indexed by $P-\ell$. Its product form records the row and column increments separately. The purpose of the next proposition is therefore a coordinate conversion of the exact finite-dimensional coefficient, not the introduction of a different Airy object or correlation scale.

\begin{prop}[Airy-time coefficient and the row-column correlation window]\label{prop:airy-time-coefficient}
Let $p,q,P,Q$ be positive integers satisfying
\[
 p\leq q,\qquad P\leq Q,\qquad p\leq P,\qquad q\leq Q.
\]
Put
\[
 A=P-p,\qquad B=Q-q,\qquad r_0=q-p,\qquad r=Q-P.
\]
For $0\leq n\leq m$, define
\[
 \varphi_{n,m}(x)
 =
 \sqrt{\frac{n!}{m!}}\,
 \ee^{-x/2}x^{(m-n)/2}L_n^{(m-n)}(x).
\]
For $1\leq\ell\leq p$, define
\[
 c_\ell
 =
 \sqrt{
 \frac{(P-\ell)!(Q-\ell)!}
 {(p-\ell)!(q-\ell)!}
 }.
\]
Then
\begin{equation}\label{eq:normalized-shifted-sum}
 \begin{split}
 &\sum_{\ell=1}^p
 \frac{(P-\ell)!}{(q-\ell)!}
 L_{P-\ell}^{(r)}(x)L_{p-\ell}^{(r_0)}(y)\\
 &\quad=
 \ee^{(x+y)/2}x^{-r/2}y^{-r_0/2}
 \sum_{\ell=1}^p
 c_\ell
 \varphi_{P-\ell,Q-\ell}(x)
 \varphi_{p-\ell,q-\ell}(y).
 \end{split}
\end{equation}
Moreover, with
\[
 c_0=\sqrt{\frac{P!Q!}{p!q!}},
\]
one has
\begin{equation}\label{eq:coefficient-product}
 \frac{c_\ell}{c_0}
 =
 \left\{
 \prod_{h=0}^{\ell-1}
 \frac{(p-h)(q-h)}{(P-h)(Q-h)}
 \right\}^{1/2}.
\end{equation}

Define
\[
 \delta_{p,q}
 =
 \frac{(\sqrt p+\sqrt q)^{2/3}}{(pq)^{1/3}},
 \qquad
 \tau_{p,q;A,B}
 =
 \frac{A/p+B/q}{2\delta_{p,q}}.
\]
Fix $c_{\mathrm{box}}>0$ and $C_{\mathrm{box}}<\infty$. There is a constant $C$ depending only on them such that, whenever
\[
 c_{\mathrm{box}}N\leq p,q\leq C_{\mathrm{box}}N,
 \qquad
 A+B\leq\frac{c_{\mathrm{box}}N}{4},
 \qquad
 1\leq\ell\leq\frac{c_{\mathrm{box}}N}{2},
\]
one has
\begin{equation}\label{eq:coefficient-log-error}
 \left|
 \log\frac{c_\ell}{c_0}
 +
 \tau_{p,q;A,B}\delta_{p,q}\ell
 \right|
 \leq
 C\frac{(A+B)\ell^2+(A^2+B^2)\ell}{N^2}.
\end{equation}
If $D_N>0$ satisfies
\[
 \frac{\log D_N}{\log N}\to0,
 \qquad
 A+B\leq N^{2/3}D_N,
 \qquad
 \ell\leq N^{1/3}D_N,
\]
then uniformly over these $\ell$,
\begin{equation}\label{eq:coefficient-asymptotic}
 \frac{c_\ell}{c_0}
 =
 \exp(-\tau_{p,q;A,B}\delta_{p,q}\ell)(1+o(1)).
\end{equation}
If
\[
 \frac pN\to a>0,\qquad
 \frac qN\to b>0,\qquad
 \frac A{N^{2/3}}\to\alpha,\qquad
 \frac B{N^{2/3}}\to\beta,
\]
then
\begin{equation}\label{eq:airy-time-limit}
 \tau_{p,q;A,B}
 \longrightarrow
 \frac{\alpha/a+\beta/b}{2d_{a,b}},
 \qquad
 d_{a,b}
 =
 \frac{(\sqrt a+\sqrt b)^{2/3}}{(ab)^{1/3}}.
\end{equation}
Thus $\tau$ tends to zero when $A+B=o(N^{2/3})$, remains of order one when $A+B$ is of order $N^{2/3}$, and tends to infinity when
\[
 \frac{A+B}{N^{2/3}}\to\infty,
 \qquad
 A+B=o(N).
\]

For the nested rectangles $(M_N,N)$, assume
\[
 \frac{M_N}{N}\to\gamma\in(0,\infty),
 \qquad
 0\leq M_{N+1}-M_N\leq K.
\]
Compare $(M_N,N)$ with $(M_{N+H},N+H)$, where $1\leq H=o(N)$, and sort their dimensions into $(p,q)$ and $(P,Q)$. Then there are constants $0<c<C<\infty$, depending only on $\gamma$ and $K$, such that
\begin{equation}\label{eq:project-correlation-window}
 c\frac{H}{N^{2/3}}
 \leq
 \tau_{p,q;P-p,Q-q}
 \leq
 C\frac{H}{N^{2/3}}
\end{equation}
for all sufficiently large $N$.
\end{prop}

\begin{proof}
By definition,
\[
 \begin{split}
 &\varphi_{P-\ell,Q-\ell}(x)
 \varphi_{p-\ell,q-\ell}(y)\\
 &\quad=
 \sqrt{
 \frac{(P-\ell)!(p-\ell)!}
 {(Q-\ell)!(q-\ell)!}}
 \ee^{-(x+y)/2}x^{r/2}y^{r_0/2}
 L_{P-\ell}^{(r)}(x)L_{p-\ell}^{(r_0)}(y).
 \end{split}
\]
Multiplication by $c_\ell$ reduces the coefficient to $(P-\ell)!/(q-\ell)!$, proving \eqref{eq:normalized-shifted-sum}. Dividing $c_\ell$ by $c_0$ and expanding the factorial quotients gives \eqref{eq:coefficient-product}.

Taking logarithms in \eqref{eq:coefficient-product} yields
\begin{equation}\label{eq:coefficient-log-sum}
 \log\frac{c_\ell}{c_0}
 =
 -\frac12
 \sum_{h=0}^{\ell-1}
 \left[
 \log\left(1+\frac{A}{p-h}\right)
 +
 \log\left(1+\frac{B}{q-h}\right)
 \right].
\end{equation}
Under the macroscopic hypotheses, $p-h$ and $q-h$ are bounded below by a fixed positive multiple of $N$, while $A/(p-h)$ and $B/(q-h)$ lie in a fixed compact interval. The Taylor bound
\[
 \abs{\log(1+z)-z}\leq Cz^2
\]
on this interval and
\[
 \left|
 \frac{A}{p-h}-\frac Ap
 \right|
 =
 \frac{Ah}{p(p-h)}
 \leq
 C\frac{Ah}{N^2}
\]
give
\[
 \left|
 \log\left(1+\frac{A}{p-h}\right)-\frac Ap
 \right|
 \leq
 C\frac{A^2+Ah}{N^2}.
\]
The analogous estimate holds for $B$ and $q$. Summing over $h$ and using
\[
 \sum_{h=0}^{\ell-1}h\leq\frac{\ell^2}{2},
\]
together with
\[
 \tau_{p,q;A,B}\delta_{p,q}
 =
 \frac12\left(\frac Ap+\frac Bq\right),
\]
proves \eqref{eq:coefficient-log-error}.

Under the hypotheses for \eqref{eq:coefficient-asymptotic}, the right side of \eqref{eq:coefficient-log-error} is at most
\[
 C\left(N^{-2/3}D_N^3+N^{-1/3}D_N^3\right)=o(1),
\]
because $D_N=N^{o(1)}$. Exponentiating proves \eqref{eq:coefficient-asymptotic}.

For \eqref{eq:airy-time-limit}, direct substitution gives
\[
 N^{1/3}\delta_{p,q}\to d_{a,b}
\]
and
\[
 N^{1/3}\left(\frac Ap+\frac Bq\right)
 \to
 \frac\alpha a+\frac\beta b.
\]
The claimed limit follows. More generally, if $p,q$ are comparable to $N$ and $A,B=o(N)$, then $\tau$ is bounded above and below by positive constants times $(A+B)/N^{2/3}$. This proves the three scale assertions.

For the nested path, all four sorted dimensions are bounded between positive constant multiples of $N$ for sufficiently large $N$. Sorting preserves the sum of the physical dimensions, so
\[
 A+B
 =
 (P+Q)-(p+q)
 =
 H+M_{N+H}-M_N.
\]
Monotonicity and the increment bound give
\[
 H\leq A+B\leq(K+1)H.
\]
Moreover, $\delta_{p,q}$ is comparable to $N^{-1/3}$ and $A/p+B/q$ is comparable to $(A+B)/N$. Substitution proves \eqref{eq:project-correlation-window}. The argument does not depend on which physical side is larger, so it also covers orientation reversal.
\end{proof}

On compact aspect-ratio sets, $\delta_{p,q}$ is comparable to $N^{-1/3}$, whereas $A/p+B/q$ is comparable to $(A+B)/N$. Hence the Airy-time coordinate is comparable to $(A+B)/N^{2/3}$. A displacement of order $N^{2/3}$ produces order-one Airy time, while a supercritical displacement with $(A+B)/N^{2/3}\to\infty$ and $A+B=o(N)$ sends that time to infinity. In the mode range for which $\delta_{p,q}\ell$ is order one, the asymptotic coefficient is correspondingly damped by an exponential with diverging time parameter. The exponent $2/3$ is the classical soft-edge correlation scale; the proposition calibrates its row--column conversion coefficient in the present normalization.

The extended-Airy framework and its $N^{2/3}$ correlation window are classical. Fixed-orientation Laguerre/Wishart-minor structure is prior in substance \cite{AdlerVanMoerbekeWang2013}, as is the one-time complex-Wishart soft-edge normalization \cite{EK06}. I use that framework here to put the two endpoint geometries admitted below into one orthonormal spectral gauge. The paper-specific output at this stage is the exact finite-$N$ coefficient conversion and the uniform error analysis in the present row--column coordinates. Direct priority for the algebraic form of the Airy-time coefficient remains unresolved.

\begin{thm}[Discrete extended-Airy spectral form]\label{thm:discrete-extended-airy}
Consider either the fixed-orientation geometry of Theorem~\ref{thm:fixed-orientation-kernel} or the orientation-reversing geometry of Theorem~\ref{thm:orientation-reversing-kernel}. Use the resulting integers
\[
 \begin{aligned}
 p&\leq q,
 &P&\leq Q,
 &A&=P-p,\\
 B&=Q-q,
 &r_0&=q-p,
 &r&=Q-P,
 \end{aligned}
\]
and assume $A+B\geq2$. Define
\[
 \varphi_{n,m}(x)
 =
 \sqrt{\frac{n!}{m!}}
 \ee^{-x/2}x^{(m-n)/2}L_n^{(m-n)}(x),
\]
and, for every integer $\ell\leq p$,
\[
 c_\ell
 =
 \sqrt{
 \frac{(P-\ell)!(Q-\ell)!}
 {(p-\ell)!(q-\ell)!}},
 \qquad
 c_0=\sqrt{\frac{P!Q!}{p!q!}},
 \qquad
 d_\ell=\frac{c_\ell}{c_0}.
\]
The two-level process admits a correlation kernel whose blocks are
\begin{align}
 \widehat K_{00}(x,y)
 &=
 \sum_{\ell=1}^p
 \varphi_{p-\ell,q-\ell}(x)
 \varphi_{p-\ell,q-\ell}(y),
 \label{eq:airy-discrete-00}\\
 \widehat K_{11}(x,y)
 &=
 \sum_{\ell=1}^P
 \varphi_{P-\ell,Q-\ell}(x)
 \varphi_{P-\ell,Q-\ell}(y),
 \label{eq:airy-discrete-11}\\
 \widehat K_{10}(x,y)
 &=
 \sum_{\ell=1}^p
 d_\ell
 \varphi_{P-\ell,Q-\ell}(x)
 \varphi_{p-\ell,q-\ell}(y),
 \label{eq:airy-discrete-10}\\
 \widehat K_{01}(x,y)
 &=
 -\sum_{\ell\leq0}
 d_\ell^{-1}
 \varphi_{p-\ell,q-\ell}(x)
 \varphi_{P-\ell,Q-\ell}(y).
 \label{eq:airy-discrete-01}
\end{align}
The series in \eqref{eq:airy-discrete-01} converges in
\[
 L^2((0,\infty)^2)
\]
and defines a Hilbert--Schmidt operator.
\end{thm}

\begin{proof}
For every nonnegative integer $\alpha$, the functions
\[
 \{\varphi_{k,k+\alpha}:k\geq0\}
\]
form an orthonormal basis of $L^2(0,\infty)$. Orthonormality follows from \eqref{eq:laguerre-orthogonality}. To prove completeness, suppose $f$ is orthogonal to every such function. Then it is orthogonal to
\[
 \ee^{-x/2}x^{\alpha/2}
\]
times every polynomial. The Laplace transform
\[
 \int_0^\infty
 f(x)\ee^{-x/2}x^{\alpha/2}\ee^{zx}\,\dd x
\]
is analytic for $\operatorname{Re}z$ in a neighborhood of zero by Cauchy--Schwarz, and every derivative at zero vanishes. It therefore vanishes on a real interval. Uniqueness of the Laplace transform gives $f=0$.

We first derive \eqref{eq:airy-discrete-00}--\eqref{eq:airy-discrete-10}. Apply the scalar gauge
\[
 \widehat K_{ab}(x,y)
 =
 g_a(x)^{-1}K_{ab}(x,y)g_b(y),
\]
where
\[
 g_1(x)=(-1)^Ac_0\ee^{x/2}x^{-r/2}.
\]
At level zero, take
\[
 g_0(x)=\ee^{-x/2}x^{-r_0/2}
\]
in the fixed-orientation case with $B\geq1$, and
\[
 g_0(x)=\ee^{-x/2}x^{r_0/2}
\]
in the fixed-orientation case with $B=0$ and in the orientation-reversing case. A blockwise scalar gauge preserves every determinantal correlation function. Substitution into the equal-time blocks of Theorems~\ref{thm:fixed-orientation-kernel} and~\ref{thm:orientation-reversing-kernel} gives \eqref{eq:airy-discrete-00} and \eqref{eq:airy-discrete-11}. Proposition~\ref{prop:shifted-laguerre-block}, together with \eqref{eq:normalized-shifted-sum}, gives \eqref{eq:airy-discrete-10}.

It remains to derive \eqref{eq:airy-discrete-01}. We first establish two Laguerre convolution identities. For $h\geq1$ and $n\geq0$,
\begin{equation}\label{eq:laguerre-convolution-one}
 \int_x^\infty
 \ee^{-z}\frac{(z-x)^{h-1}}{(h-1)!}
 L_n^{(\alpha+h)}(z)\,\dd z
 =
 \ee^{-x}L_n^{(\alpha)}(x).
\end{equation}
This follows by multiplying the generating function
\[
 \sum_{n\geq0}L_n^{(\alpha+h)}(z)t^n
 =
 (1-t)^{-\alpha-h-1}
 \exp\left(-\frac{zt}{1-t}\right)
\]
by the convolution weight, integrating for $\abs t<1$, and comparing coefficients.

For $n\geq h$,
\begin{equation}\label{eq:laguerre-convolution-two}
 \int_x^\infty
 \ee^{-z}z^\alpha
 \frac{(z-x)^{h-1}}{(h-1)!}
 L_n^{(\alpha)}(z)\,\dd z
 =
 (-1)^h\frac{(n-h)!}{n!}
 \ee^{-x}x^{\alpha+h}L_{n-h}^{(\alpha+h)}(x).
\end{equation}
For $h=1$, differentiation of the Rodrigues formula gives
\[
 \frac{\dd}{\dd x}
 \left[
 \ee^{-x}x^{\alpha+1}L_{n-1}^{(\alpha+1)}(x)
 \right]
 =
 n\ee^{-x}x^\alpha L_n^{(\alpha)}(x).
\]
The expression in brackets vanishes at positive infinity. Integration from $x$ to infinity gives the $h=1$ case of \eqref{eq:laguerre-convolution-two}, including its negative sign. Iteration proves the general identity. We also use
\begin{equation}\label{eq:negative-laguerre-parameter}
 L_n^{(-m)}(x)
 =
 (-x)^m\frac{(n-m)!}{n!}L_{n-m}^{(m)}(x),
 \qquad n\geq m\geq0,
\end{equation}
which follows directly from the finite series.

Let $W$ be the applicable real free propagator and put
\[
 \widehat W=g_0^{-1}Wg_1.
\]
We claim that, for every $k\geq P$,
\begin{equation}\label{eq:free-mode-mapping}
 \widehat W\varphi_{k,k+r}
 =
 d_{P-k}^{-1}
 \varphi_{k-A,k-A+r_0}.
\end{equation}

For fixed orientation with $A=0$, use \eqref{eq:fixed-W-A-zero} and \eqref{eq:laguerre-convolution-one} with $h=B$ to lower the terminal parameter from $r=r_0+B$ to $r_0$. The wavefunction normalizations and endpoint gauges give the coefficient $d_{P-k}^{-1}$.

For fixed orientation with $B=0$, use \eqref{eq:fixed-W-B-zero} and \eqref{eq:laguerre-convolution-two} with $h=A$ and $\alpha=r$. This lowers the degree from $k$ to $k-A$ and raises the parameter from $r$ to $r+A=r_0$. The sign $(-1)^A$ cancels the sign in $g_1$, and normalization gives \eqref{eq:free-mode-mapping}.

For fixed orientation with $A,B\geq1$, use \eqref{eq:fixed-W-positive}. First apply \eqref{eq:laguerre-convolution-two} in the terminal variable with $h=A$ and $\alpha=r$. This produces
\[
 (-1)^A\frac{(k-A)!}{k!}
 \ee^{-u}u^{r+A}L_{k-A}^{(r+A)}(u).
\]
The factor $u^{r+A}$ cancels the turning factor $u^{-(r+A)}$. Identity \eqref{eq:laguerre-convolution-one}, with $h=B$, lowers the parameter from $r+A$ to
\[
 r+A-B=r_0.
\]
Normalization again gives \eqref{eq:free-mode-mapping}.

For orientation reversal, write $C=P-q$, so $A=r_0+C$ and $r=B-C$. If $C\geq1$, use \eqref{eq:reverse-W-positive}. Identity \eqref{eq:laguerre-convolution-two} first performs the $C$ final births and produces
\[
 (-1)^C\frac{(k-C)!}{k!}
 \ee^{-u}u^{r+C}L_{k-C}^{(r+C)}(u).
\]
Since $r+C=B$, the factor $u^B$ cancels the turning factor $u^{-B}$. Identity \eqref{eq:laguerre-convolution-one}, with $h=r_0+B$, changes the remaining parameter from $B$ to $-r_0$. Since $k\geq P$ implies $k-C\geq r_0$, \eqref{eq:negative-laguerre-parameter} changes
\[
 L_{k-C}^{(-r_0)}
\]
into
\[
 (-x)^{r_0}\frac{(k-A)!}{(k-C)!}L_{k-A}^{(r_0)}.
\]
The total sign is $(-1)^{C+r_0}=(-1)^A$, and the total factorial coefficient is $(k-A)!/k!$. The sign cancels the one in $g_1$, and normalization gives \eqref{eq:free-mode-mapping}. If $C=0$, use \eqref{eq:reverse-W-C-zero}, followed by \eqref{eq:laguerre-convolution-one} and \eqref{eq:negative-laguerre-parameter}, to obtain the same result.

The normalization in every case is
\[
 d_{P-k}^2
 =
 \frac{k!(k+r)!}{(k-A)!(k-A+r_0)!}
 \frac{p!q!}{P!Q!}.
\]

The Eynard--Mehta form gives, after the same gauge,
\[
 \widehat K_{01}
 =
 -\widehat W+\widehat W\widehat K_{11}
 =
 -\widehat W(I-\widehat K_{11}).
\]
By \eqref{eq:airy-discrete-11} and completeness, $I-\widehat K_{11}$ is the orthogonal projection onto
\[
 \{\varphi_{k,k+r}:k\geq P\}.
\]
Insert this expansion, apply \eqref{eq:free-mode-mapping}, and put $\ell=P-k$. This gives \eqref{eq:airy-discrete-01}.

Finally, for $L\geq0$,
\[
 d_{-L}^2
 =
 \prod_{i=1}^A\frac{p+L+i}{p+i}
 \prod_{j=1}^B\frac{q+L+j}{q+j}.
\]
Since $A+B\geq2$,
\[
 \sum_{L\geq0}d_{-L}^{-2}<\infty.
\]
Equation \eqref{eq:free-mode-mapping} maps one orthonormal family to another, so the series in \eqref{eq:airy-discrete-01} converges in $L^2((0,\infty)^2)$ and defines a Hilbert--Schmidt operator.
\end{proof}

The four blocks above assign distinct later roles to the same exact two-level law. The equal-level sums are finite Laguerre projections. The positive-index series in $\widehat K_{10}$ and the nonpositive-index series in $\widehat K_{01}$ isolate cross-level dependence, with all row--column time information carried by $d_\ell$ and $d_\ell^{-1}$. This is a discrete extended-Airy spectral representation at finite $N$; the subsequent estimates determine how its modes behave on the moving soft-edge windows required by the probability theorem.

A product-scale joint bound also requires the marginal rarity scale. The next proposition specializes finite-dimensional Laguerre deviation estimates to the logarithmic walls used later. Its right- and left-tail exponents become $\kappa_+=4u^{3/2}/3$ and $\kappa_-=v^3/12$. Those powers determine the vanishing product of the two marginals against which the cross-level correction must eventually be compared.

\begin{prop}[One-time Laguerre log-window tails]\label{prop:laguerre-log-tails}
Fix $M_0>1$. For each $n\geq1$, let $m_n$ be an integer satisfying
\[
 m_n>n-1,
 \qquad
 \frac{m_n}{n}\leq M_0.
\]
Let $\Lambda_n$ be the largest eigenvalue of $X_n^*X_n$, where $X_n$ is an $m_n$-by-$n$ matrix of independent centered circular complex Gaussian variables of second absolute moment one. Define
\[
 a_n=(\sqrt{m_n}+\sqrt n)^2,
 \qquad
 b_n=(m_nn)^{-1/6}(\sqrt{m_n}+\sqrt n)^{4/3}.
\]
For every fixed $u,v>0$,
\begin{align}
 \Prob\left(
 \Lambda_n\geq a_n+ub_n(\log n)^{2/3}
 \right)
 &=
 \exp\left(
 -\frac43u^{3/2}\log n+o(\log n)
 \right),
 \label{eq:laguerre-right-log-tail}\\
 \Prob\left(
 \Lambda_n\leq a_n-vb_n(\log n)^{1/3}
 \right)
 &=
 \exp\left(
 -\frac1{12}v^3\log n+o(\log n)
 \right).
 \label{eq:laguerre-left-log-tail}
\end{align}
\end{prop}

\begin{proof}
The required uniform finite-dimensional Laguerre tail estimates are supplied by \cite[Theorems~1.1(ii), 1.3(ii), 1.4(ii), and 1.5(ii)]{BBBK24}.

Write
\[
 A(m,n)=(\sqrt m+\sqrt n)^2,
 \qquad
 B(m,n)=(mn)^{-1/6}(\sqrt m+\sqrt n)^{4/3}.
\]
For every $M>0$ and sufficiently small $\eta>0$, the cited right-tail estimates give constants $N_R,T_R,\gamma_R>0$ such that, uniformly for
\[
 n\geq N_R,\qquad
 m>n-1,\qquad
 \frac mn\leq M,\qquad
 T_R\leq t\leq\gamma_Rn^{2/3},
\]
one has
\begin{equation}\label{eq:finite-right-tail}
 \begin{aligned}
 \exp\left(-\frac43(1+\eta)t^{3/2}\right)
 &\leq\Prob\left(
 \Lambda_1^{(n,m)}\geq A(m,n)+tB(m,n)\right)\\
 &\leq\exp\left(-\frac43(1-\eta)t^{3/2}\right).
 \end{aligned}
\end{equation}
Here $\Lambda_1^{(n,m)}$ denotes the largest point of the corresponding complex Laguerre ensemble.

Choose $\delta$ in the sufficiently small range permitted by both cited left-tail estimates and with $0<\delta<1/9$. Then, for every $M>0$ and sufficiently small $\eta>0$, there are $N_L,T_L$ such that, uniformly for
\[
 n\geq N_L,\qquad
 m>n-1,\qquad
 \frac mn\leq M,\qquad
 T_L\leq t\leq n^{1/9-\delta},
\]
one has
\begin{equation}\label{eq:finite-left-tail}
 \begin{aligned}
 \exp\left(-\frac1{12}(1+\eta)t^3\right)
 &\leq\Prob\left(
 \Lambda_1^{(n,m)}\leq A(m,n)-tB(m,n)\right)\\
 &\leq\exp\left(-\frac1{12}(1-\eta)t^3\right).
 \end{aligned}
\end{equation}
The upper probability bound is valid through the larger range $t\leq n^{1/6-\delta}$, while the lower bound is valid through $t\leq n^{1/9-\delta}$, so their common range is the one displayed.

Fix $u>0$ and put
\[
 t_n^+=u(\log n)^{2/3}.
\]
Then $t_n^+\to\infty$ and $t_n^+/n^{2/3}\to0$. With $M=M_0$ and $m=m_n$, all hypotheses of \eqref{eq:finite-right-tail} hold for sufficiently large $n$. Since
\[
 (t_n^+)^{3/2}=u^{3/2}\log n,
\]
taking logarithms and dividing by $\log n$ gives
\[
 -\frac43(1+\eta)u^{3/2}
 \leq
 \frac{
 \log\Prob(\Lambda_n\geq a_n+ub_n(\log n)^{2/3})
 }{\log n}
 \leq
 -\frac43(1-\eta)u^{3/2}.
\]
Taking the lower and upper limits and then letting $\eta\downarrow0$ proves \eqref{eq:laguerre-right-log-tail}.

Fix $v>0$ and put
\[
 t_n^-=v(\log n)^{1/3}.
\]
Since $1/9-\delta>0$,
\[
 \log\left(
 \frac{t_n^-}{n^{1/9-\delta}}
 \right)
 =
 \log v+\frac13\log\log n-(1/9-\delta)\log n
 \longrightarrow-\infty.
\]
Thus $t_n^-\to\infty$ and
\[
 \frac{t_n^-}{n^{1/9-\delta}}\to0.
\]
The hypotheses of \eqref{eq:finite-left-tail} hold for sufficiently large $n$. Since
\[
 (t_n^-)^3=v^3\log n,
\]
the same logarithmic squeeze proves \eqref{eq:laguerre-left-log-tail}.
\end{proof}

\section{Airy asymptotics and Gaussian logarithmic tails}\label{sec:airy-tails}

\emph{Goal.} This section supplies the uniform soft-edge estimates needed to turn the exact two-level Laguerre representation into probability bounds at moving logarithmic walls. \emph{Obstacle.} The upper wall recedes to the right on the $(\log N)^{2/3}$ scale, the lower wall recedes to the left on the $(\log N)^{1/3}$ scale, neighboring Laguerre modes enter through a Christoffel--Darboux divided difference, and lower gap probabilities require control of an inverse whose norm grows as the wall moves left. \emph{Mechanism.} Positive-, negative-, and mixed-window Laguerre approximations are combined with Airy envelopes and the Airy-kernel spectral gap; the one-time logarithmic tails supply the marginal powers. \emph{Output.} The resulting bounds are uniform over compact aspect ratios, logarithmic windows, bounded deterministic shifts, and the mode ranges used in the following two-time proof.

The division of labor is deliberate. The positive approximation controls upper-half-line mode norms and the far-right tail. The negative approximation reaches the lower wall with a uniform $C^0$ error. The mixed-window $C^1$ estimate controls neighboring modes and the divided-difference kernel on a box spanning both signs. The Airy gap converts equal-level kernel comparison into a usable inverse bound. Finally, the Gaussian one-time statements transfer the Laguerre exponents to the project normalization and isolate the supercritical marginal regime. The Airy function, Airy kernel, fixed-rectangle soft-edge scale, and underlying one-time asymptotics are classical inputs; the model-specific task is their uniform assembly in the half-integer row--column normalization needed for the finite-$N$ estimate.

\subsection{Airy and Laguerre asymptotics}

\begin{lem}[Airy envelopes]\label{lem:airy-envelopes}
Let \(\Ai\) be the nonzero real solution of
\[
\Ai''(x)=x\Ai(x),\qquad x\in\R,
\]
that tends to zero as \(x\to+\infty\). There are constants
\(0<c_{\Ai}<C_{\Ai}<\infty\) such that, for every \(x\geq1\),
\[
c_{\Ai}x^{-1/4}\exp\left(-\frac23x^{3/2}\right)
\leq \abs{\Ai(x)}
\leq C_{\Ai}x^{-1/4}\exp\left(-\frac23x^{3/2}\right),
\]
and, for every \(x\geq0\),
\[
\abs{\Ai(-x)}\leq C_{\Ai}(1+x)^{-1/4}.
\]
\end{lem}

\begin{proof}
The proof is recorded in Appendix~\ref{app:proof-airy-envelopes}.
\end{proof}

The positive Airy envelope is used on an unbounded half-line, so the next approximation must retain exponential decay uniformly for every nonnegative soft-edge coordinate. This global form controls both the upper-event wall and the tail beyond the logarithmic cutoff introduced in the equal-level kernel comparison. Its half-integer center and scale are the normalization used throughout the operator estimates below.

\begin{prop}[Uniform positive-edge Laguerre approximation]\label{prop:positive-airy-laguerre}
Fix constants \(0<c_0<C_0<\infty\). For positive integers \(n\leq m\), define
\[
\phi_{n,m}(x)
=
\sqrt{\frac{n!}{m!}}\,
\exp(-x/2)x^{(m-n)/2}L_n^{(m-n)}(x),
\qquad x>0,
\]
\[
\mu_{n,m}
=
\left(\sqrt{m+\frac12}+\sqrt{n+\frac12}\right)^2,
\]
\[
\sigma_{n,m}
=
\left(\sqrt{m+\frac12}+\sqrt{n+\frac12}\right)
\left[
\left(m+\frac12\right)^{-1/2}
+
\left(n+\frac12\right)^{-1/2}
\right]^{1/3},
\]
and
\[
\delta_{n,m}
=
\left[
\left(m+\frac12\right)^{-1/2}
+
\left(n+\frac12\right)^{-1/2}
\right]^{2/3}.
\]
There are constants \(C_1<\infty\) and \(N_0\), depending only on \(c_0,C_0\), such that, for every integer \(N\geq N_0\), every pair \(n,m\) satisfying
\[
c_0N\leq n\leq m\leq C_0N,
\]
and every \(s\geq0\),
\[
\begin{aligned}
&\abs{(-1)^n\sqrt{\sigma_{n,m}}\,
\phi_{n,m}(\mu_{n,m}+\sigma_{n,m}s)
-\sqrt{\delta_{n,m}}\,\Ai(s)}\\
&\qquad\leq
C_1\sqrt{\delta_{n,m}}\,N^{-2/3}\exp(-s/4).
\end{aligned}
\]
\end{prop}

\begin{proof}
We first record the auxiliary function and estimate from
\cite[Section~2.2.2, equation~(1), Section~3.2.1, Section~3.3, estimate~(I1), and Appendix~A.6]{EK06}.
That source takes a larger dimension \(m_p\) and a polynomial degree \(n_p\), with \(m_p\geq n_p\), and defines
\begin{equation}\label{eq:ek-F}
F_p(x)
=
(-1)^{n_p}\sigma_p^{-1/2}
\sqrt{\frac{n_p!}{m_p!}}\,
x^{(m_p-n_p+1)/2}
\exp(-x/2)L_{n_p}^{(m_p-n_p)}(x).
\end{equation}
Its half-integer center and scale are
\[
\mu_p
=
\left(\sqrt{m_p+\frac12}+\sqrt{n_p+\frac12}\right)^2
\]
and
\[
\sigma_p
=
\left(\sqrt{m_p+\frac12}+\sqrt{n_p+\frac12}\right)
\left[
\left(m_p+\frac12\right)^{-1/2}
+
\left(n_p+\frac12\right)^{-1/2}
\right]^{1/3}.
\]
The auxiliary function is
\[
\theta_p(x)=F_p(x)\frac{\mu_p}{x}.
\]
The positive normalization \(r_p\) satisfies
\begin{equation}\label{eq:ek-r}
r_p^2
=
2\pi\exp[-(m_p+n_p+1)]
\frac{
(m_p+\frac12)^{m_p+1/2}
(n_p+\frac12)^{n_p+1/2}
}{
m_p!n_p!
}.
\end{equation}
Whenever \(m_p/n_p\to\gamma\in[1,\infty)\), for every fixed \(s_0\) and all sufficiently large \(n_p\), uniformly for \(s\geq s_0\),
\begin{equation}\label{eq:ek-estimate}
\abs{
\theta_p(\mu_p+\sigma_ps)-r_p\Ai(s)
}
\leq
C(s_0)n_p^{-2/3}\exp(-s/2).
\end{equation}
The constants may be chosen independently of \(\gamma\geq1\). The positive-coordinate threshold used in the proof is likewise uniform in \(\gamma\). The same source gives
\begin{equation}\label{eq:ek-r-asymptotic}
r_p
=
1+\frac1{48}\left(\frac1{m_p}+\frac1{n_p}\right)
+
o(m_p^{-1},n_p^{-1}),
\end{equation}
with a remainder independent of how the two dimensions tend to infinity.

Apply these statements with the exact substitution
\[
(m_p,n_p)=(m,n).
\]
Equations~\eqref{eq:ek-F} and the definition of \(\theta_p\) give
\begin{equation}\label{eq:ek-conversion}
\begin{aligned}
F_{m,n}(x)
&=(-1)^n\sigma_{n,m}^{-1/2}x^{1/2}\phi_{n,m}(x),\\
\theta_{m,n}(x)
&=(-1)^n\mu_{n,m}
[\sigma_{n,m}x]^{-1/2}\phi_{n,m}(x).
\end{aligned}
\end{equation}
Thus the factorial normalization, polynomial parameter, center, and scale coincide exactly.

We next make the uniformity quantifier explicit. Suppose that no compact-aspect threshold existed for \eqref{eq:ek-estimate} with \(s_0=0\) after replacing \(n_p^{-2/3}\) by \(N^{-2/3}\). There would then be a failing sequence with \(N\to\infty\) and
\[
c_0N\leq n\leq m\leq C_0N.
\]
Along a subsequence, \(m/n\) would tend to some
\(\gamma\in[1,C_0/c_0]\). Estimate~\eqref{eq:ek-estimate}, with its \(\gamma\)-independent constant, would apply eventually on that subsequence. Since
\[
n^{-2/3}\leq c_0^{-2/3}N^{-2/3},
\]
this is a contradiction. Hence there are constants \(C_{\mathrm{EK}}\) and \(N_{\mathrm{EK}}\), depending only on \(c_0,C_0\), such that, throughout the displayed family and for every \(s\geq0\),
\begin{equation}\label{eq:ek-uniform}
\abs{
\theta_{m,n}(\mu_{n,m}+\sigma_{n,m}s)
-
r_{m,n}\Ai(s)
}
\leq
C_{\mathrm{EK}}N^{-2/3}\exp(-s/2).
\end{equation}
The same compactness argument applied to \eqref{eq:ek-r-asymptotic}, or directly to its aspect-independent remainder, gives
\begin{equation}\label{eq:ek-r-uniform}
\abs{r_{m,n}-1}\leq C_rN^{-1},
\qquad
\abs{r_{m,n}}\leq2
\end{equation}
after increasing \(N_{\mathrm{EK}}\).

We also use the global Airy estimate
\cite[Lemma~2.1]{TU21}. There is \(C_A\) such that, for every complex \(z\),
\[
\abs{\Ai(z)}
\leq
C_A\frac{g_A(z)}{1+\abs{z}^{1/4}},
\qquad
\abs{\Ai'(z)}
\leq
C_A(1+\abs{z}^{1/4})g_A(z),
\]
where
\[
g_A(z)
=
\exp\left[-\frac23\operatorname{Re}(z^{3/2})\right]
\]
with the branch convention specified in the cited result. For real \(s\geq0\),
\begin{equation}\label{eq:global-positive-airy}
\abs{\Ai(s)}
\leq
C_A(1+s^{1/4})^{-1}
\exp\left[-\frac23s^{3/2}\right].
\end{equation}
This implies
\begin{equation}\label{eq:airy-weighted-suprema}
\sup_{s\geq0}\exp(s/4)\abs{\Ai(s)}<\infty,
\qquad
\sup_{s\geq0}s\exp(s/4)\abs{\Ai(s)}<\infty.
\end{equation}
On \(0\leq s\leq1\), this follows from \eqref{eq:global-positive-airy}. For \(s\geq1\),
\[
-\frac23s^{3/2}+\frac{s}{4}
\leq
-\frac5{12}s,
\]
and the remaining polynomial factors are bounded by the exponential.

Abbreviate
\[
\mu=\mu_{n,m},
\qquad
\sigma=\sigma_{n,m},
\qquad
\delta=\delta_{n,m},
\qquad
a=\frac{\sigma}{\mu}.
\]
The definitions give the exact identity
\begin{equation}\label{eq:sigma-delta-identity}
\frac{\sigma}{\sqrt{\mu}}=\sqrt{\delta},
\end{equation}
and compact-aspect comparability gives
\begin{equation}\label{eq:a-comparability}
cN^{-2/3}\leq a\leq CN^{-2/3}.
\end{equation}
For
\[
x=\mu+\sigma s=\mu(1+as),
\]
equations~\eqref{eq:ek-conversion} and~\eqref{eq:sigma-delta-identity} give
\begin{equation}\label{eq:positive-conversion}
(-1)^n\sqrt{\sigma}\,\phi_{n,m}(x)
=
\sqrt{\delta}\sqrt{1+as}\,\theta_{m,n}(x).
\end{equation}

By \eqref{eq:ek-uniform}, the contribution in
\eqref{eq:positive-conversion} of
\(\theta_{m,n}(x)-r_{m,n}\Ai(s)\) is at most
\begin{align}
C\sqrt{\delta}N^{-2/3}\sqrt{1+as}\exp(-s/2)
&\leq
C\sqrt{\delta}N^{-2/3}\exp(-s/4).
\label{eq:positive-error-one}
\end{align}
For the last inequality, \(a\) is uniformly bounded and
\[
\sqrt{1+as}\exp(-s/4)
\leq
C\sqrt{1+s}\exp(-s/4)
\]
is uniformly bounded.

It remains to estimate
\begin{equation}\label{eq:positive-prefactor-error}
\abs{\sqrt{1+as}\,r_{m,n}-1}\abs{\Ai(s)}.
\end{equation}
If \(as\leq1\), then \eqref{eq:ek-r-uniform} and
\(\sqrt{1+u}-1\leq u\) for \(u\geq0\) give
\[
\abs{\sqrt{1+as}\,r_{m,n}-1}
\leq
2as+C_rN^{-1}.
\]
Using \eqref{eq:airy-weighted-suprema},
\eqref{eq:a-comparability}, and \(N^{-1}\leq N^{-2/3}\), we obtain
\begin{equation}\label{eq:positive-small-as}
\abs{\sqrt{1+as}\,r_{m,n}-1}\abs{\Ai(s)}
\leq
CN^{-2/3}\exp(-s/4).
\end{equation}

If \(as>1\), equation~\eqref{eq:ek-r-uniform} gives
\[
\abs{\sqrt{1+as}\,r_{m,n}-1}
\leq
C\sqrt{as}.
\]
By \eqref{eq:global-positive-airy}, and using
\[
\frac{s^{1/2}}{1+s^{1/4}}\leq s^{1/4},
\]
we have
\begin{equation}\label{eq:positive-large-as-first}
\abs{\sqrt{1+as}\,r_{m,n}-1}\abs{\Ai(s)}
\leq
C\sqrt{a}s^{1/4}
\exp\left[-\frac23s^{3/2}\right].
\end{equation}
Since \(as>1\), one has \(a^{-1/2}<s^{1/2}\). Consequently,
\begin{align*}
\sqrt{a}s^{1/4}\exp\left[-\frac23s^{3/2}\right]
&=
a\left[
a^{-1/2}s^{1/4}
\exp\left(-\frac23s^{3/2}+\frac{s}{4}\right)
\right]\exp(-s/4)
\\
&\leq
a\left[
s^{3/4}
\exp\left(-\frac23s^{3/2}+\frac{s}{4}\right)
\right]\exp(-s/4).
\end{align*}
The bracketed function is bounded on \([0,\infty)\). It is continuous on bounded intervals, and for \(s\geq1\) its exponential is at most \(\exp(-5s/12)\). By \eqref{eq:a-comparability}, \eqref{eq:positive-large-as-first} is therefore at most
\begin{equation}\label{eq:positive-large-as}
CN^{-2/3}\exp(-s/4).
\end{equation}

Multiplying \eqref{eq:positive-small-as} or
\eqref{eq:positive-large-as} by \(\sqrt{\delta}\), and combining the result with
\eqref{eq:positive-conversion} and \eqref{eq:positive-error-one}, proves the claimed bound uniformly for every \(s\geq0\) and every
\(c_0N\leq n\leq m\leq C_0N\).
\end{proof}

The preceding estimate does not reach a lower wall that moves into the oscillatory side. The next proposition extends the comparison to $-R\leq s\leq0$ by matching the exact Laguerre differential equation to the Airy equation at the turning point. When $R$ is at most a constant multiple of $(\log N)^{1/3}$, the factor $\exp(C(1+R)^2)$ is $N^{o(1)}$; the error therefore remains $N^{-2/3+o(1)}$ uniformly on precisely the negative window used by the lower gap event.

\begin{prop}[Laguerre approximation on a logarithmic negative window]\label{prop:negative-window-airy}
Fix constants \(0<c_0<C_0<\infty\) and \(D<\infty\). For positive integers \(n\leq m\), let
\(\phi_{n,m}\), \(\mu_{n,m}\), \(\sigma_{n,m}\), and \(\delta_{n,m}\) be as in Proposition~\ref{prop:positive-airy-laguerre}. There are constants \(C<\infty\) and \(N_0\) such that, whenever
\[
N\geq N_0,\qquad
c_0N\leq n\leq m\leq C_0N,\qquad
0\leq R\leq D(\log N)^{1/3},
\]
and \(-R\leq s\leq0\),
\[
\abs{
\frac{
(-1)^n\sqrt{\sigma_{n,m}}\,
\phi_{n,m}(\mu_{n,m}+\sigma_{n,m}s)
}{
\sqrt{\delta_{n,m}}
}
-\Ai(s)
}
\leq
CN^{-2/3}(1+R)^3\exp\bigl(C(1+R)^2\bigr).
\]
In particular, the right side is \(N^{-2/3+o(1)}\) uniformly over the stated family.
\end{prop}

\begin{proof}
We use Proposition~\ref{prop:positive-airy-laguerre} only for its definitions and its positive-edge Airy approximation. The exact differential equation below is derived directly from the finite defining series of the generalized Laguerre polynomial. We also use
\cite[Lemma~2.1]{TU21}, which gives
\begin{equation}\label{eq:negative-global-airy}
\abs{\Ai(z)}
\leq
C_A
\frac{
\exp[-\frac23\operatorname{Re}(z^{3/2})]
}{
1+\abs{z}^{1/4}
},
\qquad z\in\C.
\end{equation}
For real \(z\leq0\), this yields
\[
\abs{\Ai(z)}\leq C_A(1+\abs z)^{-1/4}.
\]

Fix \(N,n,m,R\) satisfying the hypotheses, and abbreviate
\[
\phi=\phi_{n,m},
\qquad
\mu=\mu_{n,m},
\qquad
\sigma=\sigma_{n,m},
\qquad
\delta=\delta_{n,m}.
\]
Put \(\alpha=m-n\). The generalized Laguerre polynomial has the finite defining series
\begin{equation}\label{eq:laguerre-finite-series}
L_n^{(\alpha)}(x)
=
\sum_{k=0}^n
(-1)^k
\binom{n+\alpha}{n-k}
\frac{x^k}{k!}.
\end{equation}
If \(a_k\) is the coefficient of \(x^k\) in
\eqref{eq:laguerre-finite-series}, then
\begin{equation}\label{eq:laguerre-coefficient-recurrence}
(k+1)(k+\alpha+1)a_{k+1}+(n-k)a_k=0,
\qquad 0\leq k<n.
\end{equation}
The coefficient of \(x^k\) in
\[
x(L_n^{(\alpha)})''(x)
+
(\alpha+1-x)(L_n^{(\alpha)})'(x)
+
nL_n^{(\alpha)}(x)
\]
is the left side of \eqref{eq:laguerre-coefficient-recurrence} for \(0\leq k<n\), and is zero also for \(k=n\). Hence
\begin{equation}\label{eq:laguerre-equation}
x(L_n^{(\alpha)})''
+
(\alpha+1-x)(L_n^{(\alpha)})'
+
nL_n^{(\alpha)}
=
0.
\end{equation}

Define
\[
U(x)=\sqrt{x}\phi(x).
\]
The constant \(\sqrt{n!/m!}\) does not affect a homogeneous differential equation, and
\begin{equation}\label{eq:laguerre-U}
U(x)
=
\sqrt{\frac{n!}{m!}}\,
\exp(-x/2)x^{(\alpha+1)/2}L_n^{(\alpha)}(x).
\end{equation}
For
\[
h(x)=\frac{\alpha+1}{2x}-\frac12,
\]
differentiating \eqref{eq:laguerre-U} twice and using
\eqref{eq:laguerre-equation} cancels the coefficient of
\((L_n^{(\alpha)})'\). The remaining coefficient is
\[
h'(x)+h(x)^2-\frac{n}{x}
=
\frac14-\frac{2n+\alpha+1}{2x}
+\frac{\alpha^2-1}{4x^2}.
\]
Since \(2n+\alpha+1=n+m+1\), this proves the exact equation
\begin{equation}\label{eq:laguerre-schrodinger}
U''(x)=V(x)U(x),
\qquad
V(x)
=
\frac{x^2-2(n+m+1)x+(m-n)^2-1}{4x^2}.
\end{equation}

For \(x_s=\mu+\sigma s\), define
\begin{equation}\label{eq:negative-Y}
Y(s)
=
(-1)^n\sqrt{\frac{\sigma}{\delta\mu}}\,U(x_s)
=
(-1)^n\sqrt{\frac{\sigma}{\delta}}
\sqrt{\frac{x_s}{\mu}}\,
\phi(x_s).
\end{equation}
On the entire displayed \(s\)-window, \(x_s>0\) for all sufficiently large \(N\). The chain rule and \eqref{eq:laguerre-schrodinger} give
\begin{equation}\label{eq:negative-Y-ode}
Y''(s)=q(s)Y(s),
\qquad
q(s)=\sigma^2V(\mu+\sigma s).
\end{equation}

Put
\[
u=n+\frac12,
\qquad
v=m+\frac12.
\]
Uniformly over the compact aspect-ratio family, \(u,v,\mu\) are comparable to \(N\), \(\sigma\) is comparable to \(N^{1/3}\), and
\begin{equation}\label{eq:half-integer-parameters}
\mu=(\sqrt u+\sqrt v)^2,
\qquad
\sigma=(\sqrt u+\sqrt v)(u^{-1/2}+v^{-1/2})^{1/3}.
\end{equation}
Let
\[
f(x)=x^2-2(n+m+1)x+(m-n)^2-1.
\]
Since \(n+m+1=u+v\) and \(m-n=v-u\), direct substitution in
\eqref{eq:half-integer-parameters} gives
\begin{equation}\label{eq:turning-identities}
f(\mu)=-1,
\qquad
f'(\mu)=4\sqrt{uv},
\qquad
\sigma^3=\frac{\mu^2}{\sqrt{uv}}.
\end{equation}
Equations~\eqref{eq:laguerre-schrodinger},
\eqref{eq:negative-Y-ode}, and \eqref{eq:turning-identities} imply
\begin{equation}\label{eq:q-zero}
q(0)
=
-\frac{\sigma^2}{4\mu^2}
=
O(N^{-4/3}).
\end{equation}
Moreover,
\[
V'(\mu)=\frac{\sqrt{uv}}{\mu^2}+\frac1{2\mu^3},
\]
and hence
\begin{equation}\label{eq:q-prime-zero}
q'(0)
=
\sigma^3V'(\mu)
=
1+\frac1{2\mu\sqrt{uv}}
=
1+O(N^{-2}).
\end{equation}
The identity
\begin{equation}\label{eq:V-half-integer}
V(x)
=
\frac14-\frac{u+v}{2x}
+
\frac{(v-u)^2-1}{4x^2}
\end{equation}
shows by two differentiations that
\[
V''(x)=O(N^{-2})
\]
whenever \(x\) is comparable to \(N\). Since
\(\abs{s}\leq D(\log N)^{1/3}\), \(x_s\) is comparable to \(N\) uniformly for all sufficiently large \(N\). Thus
\begin{equation}\label{eq:q-second}
q''(s)=\sigma^4V''(x_s)=O(N^{-2/3}).
\end{equation}
Taylor's theorem applied between \(0\) and \(s\), together with
\eqref{eq:q-zero}, \eqref{eq:q-prime-zero}, and \eqref{eq:q-second}, yields
\begin{equation}\label{eq:q-airy-comparison}
\abs{q(s)-s}
\leq
CN^{-2/3}(1+s^2)
\end{equation}
uniformly throughout the displayed window.

Let
\[
G(s)=(-1)^n\sqrt{\frac{\sigma}{\delta}}\,
\phi(\mu+\sigma s),
\]
so that
\[
Y(s)=\sqrt{\frac{x_s}{\mu}}\,G(s),
\]
and put \(e(s)=Y(s)-\Ai(s)\).

We reconcile the scale in Proposition~\ref{prop:positive-airy-laguerre}. Put
\[
N_*=n,
\qquad
c_*=\frac12,
\qquad
C_*=1+\frac{C_0}{c_0}.
\]
Then \(N_*\) is a positive integer and
\[
c_*N_*\leq n\leq m\leq C_*N_*,
\]
because \(m/n\leq C_0/c_0\). Also \(n\geq c_0N\), so
\[
N_*^{-2/3}\leq c_0^{-2/3}N^{-2/3}.
\]
After increasing \(N_0\), Proposition~\ref{prop:positive-airy-laguerre} applied with scale \(N_*\) gives, for \(0\leq s\leq1\),
\begin{equation}\label{eq:negative-initial-c0}
\abs{G(s)-\Ai(s)}
\leq
CN^{-2/3}\exp(-s/4)
\leq
CN^{-2/3}.
\end{equation}
On this interval,
\[
\sqrt{x_s/\mu}=1+O(N^{-2/3}),
\]
and \eqref{eq:negative-global-airy} bounds \(\Ai\). Hence
\begin{equation}\label{eq:negative-e-c0}
\sup_{0\leq s\leq1}\abs{e(s)}
\leq
CN^{-2/3},
\end{equation}
and \(Y\) and \(\Ai\) are bounded uniformly on \([0,1]\).

Integrating \eqref{eq:negative-Y-ode} twice from zero to one gives
\begin{equation}\label{eq:Y-prime-zero}
Y'(0)
=
Y(1)-Y(0)
-
\int_0^1(1-t)q(t)Y(t)\dd t.
\end{equation}
The Airy equation gives
\begin{equation}\label{eq:Ai-prime-zero}
\Ai'(0)
=
\Ai(1)-\Ai(0)
-
\int_0^1(1-t)t\Ai(t)\dd t.
\end{equation}
Subtracting \eqref{eq:Ai-prime-zero} from
\eqref{eq:Y-prime-zero}, and using
\eqref{eq:q-airy-comparison}, \eqref{eq:negative-e-c0}, and the uniform boundedness just noted, gives
\begin{equation}\label{eq:negative-initial-c1}
\abs{e(0)}+\abs{e'(0)}
\leq
CN^{-2/3}.
\end{equation}

Subtracting the Airy equation from \eqref{eq:negative-Y-ode} yields
\begin{equation}\label{eq:negative-error-ode}
e''(s)
=
q(s)e(s)+[q(s)-s]\Ai(s).
\end{equation}
For \(-R\leq s\leq0\), put
\[
F(s)=\abs{e(s)}+\abs{e'(s)}.
\]
Integrating \eqref{eq:negative-error-ode} backward from zero and using
\eqref{eq:negative-initial-c1} gives
\begin{align}
F(s)
&\leq
CN^{-2/3}
+
\int_s^0[1+\abs{q(t)}]F(t)\dd t
+
\int_s^0\abs{q(t)-t}\abs{\Ai(t)}\dd t.
\label{eq:negative-gronwall-input}
\end{align}
By \eqref{eq:negative-global-airy} and
\eqref{eq:q-airy-comparison}, the final integral is at most
\begin{align}
CN^{-2/3}
\int_0^R(1+r^2)(1+r)^{-1/4}\dd r
&\leq
CN^{-2/3}(1+R)^3.
\label{eq:negative-forcing}
\end{align}
Equation~\eqref{eq:q-airy-comparison} also gives
\[
\abs{q(t)}\leq C(1+\abs t)
\]
on this window for all sufficiently large \(N\), and therefore
\begin{equation}\label{eq:negative-q-integral}
\int_{-R}^0[1+\abs{q(t)}]\dd t
\leq
C(1+R)^2.
\end{equation}
Applying the integral Gronwall inequality to
\eqref{eq:negative-gronwall-input}, with
\eqref{eq:negative-forcing} and
\eqref{eq:negative-q-integral}, proves
\begin{equation}\label{eq:negative-error-bound}
\abs{e(s)}+\abs{e'(s)}
\leq
CN^{-2/3}(1+R)^3\exp\bigl(C(1+R)^2\bigr),
\qquad -R\leq s\leq0.
\end{equation}

Finally, \eqref{eq:negative-Y} gives
\[
G(s)=\sqrt{\frac{\mu}{x_s}}\,Y(s).
\]
Uniformly on the negative window, \(\sqrt{\mu/x_s}\) is bounded and
\begin{equation}\label{eq:square-root-adjustment}
\abs{\sqrt{\mu/x_s}-1}
\leq
C\frac{\sigma\abs s}{\mu}
\leq
CN^{-2/3}\abs s.
\end{equation}
Using \(Y=\Ai+e\), the Airy estimate
\eqref{eq:negative-global-airy}, and
\eqref{eq:negative-error-bound}--\eqref{eq:square-root-adjustment}, we obtain
\[
\begin{aligned}
\abs{G(s)-\Ai(s)}
&\leq C\abs{e(s)}+CN^{-2/3}\abs s\abs{\Ai(s)}\\
&\leq CN^{-2/3}(1+R)^3\exp\bigl(C(1+R)^2\bigr).
\end{aligned}
\]
Since \(R\leq D(\log N)^{1/3}\), the logarithm of the factor multiplying \(N^{-2/3}\) is
\[
O((\log N)^{2/3})=o(\log N)
\]
uniformly over the displayed family. This proves the final assertion.
\end{proof}

Equal-level kernel comparison requires more than a one-mode $C^0$ estimate. The Christoffel--Darboux formula couples neighboring Laguerre modes through a divided difference, and the later cross-block estimates integrate from the negative wall through a positive logarithmic cutoff. The mixed-window result therefore combines the two one-sided approximations and upgrades them to uniform $C^1$ control across the entire interval.

\begin{prop}[\(C^1\) Laguerre approximation on a mixed logarithmic window]\label{prop:mixed-window-c1}
Fix constants \(0<c_0<C_0<\infty\), \(D<\infty\), and \(E<\infty\). Whenever
\[
c_0N\leq n\leq m\leq C_0N,
\]
define
\[
G_{n,m}(s)
=
(-1)^n
\sqrt{\frac{\sigma_{n,m}}{\delta_{n,m}}}\,
\phi_{n,m}(\mu_{n,m}+\sigma_{n,m}s),
\]
using the notation of Proposition~\ref{prop:positive-airy-laguerre}, and let the prime denote differentiation with respect to \(s\). Uniformly for
\[
0\leq R\leq D(\log N)^{1/3},
\qquad
-R\leq s\leq E\log N,
\]
one has
\[
\abs{G_{n,m}(s)-\Ai(s)}
+
\abs{G_{n,m}'(s)-\Ai'(s)}
=
N^{-2/3+o(1)},
\]
where \(o(1)\to0\) uniformly over all displayed parameters.
\end{prop}

\begin{proof}
We use Proposition~\ref{prop:positive-airy-laguerre} through its definitions and positive-edge \(C^0\) approximation, and Proposition~\ref{prop:negative-window-airy} through its negative-window \(C^0\) approximation. The exact Laguerre differential equation is derived coefficientwise below. The only external Airy envelope used is
\cite[Lemma~2.1]{TU21}, namely
\begin{equation}\label{eq:mixed-global-airy}
\abs{\Ai(z)}
\leq
C_A
\frac{
\exp[-\frac23\operatorname{Re}(z^{3/2})]
}{
1+\abs{z}^{1/4}
},
\qquad z\in\C.
\end{equation}
For real \(z\), the right side is at most \(C_A\), so \(\Ai\) is bounded on the real axis. Increasing \(D\) and \(E\) to nonnegative finite constants, if necessary, only enlarges the interval to be proved. We may therefore assume \(D,E\geq0\).

Fix \(N,n,m,R\) satisfying the hypotheses and take \(N\) sufficiently large. Abbreviate
\[
\phi=\phi_{n,m},
\qquad
\mu=\mu_{n,m},
\qquad
\sigma=\sigma_{n,m},
\qquad
\delta=\delta_{n,m},
\qquad
\alpha=m-n.
\]
The finite defining series is
\[
L_n^{(\alpha)}(x)
=
\sum_{k=0}^n
(-1)^k
\binom{n+\alpha}{n-k}\frac{x^k}{k!}.
\]
If \(a_k\) is its coefficient of \(x^k\), direct cancellation of factorials gives
\[
(k+1)(k+\alpha+1)a_{k+1}+(n-k)a_k=0,
\qquad 0\leq k<n.
\]
Coefficient comparison therefore yields
\[
x(L_n^{(\alpha)})''
+
(\alpha+1-x)(L_n^{(\alpha)})'
+
nL_n^{(\alpha)}
=
0.
\]

Define
\[
U(x)=\sqrt{x}\phi(x)
=
\sqrt{\frac{n!}{m!}}\,
\exp(-x/2)x^{(\alpha+1)/2}L_n^{(\alpha)}(x).
\]
For
\[
h(x)=\frac{\alpha+1}{2x}-\frac12,
\]
two differentiations cancel the coefficient of
\((L_n^{(\alpha)})'\). The remaining coefficient is
\[
h'(x)+h(x)^2-\frac nx
=
\frac14-\frac{2n+\alpha+1}{2x}
+
\frac{\alpha^2-1}{4x^2}.
\]
Since \(2n+\alpha+1=n+m+1\),
\begin{equation}\label{eq:mixed-U-ode}
U''(x)=V(x)U(x),
\qquad
V(x)
=
\frac{x^2-2(n+m+1)x+(m-n)^2-1}{4x^2}.
\end{equation}

For \(x_s=\mu+\sigma s\), define
\[
G(s)
=
(-1)^n\sqrt{\frac{\sigma}{\delta}}\phi(x_s)
\]
and
\begin{equation}\label{eq:mixed-Y}
Y(s)
=
(-1)^n
\sqrt{\frac{\sigma}{\delta\mu}}U(x_s)
=
\sqrt{\frac{x_s}{\mu}}\,G(s).
\end{equation}
Let
\[
u=n+\frac12,
\qquad
v=m+\frac12.
\]
The definitions give
\[
\mu=(\sqrt u+\sqrt v)^2,
\qquad
\sigma=(\sqrt u+\sqrt v)(u^{-1/2}+v^{-1/2})^{1/3}.
\]
Uniformly over the compact aspect-ratio family, \(u,v,\mu\) are comparable to \(N\), \(\sigma\) is comparable to \(N^{1/3}\), and
\begin{equation}\label{eq:mixed-a}
a:=\frac{\sigma}{\mu}=O(N^{-2/3}).
\end{equation}

On the enlarged interval
\[
I_N=[-R,E\log N+1],
\]
one has \(\abs{as}\leq1/2\) for all sufficiently large \(N\). Hence \(x_s>0\), and \eqref{eq:mixed-Y} is defined. The chain rule and
\eqref{eq:mixed-U-ode} give
\begin{equation}\label{eq:mixed-Y-ode}
Y''(s)=q(s)Y(s),
\qquad
q(s)=\sigma^2V(\mu+\sigma s).
\end{equation}

Let
\[
f(x)=x^2-2(n+m+1)x+(m-n)^2-1.
\]
Since \(n+m+1=u+v\) and \(m-n=v-u\), direct substitution gives
\[
f(\mu)=-1,
\qquad
f'(\mu)=4\sqrt{uv},
\qquad
\sigma^3=\frac{\mu^2}{\sqrt{uv}}.
\]
Equations~\eqref{eq:mixed-U-ode} and \eqref{eq:mixed-Y-ode} imply
\[
q(0)
=
-\frac{\sigma^2}{4\mu^2}
=
O(N^{-4/3}),
\]
while
\[
V'(\mu)
=
\frac{\sqrt{uv}}{\mu^2}
+
\frac1{2\mu^3}
\]
gives
\[
q'(0)
=
\sigma^3V'(\mu)
=
1+\frac1{2\mu\sqrt{uv}}
=
1+O(N^{-2}).
\]
The identity
\[
V(x)
=
\frac14-\frac{u+v}{2x}
+
\frac{(v-u)^2-1}{4x^2}
\]
shows by two differentiations that \(V''(x)=O(N^{-2})\) whenever \(x\) is comparable to \(N\). This applies to \(x=x_s\) throughout \(I_N\). Therefore
\[
q''(s)=\sigma^4V''(x_s)=O(N^{-2/3})
\]
uniformly on \(I_N\). Taylor's theorem between \(0\) and \(s\) yields
\begin{equation}\label{eq:mixed-q-error}
\abs{q(s)-s}
\leq
CN^{-2/3}(1+s^2),
\qquad s\in I_N.
\end{equation}
In particular,
\[
\abs{q(s)}\leq C(1+\abs s)
\]
on \(I_N\) for all sufficiently large \(N\).

On \(-R\leq s\leq0\), Proposition~\ref{prop:negative-window-airy} gives
\begin{equation}\label{eq:mixed-negative-c0}
\abs{G(s)-\Ai(s)}
\leq
CN^{-2/3}(1+R)^3\exp\bigl(C(1+R)^2\bigr)
=
N^{-2/3+o(1)}.
\end{equation}
The last equality is uniform for \(R\leq D(\log N)^{1/3}\).

For the positive interval, put
\[
N_*=n,\qquad c_*=\frac12,\qquad C_*=1+\frac{C_0}{c_0}.
\]
Then
\[
c_*N_*\leq n\leq m\leq C_*N_*,
\qquad
N_*^{-2/3}\leq c_0^{-2/3}N^{-2/3}.
\]
Applying Proposition~\ref{prop:positive-airy-laguerre} with scale \(N_*\) and dividing its estimate by \(\sqrt{\delta}\) gives
\begin{equation}\label{eq:mixed-positive-c0}
\abs{G(s)-\Ai(s)}
\leq
CN^{-2/3}\exp(-s/4),
\qquad
0\leq s\leq E\log N+1.
\end{equation}

Put \(e=Y-\Ai\). Since \(Y=\sqrt{1+as}\,G\), equations
\eqref{eq:mixed-global-airy}, \eqref{eq:mixed-a},
\eqref{eq:mixed-negative-c0}, and \eqref{eq:mixed-positive-c0} imply
\begin{equation}\label{eq:mixed-e-c0}
\sup_{s\in I_N}\abs{e(s)}
=
N^{-2/3+o(1)}.
\end{equation}
Indeed,
\[
\abs{e(s)}
\leq
\sqrt{1+as}\abs{G(s)-\Ai(s)}
+
\abs{\sqrt{1+as}-1}\abs{\Ai(s)}.
\]
The first term is controlled by \eqref{eq:mixed-negative-c0} or
\eqref{eq:mixed-positive-c0}, while the second is at most
\[
CN^{-2/3}\abs s
\]
by \eqref{eq:mixed-global-airy} and \eqref{eq:mixed-a}. Every power of \(1+\abs s\) on \(I_N\) is \(N^{o(1)}\).

Subtracting the Airy equation from \eqref{eq:mixed-Y-ode} gives
\begin{equation}\label{eq:mixed-e-ode}
e''(s)=q(s)e(s)+[q(s)-s]\Ai(s).
\end{equation}
Equations~\eqref{eq:mixed-global-airy},
\eqref{eq:mixed-q-error}, \eqref{eq:mixed-e-c0}, and
\(\abs{q(s)}\leq C(1+\abs s)\) now yield
\begin{equation}\label{eq:mixed-e-second}
\sup_{s\in I_N}\abs{e''(s)}
=
N^{-2/3+o(1)}.
\end{equation}

For every \(-R\leq s\leq E\log N\), the interval \([s,s+1]\) lies in \(I_N\). Twice integrating \(e''\) on this unit interval gives the exact identity
\begin{equation}\label{eq:mixed-unit-identity}
e'(s)
=
e(s+1)-e(s)
-
\int_s^{s+1}(s+1-t)e''(t)\dd t.
\end{equation}
Equations~\eqref{eq:mixed-e-c0},
\eqref{eq:mixed-e-second}, and
\eqref{eq:mixed-unit-identity} prove
\begin{equation}\label{eq:mixed-e-prime}
\sup_{-R\leq s\leq E\log N}\abs{e'(s)}
=
N^{-2/3+o(1)}.
\end{equation}

It remains to remove the square-root adjustment. The same unit-interval identity applied to the Airy equation gives
\[
\Ai'(s)
=
\Ai(s+1)-\Ai(s)
-
\int_s^{s+1}(s+1-t)t\Ai(t)\dd t.
\]
By \eqref{eq:mixed-global-airy},
\begin{equation}\label{eq:mixed-Ai-prime}
\abs{\Ai'(s)}
\leq
C(1+\abs s),
\qquad s\in I_N.
\end{equation}
Differentiating \eqref{eq:mixed-Y} gives
\begin{equation}\label{eq:mixed-G-prime}
G'(s)
=
\frac{Y'(s)}{\sqrt{1+as}}
-
\frac{aY(s)}{2(1+as)^{3/2}}.
\end{equation}
On \(I_N\),
\[
\abs{\frac1{\sqrt{1+as}}-1}
\leq
Ca\abs s,
\]
while \eqref{eq:mixed-global-airy} and
\eqref{eq:mixed-e-c0} give \(\abs{Y(s)}\leq C\) for all sufficiently large \(N\). Combining
\eqref{eq:mixed-a}, \eqref{eq:mixed-e-prime},
\eqref{eq:mixed-Ai-prime}, and \eqref{eq:mixed-G-prime} yields
\[
\sup_{-R\leq s\leq E\log N}
\abs{G'(s)-\Ai'(s)}
=
N^{-2/3+o(1)}.
\]
Equations~\eqref{eq:mixed-negative-c0} and
\eqref{eq:mixed-positive-c0} give the same estimate for
\(\abs{G(s)-\Ai(s)}\). This proves the proposition, with every \(o(1)\) uniform over \(n,m,R\), and \(s\) in the displayed family.
\end{proof}

Uniform kernel convergence alone is insufficient for a lower gap probability: as the wall recedes, the top eigenvalue of the compressed Airy kernel approaches one, and the relevant diagonal block must be inverted. The next proposition quantifies the remaining spectral gap. On a wall of size $R=O((\log N)^{1/3})$, its inverse bound is $\exp(O(R^{3/2}))=N^{o(1)}$, which is the scale that can be combined with the later cross-block decay.

\begin{prop}[Airy-kernel spectral gap at a receding edge]\label{prop:airy-kernel-gap}
For \(s\in\R\), let \(K_{\Ai,s}\) be the integral operator on
\(L^2((s,\infty))\) with kernel
\[
K_{\Ai}(x,y)
=
\int_0^\infty\Ai(x+u)\Ai(y+u)\dd u,
\]
and let \(\lambda_0(s)\) be its largest eigenvalue. As \(s\to-\infty\),
\[
1-\lambda_0(s)
=
\sqrt{\pi}\,2^{9/4}(-s)^{3/4}
\exp\left[-\frac23\sqrt2\,(-s)^{3/2}\right](1+o(1)).
\]
In particular, \(K_{\Ai,s}\) is a positive self-adjoint contraction with
\(\lambda_0(s)<1\). There are finite constants \(C\) and \(R_0\) such that, for every \(R\geq R_0\),
\[
\norm{(I-K_{\Ai,-R})^{-1}}_{\op}
\leq
\exp(CR^{3/2}).
\]
\end{prop}

\begin{proof}
Write \(K_s=K_{\Ai,s}\). The Airy-kernel spectral expansion
\cite[Corollary~1.2]{Bothner2016} states that, for every fixed nonnegative integer \(i\), as \(s\to-\infty\),
\begin{equation}\label{eq:bothner-airy-spectrum}
1-\lambda_i(s)
=
\frac{\sqrt{\pi}}{i!}
2^{(7/2)i+9/4}
t^{i+1/2}
\exp\left[-\frac23\sqrt2\,t\right]
(1+o(1)),
\qquad
t=(-s)^{3/2}.
\end{equation}

We first check that the operator in the cited result is exactly \(K_s\). For \(x\neq y\), the Airy differential equation gives
\begin{align*}
\frac{\dd}{\dd u}
\left[
\begin{aligned}
&\Ai'(x+u)\Ai(y+u)\\
&\quad{}-\Ai(x+u)\Ai'(y+u)
\end{aligned}
\right]
&=
(x-y)\Ai(x+u)\Ai(y+u).
\end{align*}
Integrating from \(u=0\) to infinity and using the decay of \(\Ai\) and \(\Ai'\) at positive infinity gives
\[
\int_0^\infty\Ai(x+u)\Ai(y+u)\dd u
=
\frac{
\Ai(x)\Ai'(y)-\Ai'(x)\Ai(y)
}{
x-y
}.
\]
The equality at \(x=y\) follows by continuity. This is the same kernel on the same space \(L^2((s,\infty))\). Taking \(i=0\) in \eqref{eq:bothner-airy-spectrum} and substituting \(t=(-s)^{3/2}\) proves the stated asymptotic, including its prefactor and exponential constant.

For completeness, positivity follows directly from the integral-kernel factorization. For every compactly supported \(f\in L^2((s,\infty))\), Fubini's theorem gives
\[
\langle f,K_sf\rangle
=
\int_0^\infty
\abs{
\int_s^\infty\Ai(x+u)f(x)\dd x
}^2
\dd u
\geq0.
\]
The cited result states that \(K_s\) is trace class, hence bounded and compact. Density extends the displayed nonnegative quadratic-form identity to every \(f\in L^2((s,\infty))\). The kernel is real and symmetric, so \(K_s\) is positive and self-adjoint. Consequently,
\[
\norm{K_s}_{\op}=\lambda_0(s).
\]

Set
\[
A=\sqrt{\pi}\,2^{9/4}>0,
\qquad
c=\frac23\sqrt2>0.
\]
The asymptotic implies that there is \(S_0<0\) such that
\(\lambda_0(r)<1\) whenever \(r\leq S_0\). Fix \(s\in\R\), and choose
\[
s_*<\min(s,S_0).
\]
Let \(E\) be extension by zero from \(L^2((s,\infty))\) into
\(L^2((s_*,\infty))\). Restriction of the same kernel to the smaller interval gives the exact compression identity
\[
K_s=E^*K_{\Ai,s_*}E.
\]
For every unit vector \(f\in L^2((s,\infty))\), positivity of \(K_{\Ai,s_*}\) and the fact that \(E\) is an isometry give
\[
\langle f,K_sf\rangle
=
\langle Ef,K_{\Ai,s_*}Ef\rangle
\leq
\norm{K_{\Ai,s_*}}_{\op}
=
\lambda_0(s_*)
<
1.
\]
Taking the supremum over unit \(f\) yields
\[
0\leq
\lambda_0(s)
=
\norm{K_s}_{\op}
\leq
\lambda_0(s_*)
<
1.
\]
Thus \(K_s\) is a positive self-adjoint contraction with a strict spectral gap for every real \(s\). This strictness follows from the asymptotic and compression monotonicity; it does not use the invalid inference that every proper compression of a projection has norm below one.

Because the spectrum of \(K_s\) is contained in \([0,\lambda_0(s)]\), the spectrum of \(I-K_s\) is contained in \([1-\lambda_0(s),1]\). Hence \(I-K_s\) is invertible and
\begin{equation}\label{eq:airy-resolvent-norm}
\norm{(I-K_s)^{-1}}_{\op}
=
[1-\lambda_0(s)]^{-1}.
\end{equation}
Choose \(R_0\geq1\) so large that the factor \(1+o(1)\) in the spectral asymptotic, with \(s=-R\), is at least \(1/2\) for every \(R\geq R_0\). Then
\[
1-\lambda_0(-R)
\geq
\frac A2R^{3/4}\exp(-cR^{3/2}),
\]
and \eqref{eq:airy-resolvent-norm} gives
\[
\norm{(I-K_{\Ai,-R})^{-1}}_{\op}
\leq
\frac2A R^{-3/4}\exp(cR^{3/2}).
\]
Let
\[
C
=
c+\frac{\max(0,\log(2/A))}{R_0^{3/2}}.
\]
Since \(R\geq R_0\geq1\), the last display is at most
\(\exp(CR^{3/2})\).
\end{proof}

\subsection{Gaussian logarithmic tails}

The remaining marginal input is a normalization and indexing transfer. The Laguerre statements are formulated using the ordered pair of positive dimensions, whereas the common-array observable is indexed by the physical rectangle and divided by $N$. The following proposition identifies the same nonzero spectrum after transposition when necessary and carries the two logarithmic exponents to the project soft-edge normalization.

\begin{prop}[One-time complex Gaussian log-window tails]\label{prop:complex-gaussian-one-time}
Let \(\gamma>0\), and let \((M_N)_{N\geq1}\) be positive integers such that
\[
\frac{M_N}{N}\longrightarrow\gamma.
\]
For each \(N\), let \(X^{(N)}\) be an \(M_N\times N\) matrix with independent centered circular complex Gaussian entries of second absolute moment one. Define
\[
Q_N=N^{-1}X^{(N)}(X^{(N)})^*,
\qquad
\lambda_1^{(N)}=\lambda_{\max}(Q_N),
\]
\[
\mu_{+,N}
=
\frac{(\sqrt{M_N}+\sqrt N)^2}{N},
\qquad
\sigma_{+,N}
=
\frac{\sqrt{M_N}+\sqrt N}{N}
(M_N^{-1/2}+N^{-1/2})^{1/3},
\]
and
\[
\chi_N^+
=
\frac{\lambda_1^{(N)}-\mu_{+,N}}{\sigma_{+,N}}.
\]
Then, for every fixed \(u>0\) and \(v>0\),
\[
\Prob\left(
\chi_N^+\geq u(\log N)^{2/3}
\right)
=
\exp\left(
-\frac43u^{3/2}\log N+o(\log N)
\right),
\]
and
\[
\Prob\left(
\chi_N^+\leq-v(\log N)^{1/3}
\right)
=
\exp\left(
-\frac1{12}v^3\log N+o(\log N)
\right)
\]
as \(N\to\infty\).
\end{prop}

\begin{proof}
Put
\[
d_N=\min(M_N,N),
\qquad
e_N=\max(M_N,N).
\]
Since \(M_N/N\to\gamma\in(0,\infty)\),
\[
d_N\to\infty,
\qquad
\frac{e_N}{d_N}=O(1),
\qquad
\frac{\log d_N}{\log N}\to1.
\]

Let
\[
L_N
=
\lambda_{\max}\left(
X^{(N)}(X^{(N)})^*
\right).
\]
The nonzero eigenvalues of \(X^{(N)}(X^{(N)})^*\) and
\((X^{(N)})^*X^{(N)}\) coincide. If \(M_N\geq N\), then \(L_N\) has the law of the largest eigenvalue of \(Y_N^*Y_N\) for an \(e_N\times d_N\) standard complex Gaussian matrix \(Y_N\). If \(M_N<N\), the same conclusion holds by taking \(Y_N\) to have the law of the conjugate transpose of \(X^{(N)}\). Thus \(L_N\) is distributed as the largest point of the \(\beta=2\) Laguerre ensemble with parameters \((d_N,e_N)\).

The unnormalised Laguerre center and scale are
\[
(\sqrt{e_N}+\sqrt{d_N})^2
=
(\sqrt{M_N}+\sqrt N)^2
\]
and
\begin{align*}
(\sqrt{e_Nd_N})^{-1/3}
(\sqrt{e_N}+\sqrt{d_N})^{4/3}
&=
(\sqrt{M_N}+\sqrt N)
(M_N^{-1/2}+N^{-1/2})^{1/3}.
\end{align*}
Therefore the latter scale is
\[
b_N=N\sigma_{+,N},
\]
and the center is
\[
a_N=N\mu_{+,N}.
\]
Since \(\lambda_1^{(N)}=L_N/N\), the event
\(\{\chi_N^+\geq x\}\) is exactly
\[
\{L_N\geq a_N+xb_N\},
\]
while \(\{\chi_N^+\leq-x\}\) is exactly
\[
\{L_N\leq a_N-xb_N\}.
\]

We next transfer the \(\beta=2\) conclusion of Proposition~\ref{prop:laguerre-log-tails} to the possibly nonconsecutive pair sequence \((d_N,e_N)\). Suppose that one of the required logarithmic tail conclusions failed along an infinite sequence of indices \(N_j\). Since \(d_{N_j}\to\infty\), pass to a subsequence on which \(d_{N_j}\) is strictly increasing. Choose \(M_0>1\) such that
\[
\frac{e_{N_j}}{d_{N_j}}\leq M_0
\]
for every retained \(j\), after discarding finitely many initial indices. Define an integer sequence \((m_n)\) by
\[
m_{d_{N_j}}=e_{N_j}
\]
at the retained indices and \(m_n=n\) at every other index. Then
\[
m_n>n-1,
\qquad
\frac{m_n}{n}\leq M_0
\]
for every \(n\). Proposition~\ref{prop:laguerre-log-tails} applies to \((m_n)\) and gives the required asymptotic along the retained subsequence, contradicting the assumed failure. Hence both \(\beta=2\) Laguerre tail asymptotics hold along the full pair sequence \((d_N,e_N)\), with \(\log d_N\) as logarithmic variable.

For the right tail, set
\[
u_N
=
u\left(\frac{\log N}{\log d_N}\right)^{2/3}.
\]
Then \(u_N\to u\) and
\[
u(\log N)^{2/3}
=
u_N(\log d_N)^{2/3}.
\]
For any \(\eta\in(0,u)\), all sufficiently large \(N\) satisfy
\[
u-\eta\leq u_N\leq u+\eta.
\]
Monotonicity gives
\begin{align*}
&\Prob\left(
L_N\geq a_N+(u+\eta)b_N(\log d_N)^{2/3}
\right)
\\
&\qquad\leq
\Prob\left(
\chi_N^+\geq u(\log N)^{2/3}
\right)
\\
&\qquad\leq
\Prob\left(
L_N\geq a_N+(u-\eta)b_N(\log d_N)^{2/3}
\right).
\end{align*}
Taking logarithms, dividing by \(\log N\), using
\(\log d_N/\log N\to1\), and applying
Proposition~\ref{prop:laguerre-log-tails} at amplitudes \(u+\eta\) and \(u-\eta\) yields lower and upper bounds with exponents
\[
-\frac43(u+\eta)^{3/2}
\quad\text{and}\quad
-\frac43(u-\eta)^{3/2}.
\]
Letting \(\eta\downarrow0\) proves the right-tail exponent.

For the left tail, set
\[
v_N
=
v\left(\frac{\log N}{\log d_N}\right)^{1/3}.
\]
Then \(v_N\to v\) and
\[
v(\log N)^{1/3}
=
v_N(\log d_N)^{1/3}.
\]
For any \(\eta\in(0,v)\) and all sufficiently large \(N\),
\begin{align*}
&\Prob\left(
L_N\leq a_N-(v+\eta)b_N(\log d_N)^{1/3}
\right)
\\
&\qquad\leq
\Prob\left(
\chi_N^+\leq-v(\log N)^{1/3}
\right)
\\
&\qquad\leq
\Prob\left(
L_N\leq a_N-(v-\eta)b_N(\log d_N)^{1/3}
\right).
\end{align*}
After taking logarithms, dividing by \(\log N\), applying
Proposition~\ref{prop:laguerre-log-tails} at amplitudes \(v+\eta\) and \(v-\eta\), and letting \(\eta\downarrow0\), we obtain the stated left-tail exponent.

The two logarithmic exponent limits are equivalent to the displayed exponential forms with \(o(\log N)\) in the exponent.
\end{proof}

The logarithmic limits also identify a one-time regime in which the marginal probability is smaller than $N^{-1/3}$ by a fixed power. The thresholds $A_2$ and $B_2$ are exactly the amplitudes at which the right- and left-tail exponents cross $1/3$. The resulting corollary is a supercritical marginal estimate only; it does not assert a two-time conclusion.

\begin{cor}[Complex Gaussian supercritical margins]\label{cor:gaussian-supercritical-margins}
Let \(\gamma>0\), let \(M_N/N\to\gamma\), and let \(G_N\) be an \(M_N\times N\) matrix with independent centered circular complex Gaussian entries of second absolute moment one. Define
\[
Q_N^G=N^{-1}G_NG_N^*,
\qquad
\lambda_{1,G}^{(N)}=\lambda_{\max}(Q_N^G),
\]
\[
\mu_{+,N}
=
\frac{(\sqrt{M_N}+\sqrt N)^2}{N},
\qquad
\sigma_{+,N}
=
\frac{\sqrt{M_N}+\sqrt N}{N}
(M_N^{-1/2}+N^{-1/2})^{1/3},
\]
\[
\chi_{N,G}^+
=
\frac{\lambda_{1,G}^{(N)}-\mu_{+,N}}{\sigma_{+,N}},
\qquad
r_N=(\log N)^{2/3},
\qquad
s_N=(\log N)^{1/3},
\]
and
\[
A_2=\left(\frac14\right)^{2/3},
\qquad
B_2=4^{1/3}.
\]
Then:

\begin{enumerate}
\item For every \(q>A_2\), there are
\(\eta_q^{G,+}>1/3\) and \(N_q^{G,+}\) such that
\[
\Prob(\chi_{N,G}^+\geq qr_N)
\leq
N^{-\eta_q^{G,+}},
\qquad
N\geq N_q^{G,+}.
\]

\item For every \(q>B_2\), there are
\(\eta_q^{G,-}>1/3\) and \(N_q^{G,-}\) such that
\[
\Prob(\chi_{N,G}^+\leq-qs_N)
\leq
N^{-\eta_q^{G,-}},
\qquad
N\geq N_q^{G,-}.
\]
\end{enumerate}
\end{cor}

\begin{proof}
We first prove the left-tail assertion. Let \(q>B_2\). Proposition~\ref{prop:complex-gaussian-one-time} gives
\[
\lim_{N\to\infty}
\frac1{\log N}
\log\Prob(\chi_{N,G}^+\leq-qs_N)
=
-\frac1{12}q^3.
\]
Since
\[
B_2^3=4
\]
and \(q>B_2\), the exponent \(q^3/12\) is strictly larger than \(1/3\). Choose
\[
\frac13<\eta_q^{G,-}<\frac1{12}q^3.
\]
The logarithmic limit implies
\[
\Prob(\chi_{N,G}^+\leq-qs_N)
\leq
N^{-\eta_q^{G,-}}
\]
for all sufficiently large \(N\).

For the right tail, Proposition~\ref{prop:complex-gaussian-one-time} gives
\[
\lim_{N\to\infty}
\frac1{\log N}
\log\Prob(\chi_{N,G}^+\geq qr_N)
=
-\frac43q^{3/2}.
\]
Since \(q>A_2\) and
\[
\frac43A_2^{3/2}
=
\frac13,
\]
choose \(\eta_q^{G,+}\) strictly between \(1/3\) and
\((4/3)q^{3/2}\). The logarithmic limit gives the desired polynomial bound for all sufficiently large \(N\).
\end{proof}

The module now supplies the two distinct scales used next. The one-time tail propositions identify the polynomially small marginals, while the uniform logarithmic-window estimates control the Laguerre modes and equal-level inverses at the corresponding moving walls. In the headline theorem's regime $\max(\kappa_+,\kappa_-)<1/3$, these uniform errors can be summed over the relevant modes and compared with the vanishing product of the marginals. Corollary~\ref{cor:gaussian-supercritical-margins} remains a separate one-time exclusion margin.

\subsection{Two-time decorrelation}

The headline statement is Theorem~\ref{thm:nested-wishart-decorrelation} in the front matter; it is not repeated here. Its pair-probability conclusions are one-sided product-scale upper bounds. No complementary lower joint bound is asserted. This distinction is essential because each marginal probability decays polynomially. An additive $o(1)$ estimate for the difference between the joint probability and the product can therefore be much larger than the product itself. The proof must instead make every cross-level correction negligible relative to that vanishing product, uniformly over the short macroscopic window, the supercritical separation $m-n\geq N^{2/3+\epsilon}$, and all bounded deterministic shifts in the theorem. The restriction $0<\epsilon<1/3-\max(\kappa_+,\kappa_-)$ is the exponent budget: it leaves a positive power between the finite-$N$ Airy-approximation scale and the combined rarity-and-separation costs used in the cross-block estimates.

\begin{proof}
We use Proposition~\ref{prop:complex-gaussian-one-time} for the one-time logarithmic tail exponents. The exact two-level Laguerre kernel is Theorem~\ref{thm:discrete-extended-airy}, and the finite-intensity determinantal reductions are Lemma~\ref{lem:two-level-reductions}. Propositions~\ref{prop:positive-airy-laguerre} and
\ref{prop:mixed-window-c1} supply the positive-edge and mixed-window Laguerre-to-Airy estimates. Proposition~\ref{prop:airy-kernel-gap} supplies the Airy-kernel contraction and receding-edge spectral gap. Lemma~\ref{lem:airy-envelopes} supplies the positive and negative Airy envelopes.

The proof separates three tasks. It begins by fixing the marginal rarity scale and converting path separation into damping of the exact mode coefficients. It then obtains equal-level inverse bounds and cross-block nuclear estimates uniformly on the moving walls. Finally, it applies different determinantal reductions to upper occupancies and lower gaps. The upper argument is one-way because an event indicator is bounded by an occupancy count; the lower argument is one-way because the Schur-complement determinant is bounded from above. The wall-regularity clause and the final change of normalization are then proved as separate uniformity statements.

\noindent\textbf{Step 1. Geometry and marginal probabilities.}
The short macroscopic window and the level separation play different roles. The window keeps all four sorted dimensions in one compact aspect-ratio set and makes the total shell size a controlled fraction of $N$. The separation forces the coefficient-damping parameter $\rho$ to be at least a positive multiple of $N^{-1/3+\epsilon}$. Bounded shifts remain $O(1)$ perturbations of the logarithmic walls throughout these comparisons.

Choose constants \(0<c_0<C_0\) such that, for all sufficiently large \(N\) and every \(j\in[N,2N]\),
\[
c_0N\leq p_j\leq q_j\leq C_0N.
\]
Choose \(\eta_0>0\) so small that
\[
(K+1)\eta_0N\leq\frac{c_0N}{4}.
\]
For \(n<m\) in the statement, coordinatewise monotonicity of the physical dimensions implies \(p\leq P\) and \(q\leq Q\) after sorting. Moreover,
\[
H=A+B=(m-n)+(M_m-M_n).
\]
Therefore
\begin{equation}\label{eq:two-time-rho}
N^{2/3+\epsilon}
\leq H\leq\frac{c_0N}{4},
\qquad
\rho
=
\frac12\left(\frac Ap+\frac Bq\right)
\geq
\frac{H}{2C_0N}
\geq
cN^{-1/3+\epsilon}.
\end{equation}

If the same physical dimension is the larger at both endpoints, the pair has fixed orientation in Theorem~\ref{thm:discrete-extended-airy}. If the physical dimensions cross, monotonicity gives
\[
p<q\leq P<Q,
\]
which is the orientation-reversal geometry. Equality at a square endpoint is obtained from either adjacent orientation because all finite sums and factorial coefficients agree at equality. Since \(H\geq2\) for large \(N\), Theorem~\ref{thm:discrete-extended-airy} applies and gives precisely the four blocks stated in the theorem.

Put
\[
e_j=(\sqrt{M_j}+\sqrt j)^2
\]
and
\[
d_j
=
(\sqrt{M_j}+\sqrt j)
(M_j^{-1/2}+j^{-1/2})^{1/3}.
\]
Direct Taylor expansion on the compact aspect-ratio set gives
\begin{equation}\label{eq:two-time-normalization-comparison}
\mu_j-e_j=O(1),
\qquad
\frac{\sigma_j}{d_j}=1+O(N^{-1})
\end{equation}
uniformly on \([N,2N]\). The statistic in Proposition~\ref{prop:complex-gaussian-one-time} is exactly
\[
\frac{\Lambda_j-e_j}{d_j}.
\]
For any fixed \(\delta>0\), every bounded additive shift at either logarithmic scale, together with the \(o(1)\) coordinate change in
\eqref{eq:two-time-normalization-comparison}, is squeezed for large \(N\) between thresholds with amplitudes \(a-\delta\) and \(a+\delta\), or \(b-\delta\) and \(b+\delta\). Applying Proposition~\ref{prop:complex-gaussian-one-time} and then letting \(\delta\downarrow0\) yields
\begin{equation}\label{eq:two-time-marginals}
u_j^+=N^{-\kappa_++o(1)},
\qquad
u_j^-=N^{-\kappa_-+o(1)}.
\end{equation}
The errors are uniform for
\(j\in[N,(1+\eta_0)N]\) and bounded shifts. Otherwise, a violating sequence would have \(j/N\) bounded and \(\log j/\log N\to1\), and the same squeeze would contradict the one-time asymptotics at index \(j\). This proves the marginal assertions for both normalizations.

These marginal estimates set the reference scale for the rest of the proof. Their products vanish polynomially, so none of the forthcoming trace or determinant corrections may be assessed only in absolute terms. Each must carry an additional negative power beyond the corresponding marginal product, or be bounded above by $1+o(1)$ as a multiplicative correction.

\noindent\textbf{Step 2. Orthonormality and coefficient damping.}
In the spectral gauge of Theorem~\ref{thm:discrete-extended-airy}, cross-level dependence is diagonalized by matched Laguerre modes. The coefficients $d_\ell$ are the normalized mode-survival factors produced in Proposition~\ref{prop:airy-time-coefficient}. Positive indices describe modes present at both levels, while nonpositive indices describe the free-transfer complement. The product formulas below turn the lower bound on $\rho$ into exponential damping in the mode index.

For a nonnegative integer \(r\), put
\[
\mathsf L_k=\phi_{k,k+r}.
\]
The finite series
\[
L_k^{(r)}(x)
=
\sum_{j=0}^k
(-1)^j
\frac{(k+r)!}{(k-j)!(r+j)!j!}x^j
\]
implies the Rodrigues identity
\begin{equation}\label{eq:airy-laguerre-rodrigues}
\exp(-x)x^rL_k^{(r)}(x)
=
\frac1{k!}
\left(\frac{\dd}{\dd x}\right)^k
\left[
\exp(-x)x^{k+r}
\right].
\end{equation}
Integrating \eqref{eq:airy-laguerre-rodrigues} by parts \(k\) times against \(L_\ell^{(r)}\) proves orthogonality when \(\ell<k\). Every boundary term vanishes because \(r\) is a nonnegative integer and the remaining factors are polynomials times \(\exp(-x)\). Taking \(\ell=k\) and using the leading coefficient \((-1)^k/k!\) gives
\[
\int_0^\infty
\exp(-x)x^rL_k^{(r)}(x)^2\dd x
=
\frac{(k+r)!}{k!}.
\]
Consequently,
\begin{equation}\label{eq:laguerre-orthonormality}
\int_0^\infty
\mathsf L_k(x)\mathsf L_\ell(x)\dd x
=
\ind_{\{k=\ell\}}.
\end{equation}
Thus the equal-level finite sums are orthogonal projections. Their restrictions are positive self-adjoint contractions.

Direct cancellation in the coefficient gives, for \(1\leq\ell\leq p\),
\begin{equation}\label{eq:d-positive-product}
d_\ell^2
=
\prod_{h=0}^{\ell-1}
\frac{(p-h)(q-h)}{(P-h)(Q-h)}.
\end{equation}
For \(1\leq\ell\leq p/2\), compact comparability and
\[
\log(1+x)\geq\frac{x}{1+x}
\]
imply
\[
-\log d_\ell
\geq
c\frac HN\ell
\geq
c\rho\ell.
\]
Monotonicity of \(d_\ell\) and the same bound at \(\lfloor p/2\rfloor\) give
\[
d_\ell\leq\exp(-cH),
\qquad
p/2<\ell\leq p.
\]

For \(v\geq0\), cancellation gives
\[
d_{-v}^{-2}
=
\prod_{h=1}^v
\frac{(p+h)(q+h)}{(P+h)(Q+h)},
\]
and hence
\[
d_{-v}^{-1}
\leq
\exp\left[
-cH\log\left(1+\frac vN\right)
\right].
\]
Integral comparison, after decreasing \(c\), yields
\begin{equation}\label{eq:d-negative-tail}
\sum_{v=v_0}^\infty d_{-v}^{-1}
\leq
C\left(1+\frac NH\right)
\left(1+\frac{v_0}{N}\right)^{-cH}.
\end{equation}
For \(v\leq N\), this also gives
\[
d_{-v}^{-1}\leq\exp(-c\rho v).
\]

Thus both cross-block series have an exponentially damped central mode range and a summable remote tail. The scale $N^{2/3+\epsilon}$ enters the proof through this damping: it makes $\rho$ larger than the reciprocal Airy mode scale by the factor $N^\epsilon$ needed below.

\noindent\textbf{Step 3. Equal-level control on logarithmic windows.}
The lower event is a gap event, so convergence of the equal-level kernel is not enough. The proof needs a quantitative inverse for $I$ minus the compressed diagonal block. The positive- and mixed-window Laguerre approximations give the operator comparison, the Airy spectral gap supplies the inverse scale, and the diagonal estimate later controls the derivative of the finite-$N$ gap probability when its wall is moved.

We prove, uniformly for
\[
c_0N\leq r\leq s\leq C_0N,
\qquad
0\leq R\leq D(\log N)^{1/3},
\]
that the rescaled finite projection
\(\mathcal L_{r,s,R}\) on \(L^2((-R,\infty))\) satisfies
\begin{equation}\label{eq:equal-level-op-comparison}
\norm{\mathcal L_{r,s,R}-K_{\Ai,-R}}_{\op}
\leq
N^{-1/3+o(1)},
\end{equation}
\begin{equation}\label{eq:equal-level-resolvent}
\norm{(I-\mathcal L_{r,s,R})^{-1}}_{\op}
\leq
N^{o(1)}
\end{equation}
whenever \(R\) is at least the fixed lower constant \(R_0\) from Proposition~\ref{prop:airy-kernel-gap}, and, uniformly for all displayed \(R\) and \(-R\leq t\leq0\),
\begin{equation}\label{eq:equal-level-diagonal}
\mathcal L_{r,s,R}(t,t)\leq N^{o(1)}.
\end{equation}

Put \(r_0=s-r\) and use
\[
\mathsf L_k=\phi_{k,k+r_0}.
\]
The finite series also gives the three-term recurrence
\begin{align*}
x\mathsf L_k(x)
={}&
-\sqrt{(k+1)(k+r_0+1)}\,\mathsf L_{k+1}(x)
\\
&+
(2k+r_0+1)\mathsf L_k(x)
-
\sqrt{k(k+r_0)}\,\mathsf L_{k-1}(x).
\end{align*}
Multiply this identity by \(\mathsf L_k\) at the other variable, subtract the transposed identity, and sum. The interior terms telescope to
\begin{equation}\label{eq:airy-christoffel-darboux}
\sum_{k=0}^{r-1}\mathsf L_k(x)\mathsf L_k(y)
=
\sqrt{rs}\,
\frac{
\mathsf L_{r-1}(x)\mathsf L_r(y)
-
\mathsf L_r(x)\mathsf L_{r-1}(y)
}{
x-y
}.
\end{equation}

Let \((\mu_0,\sigma_0,\delta_0)\) be the half-integer parameters at \((r,s)\), and let
\((\mu_1,\sigma_1,\delta_1)\) be those at \((r-1,s-1)\). Define
\[
g(\xi)=G_{r,s}(\xi),
\]
\[
h(\xi)
=
G_{r-1,s-1}
\bigl(
\mathfrak r_N\xi+\mathfrak s_N
\bigr),
\]
where
\[
\mathfrak r_N=\frac{\sigma_0}{\sigma_1},
\qquad
\mathfrak s_N=\frac{\mu_0-\mu_1}{\sigma_1},
\]
and set
\[
\mathfrak p_N
=
\sqrt{rs}
\sqrt{\frac{\delta_0\delta_1}{\sigma_0\sigma_1}}.
\]
Substitution in \eqref{eq:airy-christoffel-darboux}, including the alternating signs in \(G\), gives
\begin{equation}\label{eq:rescaled-cd}
\mathcal L_{r,s,R}(\xi,\eta)
=
\mathfrak p_N
\frac{
g(\xi)h(\eta)-h(\xi)g(\eta)
}{
\xi-\eta
}.
\end{equation}

Along the path \((u-t,v-t)\), where
\[
u=r+\frac12,
\qquad
v=s+\frac12,
\]
direct differentiation gives
\[
-\frac{\dd}{\dd t}
(\sqrt{u-t}+\sqrt{v-t})^2
=
\sigma_{r-t,s-t}\delta_{r-t,s-t}.
\]
Integration for \(0\leq t\leq1\) and the explicit parameter formulas yield
\begin{equation}\label{eq:parameter-neighbor-comparison}
\mathfrak r_N=1+O(N^{-1}),
\qquad
\mathfrak s_N=\delta_0(1+O(N^{-1})),
\qquad
\mathfrak p_N=\delta_0^{-1}(1+O(N^{-1})).
\end{equation}
For the last relation, use
\[
\sqrt{\left(r+\frac12\right)\left(s+\frac12\right)}
\frac{\delta_0}{\sigma_0}
=
\delta_0^{-1}
\]
and
\[
\frac{\sqrt{rs}}{
\sqrt{(r+\frac12)(s+\frac12)}
}
=
1+O(N^{-1}).
\]

Put
\[
\mathfrak j=\frac{h-g}{\mathfrak s_N},
\qquad
\mathfrak w_N=\mathfrak p_N\mathfrak s_N.
\]
Then
\[
\mathfrak w_N=1+O(N^{-1}),
\]
and \eqref{eq:rescaled-cd} is \(\mathfrak w_N\) times
\begin{equation}\label{eq:divided-difference}
\mathcal T(g,\mathfrak j)(\xi,\eta)
=
\frac{
g(\xi)\mathfrak j(\eta)-\mathfrak j(\xi)g(\eta)
}{
\xi-\eta
},
\end{equation}
with the continuous diagonal value.

Choose a fixed large \(L\), and set
\[
T=L\log N.
\]
On
\[
I_N=[-R,T],
\]
Proposition~\ref{prop:mixed-window-c1}, Taylor's formula,
\eqref{eq:parameter-neighbor-comparison}, and the Airy equation give
\begin{equation}\label{eq:cd-mode-approximations}
\norm{g-\Ai}_{C^1(I_N)}
\leq
N^{-2/3+o(1)},
\qquad
\norm{\mathfrak j-\Ai'}_{C^1(I_N)}
\leq
N^{-1/3+o(1)}.
\end{equation}
Lemma~\ref{lem:airy-envelopes} and
\[
\Ai'(x)
=
\Ai(x+1)-\Ai(x)
-
\int_x^{x+1}(x+1-t)t\Ai(t)\dd t
\]
show that \(\Ai\) and its first three derivatives are \(N^{o(1)}\) on every such logarithmic box. The divided-difference map \(\mathcal T\) is locally Lipschitz in the \(C^1\) norm. Integrating the Airy Wronskian and using its positive decay gives
\[
K_{\Ai}=\mathcal T(\Ai,\Ai').
\]
Hence
\begin{equation}\label{eq:boxed-kernel-comparison}
\sup_{\xi,\eta\in I_N}
\abs{
\mathcal L_{r,s,R}(\xi,\eta)-K_{\Ai}(\xi,\eta)
}
\leq
N^{-1/3+o(1)}.
\end{equation}
The boxed Hilbert--Schmidt norm has the same bound.

It remains to remove \(T\). We derive the needed Laguerre differential equation directly. Write
\[
L_n^{(r)}(x)=\sum_{k=0}^nc_kx^k,
\qquad
c_k
=
(-1)^k
\frac{(n+r)!}{(n-k)!(r+k)!k!}.
\]
For \(0\leq k\leq n\), with \(c_{n+1}=0\), the coefficient of \(x^k\) in
\[
x(L_n^{(r)})''
+
(r+1-x)(L_n^{(r)})'
+
nL_n^{(r)}
\]
is
\[
(k+1)(k+r+1)c_{k+1}+(n-k)c_k=0,
\]
because
\[
\frac{c_{k+1}}{c_k}
=
-\frac{n-k}{(k+1)(r+k+1)}.
\]
Thus
\begin{equation}\label{eq:laguerre-series-ode}
x(L_n^{(r)})''
+
(r+1-x)(L_n^{(r)})'
+
nL_n^{(r)}
=
0.
\end{equation}
Put
\[
\phi(x)
=
\sqrt{\frac{n!}{(n+r)!}}\,
\exp(-x/2)x^{r/2}L_n^{(r)}(x).
\]
Differentiating twice and using \eqref{eq:laguerre-series-ode} cancels the derivative terms and gives
\begin{equation}\label{eq:laguerre-phi-ode}
\phi''(x)+x^{-1}\phi'(x)
+
\left[
-\frac14
+
\frac{n+(r+1)/2}{x}
-
\frac{r^2}{4x^2}
\right]\phi(x)
=
0.
\end{equation}

Let
\[
u=n+\frac12,
\qquad
v=n+r+\frac12,
\qquad
\mu=(\sqrt u+\sqrt v)^2,
\]
and let \(\sigma\) be the corresponding half-integer soft-edge scale. Since
\[
\mu-(\sqrt v-\sqrt u)^2=4\sqrt{uv},
\qquad
\sigma^3=\frac{\mu^2}{\sqrt{uv}},
\]
the change \(x=\mu+\sigma z\), followed by the conjugation
\[
Y(z)=\sqrt{1+(\sigma/\mu)z}\,G_{n,n+r}(z),
\]
transforms \eqref{eq:laguerre-phi-ode} into
\begin{equation}\label{eq:positive-rescaled-ode}
Y''(z)=q(z)Y(z),
\end{equation}
where
\[
q(z)
=
\frac{z}{(1+(\sigma/\mu)z)^2}
+
\frac{
\sigma^4z^2-\sigma^2
}{
4\mu^2(1+(\sigma/\mu)z)^2
}.
\]
For comparable \(n\) and \(n+r\),
\[
\abs{q(z)}\leq CN^{2/3},
\qquad z\geq0.
\]
Proposition~\ref{prop:positive-airy-laguerre} and
Lemma~\ref{lem:airy-envelopes} give
\[
\abs{G_{n,n+r}(z)}
\leq
C\exp(-z/4).
\]
Taylor's integral formula on \([z,z+1]\), applied to
\eqref{eq:positive-rescaled-ode}, gives
\begin{equation}\label{eq:positive-G-derivative}
\abs{G_{n,n+r}'(z)}
\leq
CN^{2/3}\exp(-z/8).
\end{equation}
The same estimates hold for \(h\) at its shifted coordinate. Taking the diagonal limit in \eqref{eq:rescaled-cd} yields
\begin{equation}\label{eq:laguerre-positive-diagonal}
0\leq
\mathcal L_{r,s,R}(z,z)
\leq
CN\exp(-z/16),
\qquad z\geq0.
\end{equation}

Since the uncompressed finite kernel is a projection by
\eqref{eq:laguerre-orthonormality}, its off-box Hilbert--Schmidt norm squared is at most the tail trace. By
\eqref{eq:laguerre-positive-diagonal}, this is at most
\[
CN\exp(-T/16).
\]
The Airy contraction has the same projection-type bound, while its tail trace equals
\begin{equation}\label{eq:airy-tail-trace}
\int_T^\infty(x-T)\Ai(x)^2\dd x
\leq
C\exp(-cT^{3/2}).
\end{equation}
Choosing \(L\) large proves
\eqref{eq:equal-level-op-comparison}.

Proposition~\ref{prop:airy-kernel-gap} gives
\[
1-\norm{K_{\Ai,-R}}_{\op}
\geq
\exp(-CR^{3/2})
=
N^{-o(1)}.
\]
Combining this with \eqref{eq:equal-level-op-comparison} proves
\eqref{eq:equal-level-resolvent} by the resolvent identity. Finally,
\eqref{eq:boxed-kernel-comparison} on the diagonal and
Lemma~\ref{lem:airy-envelopes} give
\[
K_{\Ai}(t,t)\leq N^{o(1)},
\qquad
-R\leq t\leq0,
\]
which proves \eqref{eq:equal-level-diagonal}.

The three outputs of this step have separate uses: operator comparison transfers the classical Airy kernel to finite $N$, the inverse bound controls the diagonal gap blocks in the lower Schur complement, and the diagonal bound supplies relative lower-wall regularity in Step~6.

\noindent\textbf{Step 4. Restricted modes and cross-block norms.}
The cross blocks are now restricted to the actual event sets. For central modes, the uniform positive and mixed-window approximations estimate the two restricted vector norms. For remote modes, orthonormality and coefficient damping are sufficient. The cutoff is chosen so that the separation gain survives the sum over the Airy-scale mode window.

Put
\[
L=\left\lfloor N^{1/3-\epsilon/2}\right\rfloor.
\]
For \(i\in\{n,m\}\) and \(\abs\ell\leq L\), smoothness of the explicit edge parameters gives
\begin{equation}\label{eq:shifted-mode-parameters}
\frac{\sigma_{p_i-\ell,q_i-\ell}}{\sigma_i}
=
1+O(\abs\ell/N),
\end{equation}
and
\begin{equation}\label{eq:shifted-mode-centers}
\mu_i-\mu_{p_i-\ell,q_i-\ell}
=
\sigma_i\delta_i\ell
+
O(\sigma_i\delta_i\ell^2/N),
\end{equation}
where \(\delta_i\) is comparable to \(N^{-1/3}\). Thus the upper-event wall in the mode's natural coordinate is
\[
z_{i,\ell}^+
=
a(\log N)^{2/3}
+
O(1)
+
\delta_i\ell
+
o(1),
\]
and the lower forbidden wall is
\[
z_{i,\ell}^-
=
-b(\log N)^{1/3}
+
O(1)
+
\delta_i\ell
+
o(1).
\]

Proposition~\ref{prop:positive-airy-laguerre} and the positive Airy envelope imply
\begin{equation}\label{eq:upper-mode-norm}
R_{i,\ell}^+
:=
\norm{
\ind_{I_i^+}
\phi_{p_i-\ell,q_i-\ell}
}_2
\leq
N^{-1/6-\kappa_+/2+o(1)}.
\end{equation}
Its squared norm is \(\delta_i\) times the Airy-square tail above the first wall, up to a smaller approximation error. Proposition~\ref{prop:mixed-window-c1} and both Airy envelopes similarly give
\begin{equation}\label{eq:lower-mode-norm}
R_{i,\ell}^-
:=
\norm{
\ind_{B_i^-}
\phi_{p_i-\ell,q_i-\ell}
}_2
\leq
N^{-1/6+o(1)}.
\end{equation}
For \eqref{eq:lower-mode-norm}, integrate the squared \(C^0\) part of the mixed-window estimate from the negative wall to \(L\log N\), and use the positive-edge estimate beyond that cutoff. The Airy-square integral over a logarithmic negative window is \(N^{o(1)}\).

Restrict the exact \(\widehat K_{10}\) and \(\widehat K_{01}\) series. The nuclear norm of each rank-one term is the product of its two restricted vector norms. For \(\abs\ell\leq L\), use
\eqref{eq:upper-mode-norm} or \eqref{eq:lower-mode-norm} together with the exponential parts of
\eqref{eq:d-positive-product}--\eqref{eq:d-negative-tail}. Beyond \(L\), use full \(L^2\)-normalization
\eqref{eq:laguerre-orthonormality} and the summable coefficient tails. Since
\[
\rho L\geq cN^{\epsilon/2},
\]
this gives
\begin{equation}\label{eq:upper-cross-nuclear}
\norm{T_{10}^+}_1+\norm{T_{01}^+}_1
\leq
N^{-\kappa_+-\epsilon+o(1)}
\end{equation}
and
\begin{equation}\label{eq:lower-cross-nuclear}
\norm{T_{10}^-}_1+\norm{T_{01}^-}_1
\leq
N^{-\epsilon+o(1)}.
\end{equation}
Operator and Hilbert--Schmidt norms do not exceed nuclear norm, so all six asserted bounds on each side follow.

The two signs carry different norm budgets. On the upper half-line, each restricted mode already contains half of the marginal rarity exponent, so the product of two mode norms and the coefficient damping yields a cross correction smaller than the upper marginal product. On the lower forbidden half-line, the restricted mode norms have no comparable rarity power; the required smallness instead comes from separation damping, while the marginal gap probabilities remain in the diagonal determinants.

\noindent\textbf{Step 5. Upper mixed counts and lower gap events.}
For the upper event, the useful observable is the occupancy count of the upper half-line. The mixed-count identity compares its second mixed moment with the product of the two intensities and a cross-block trace. Because occurrence of the event only implies positive occupancy, the final comparison is intrinsically an upper bound.

Let \(Z_i^+\) be the number of points in \(I_i^+\), and put
\[
\nu_i^+=\E Z_i^+.
\]
The equal-level restrictions are finite-rank Hermitian operators and therefore have finite intensities. The cross restrictions in
\eqref{eq:upper-cross-nuclear} are nuclear and hence Hilbert--Schmidt. Cauchy--Schwarz for their \(L^2\) kernels gives absolute integrability of the mixed-kernel product. Lemma~\ref{lem:two-level-reductions} therefore gives
\begin{equation}\label{eq:mixed-count-identity}
\E[Z_n^+Z_m^+]
=
\nu_n^+\nu_m^+
-
\Tr(T_{01}^+T_{10}^+).
\end{equation}

We next verify that \(\nu_i^+\to0\). Let
\[
t_i=a(\log i)^{2/3}+h_{i,+}.
\]
On \([t_i,T]\), the boxed diagonal estimate
\eqref{eq:boxed-kernel-comparison} gives
\[
\int_{t_i}^{T}
\mathcal L_{p_i,q_i,0}(z,z)\dd z
\leq
\int_{t_i}^{T}K_{\Ai}(z,z)\dd z
+
N^{-1/3+o(1)}.
\]
The Airy integral tends to zero by Lemma~\ref{lem:airy-envelopes}. The finite-Laguerre contribution beyond \(T\) tends to zero by
\eqref{eq:laguerre-positive-diagonal} after the constant \(L\) in
\(T=L\log N\) is chosen large. Therefore
\[
\nu_i^+=o(1)
\]
uniformly. The one-level finite-intensity inequality gives
\[
\nu_i^+-\frac{(\nu_i^+)^2}{2}
\leq
u_i^+
\leq
\nu_i^+.
\]
Using \eqref{eq:two-time-marginals},
\[
\nu_i^+=(1+o(1))u_i^+.
\]
The ideal property and \eqref{eq:upper-cross-nuclear} make the absolute trace correction in
\eqref{eq:mixed-count-identity} at most
\[
N^{-2\kappa_+-2\epsilon+o(1)}
=
o(u_n^+u_m^+).
\]
Since
\[
\ind_{E_n^+}\ind_{E_m^+}
\leq
Z_n^+Z_m^+,
\]
equation~\eqref{eq:mixed-count-identity} proves
\[
\Prob(E_n^+\cap E_m^+)
\leq
(1+o(1))u_n^+u_m^+.
\]
The argument supplies no reverse comparison: the inequality between the event indicator and the occupancy product is used only in the displayed direction.

For the lower sets, the relevant event is instead the absence of every point from a forbidden upper half-line. Its probability is a Fredholm gap determinant. Factoring the two-level block determinant by a Schur complement leaves the product of the two marginal gap determinants multiplied by a correction involving both diagonal gap inverses and the two restricted cross blocks. The inverse bounds from Step~3 and the nuclear estimates from Step~4 bound that correction above by $\exp(N^{-2\epsilon+o(1)})=1+o(1)$.

For the lower sets, apply the logarithmic-window estimate with
\[
R_i=b(\log i)^{1/3}-h_{i,-}.
\]
For large \(N\), \eqref{eq:equal-level-resolvent} gives
\begin{equation}\label{eq:lower-block-resolvents}
\norm{(A_n^-)^{-1}}_{\op}
+
\norm{(D_m^-)^{-1}}_{\op}
\leq
N^{o(1)}.
\end{equation}
The diagonal blocks are finite rank, and the cross blocks are nuclear by
\eqref{eq:lower-cross-nuclear}. Thus the full restricted two-level block is trace class. Its Schur complement is trace class and, by
\eqref{eq:lower-cross-nuclear},
\eqref{eq:lower-block-resolvents}, and the ideal property,
\begin{equation}\label{eq:lower-schur-norm}
\norm{
(D_m^-)^{-1}T_{10}^-
(A_n^-)^{-1}T_{01}^-
}_1
\leq
N^{-2\epsilon+o(1)}.
\end{equation}
The gap-event formula in Lemma~\ref{lem:two-level-reductions} now yields
\[
\Prob(G_n^-\cap G_m^-)
\leq
\exp\bigl(N^{-2\epsilon+o(1)}\bigr)
u_n^-u_m^-
=
(1+o(1))u_n^-u_m^-.
\]
This is again only a joint-probability upper estimate. The determinant inequality used in the Schur-complement step does not produce a complementary lower bound.

\noindent\textbf{Step 6. Lower-wall deletion.}
This step proves the theorem's separate wall-regularity clause. Since the lower gap probability is itself polynomially small, an absolute continuity estimate for a moved wall would not be sufficient. Jacobi's formula identifies the logarithmic derivative of the finite-$N$ gap determinant with a resolvent diagonal. The equal-level inverse and diagonal bounds make this derivative $N^{o(1)}$, so deleting a polynomially short wall interval changes the gap probability only by a relative $N^{-\alpha_{\mathrm{wall}}+o(1)}$ amount.

Rescale the equal-time Laguerre projection at level \(j\) to the half-integer soft-edge coordinate. By \eqref{eq:laguerre-orthonormality}, its full kernel has a finite orthonormal expansion
\[
K_j(x,y)
=
\sum_{r=1}^{p_j}f_r(x)f_r(y).
\]
Let \(K_{j,t}\) be its compression to \((t,\infty)\), and define the \(p_j\times p_j\) Gram matrix
\begin{equation}\label{eq:wall-Gram}
G_t(r,s)
=
\int_t^\infty f_r(x)f_s(x)\dd x.
\end{equation}

Let \(N_t\) count the level-\(j\) points in \((t,\infty)\). Since there are exactly \(p_j\) points at this level, the finite inclusion-exclusion identity is
\[
\ind_{\{N_t=0\}}
=
\sum_{k=0}^{p_j}
\frac{(-1)^k}{k!}(N_t)_k,
\]
where \((N_t)_k\) is the falling factorial. Taking expectations and using the determinantal correlation functions gives
\begin{align}
\Prob(N_t=0)
&=
\sum_{k=0}^{p_j}
\frac{(-1)^k}{k!}
\int_{(t,\infty)^k}
\det[K_j(x_a,x_b)]_{a,b=1}^k
\dd x_1\cdots\dd x_k
\nonumber\\
&=
\det(I-G_t).
\label{eq:wall-gap-determinant}
\end{align}
For the last equality, Cauchy--Binet expands the integrand as the sum over \(k\)-element index sets of squares of determinants
\[
\det[f_r(x_a)]^2.
\]
Integration gives \(k!\) times the sum of the principal \(k\times k\) minors of \(G_t\). Summation with coefficient \((-1)^k/k!\) is the principal-minor expansion of \(\det(I-G_t)\).

The nonzero eigenvalues of \(G_t\) and \(K_{j,t}\) agree, so
\[
\det(I-G_t)=\det(I-K_{j,t}).
\]
Since \(N_t=0\) is precisely the event
\[
\Lambda_j\leq\mu_j+\sigma_jt,
\]
the left side of \eqref{eq:wall-gap-determinant} is \(F_j(t)\).

On the wall window, \eqref{eq:equal-level-resolvent} makes \(I-K_{j,t}\), and equivalently \(I-G_t\), invertible. If
\[
f(t)=(f_1(t),\ldots,f_{p_j}(t))^T,
\]
then
\[
G_t'=-f(t)f(t)^T.
\]
Jacobi's determinant formula applied to
\eqref{eq:wall-gap-determinant} gives
\begin{equation}\label{eq:wall-jacobi}
F_j'(t)
=
F_j(t)
f(t)^T(I-G_t)^{-1}f(t).
\end{equation}
Define
\[
V_tc
=
\sum_r c_r\ind_{(t,\infty)}f_r.
\]
Then
\[
K_{j,t}=V_tV_t^*,
\qquad
G_t=V_t^*V_t.
\]
The identity
\[
V_t^*(I-V_tV_t^*)^{-1}
=
(I-V_t^*V_t)^{-1}V_t^*
\]
identifies the scalar in \eqref{eq:wall-jacobi} with the diagonal of
\[
R_{j,t}
=
K_{j,t}(I-K_{j,t})^{-1}.
\]
Consequently,
\begin{equation}\label{eq:wall-derivative}
F_j'(t)=F_j(t)R_{j,t}(t,t).
\end{equation}

Since the full \(K_j\) is an orthogonal projection, its compressed resolvent satisfies
\[
R_{j,t}
=
K_{j,t}
+
K_{j,t}(I-K_{j,t})^{-1}K_{j,t}.
\]
The squared \(L^2\)-norm of the restricted kernel row at \(t\) is at most \(K_j(t,t)\). Equations~\eqref{eq:equal-level-resolvent} and
\eqref{eq:equal-level-diagonal} therefore give
\begin{equation}\label{eq:wall-resolvent-diagonal}
0\leq
R_{j,t}(t,t)
\leq
K_j(t,t)
\left[
1+\norm{(I-K_{j,t})^{-1}}_{\op}
\right]
\leq
N^{o(1)}
\end{equation}
uniformly for
\[
-D_{\mathrm{wall}}(\log N)^{1/3}
\leq
t
\leq
-R_*
\]
after increasing \(R_*\).

Integrating \eqref{eq:wall-derivative} over
\([s,s+\ell_{\mathrm{wall}}]\), and using the monotonicity of \(F_j\) together with \eqref{eq:wall-resolvent-diagonal}, gives
\[
F_j(s+\ell_{\mathrm{wall}})-F_j(s)
\leq
\ell_{\mathrm{wall}}N^{o(1)}
F_j(s+\ell_{\mathrm{wall}}).
\]
The finite Laguerre density is strictly positive on its ordered chamber, so
\(F_j(s)>0\) for every finite wall in this range. If
\[
\ell_{\mathrm{wall}}
\leq
N^{-\alpha_{\mathrm{wall}}},
\]
rearrangement gives
\[
\frac{F_j(s+\ell_{\mathrm{wall}})}{F_j(s)}
\leq
1+N^{-\alpha_{\mathrm{wall}}+o(1)}.
\]
Monotonicity gives the lower bound \(1\). All estimates are uniform. Taking a maximum over at most
\(C_{\mathrm{lad}}\log N\) deterministic intervals preserves the \(N^{o(1)}\) factor.

The logarithmic-size deterministic family does not alter the exponent budget, because its cardinality is absorbed into the existing $N^{o(1)}$ term. This wall statement is independent of the upper occupancy reduction and is not needed to manufacture a missing lower joint bound.

\noindent\textbf{Step 7. Change of normalization.}
The operator analysis uses the half-integer Laguerre center and scale because they match the uniform Airy approximation. The project statistic uses the unshifted rectangular center and scale. Their center difference is bounded and their scale ratio is $1+O(N^{-1})$, so the induced change in either logarithmic threshold is $o(1)$ throughout the short macroscopic window. The already established uniformity over bounded deterministic shifts absorbs this perturbation; the transfer does not strengthen the one-sided conclusion.

Equation~\eqref{eq:two-time-normalization-comparison} shows that replacing \(\overline\chi_j\) by \(\chi_j^{\mathrm{proj}}\) changes each displayed logarithmic threshold by \(o(1)\), uniformly on the macroscopic window. The bounded-shift conclusions already proved therefore transfer all marginal and pair estimates.
\end{proof}

The one-sided pair theorem has an intrinsic counting consequence. For a sign \(\mathfrak s\), the exponent \(\kappa_{\mathfrak s}\) is the marginal rarity cost: a single active event has probability \(N^{-\kappa_{\mathfrak s}+o(1)}\). The exponent \(\theta_{\mathfrak s}\) prescribes the desired count growth. Choosing \(\lfloor N^{\kappa_{\mathfrak s}+\theta_{\mathfrak s}}\rfloor\) test levels compensates for the rarity and makes the first moment of the count of order \(N^{\theta_{\mathfrak s}}\). The exponent \(\epsilon_{\mathfrak s}\) is the supercritical-spacing cost: the constructed arithmetic grid has step \(\lceil N^{2/3+\epsilon_{\mathfrak s}}\rceil\), so every distinct pair lies in the separation range of Theorem~\ref{thm:nested-wishart-decorrelation}. Its span is at most
\[
N^{\kappa_{\mathfrak s}+\theta_{\mathfrak s}+2/3+\epsilon_{\mathfrak s}+o(1)},
\]
and the strict inequality
\(\kappa_{\mathfrak s}+\theta_{\mathfrak s}+\epsilon_{\mathfrak s}<1/3\)
is exactly the budget that places the entire grid inside the theorem's short macroscopic window.

The second moment separates into diagonal and off-diagonal terms. The diagonal contribution equals the first moment because the summands are indicators. For distinct grid points, the one-sided product-scale upper estimate bounds each joint probability by the corresponding marginal product up to the single uniform \(1+o(1)\) factor, which controls the full off-diagonal sum without requiring a reverse pair inequality. The bounded deterministic shifts are retained uniformly at every grid point. The conditioning clause records a separate exact stability property: when \(\mathcal B_N\) is independent of the entire auxiliary complex Gaussian array, the conditional first and second moments agree almost surely with their unconditional counterparts.

\begin{cor}[Separated Gaussian grid moments]\label{cor:gaussian-grid-moments}

Assume the nested complex Gaussian covariance path and the normalization
\(\chi_j^{\mathrm{proj}}\) of Theorem~\ref{thm:nested-wishart-decorrelation}. Fix \(a,b>0\), put
\[
\kappa_+=\frac43a^{3/2},
\qquad
\kappa_-=\frac{b^3}{12},
\]
and assume
\[
\max(\kappa_+,\kappa_-)<\frac13.
\]
For each sign \(\mathfrak s\in\{+,-\}\), choose
\(\theta_{\mathfrak s}>0\) and \(\epsilon_{\mathfrak s}>0\) such that
\[
\kappa_{\mathfrak s}
+
\theta_{\mathfrak s}
+
\epsilon_{\mathfrak s}
<
\frac13.
\]
Fix \(H_{\mathrm{shift}}<\infty\). For every sufficiently large integer \(N\), there are grids
\[
\mathcal T_N^+,\mathcal T_N^-
\subseteq[N,2N)
\]
such that distinct points of \(\mathcal T_N^{\mathfrak s}\) are separated by at least
\[
N^{2/3+\epsilon_{\mathfrak s}},
\]
\[
\abs{\mathcal T_N^{\mathfrak s}}
=
\left\lfloor
N^{\kappa_{\mathfrak s}+\theta_{\mathfrak s}}
\right\rfloor,
\]
and the following holds uniformly for every deterministic family of shifts
\(h_{j,\mathfrak s}\), \(j\in\mathcal T_N^{\mathfrak s}\), with
\(\abs{h_{j,\mathfrak s}}\leq H_{\mathrm{shift}}\). Put
\[
S_N^+
=
\sum_{j\in\mathcal T_N^+}
\ind_{\{
\chi_j^{\mathrm{proj}}
\geq
a(\log j)^{2/3}+h_{j,+}
\}},
\]
\[
S_N^-
=
\sum_{j\in\mathcal T_N^-}
\ind_{\{
\chi_j^{\mathrm{proj}}
\leq
-b(\log j)^{1/3}+h_{j,-}
\}}.
\]
Then
\[
\E S_N^{\mathfrak s}
=
N^{\theta_{\mathfrak s}+o(1)}
\]
and
\[
\E[(S_N^{\mathfrak s})^2]
\leq
(1+o(1))(\E S_N^{\mathfrak s})^2
+
\E S_N^{\mathfrak s}.
\]
In particular, for all sufficiently large \(N\),
\[
\E S_N^{\mathfrak s}
\geq
N^{\theta_{\mathfrak s}/2}
\]
and
\[
\frac{
\E[(S_N^{\mathfrak s})^2]
}{
(\E S_N^{\mathfrak s})^2
}
\leq
1+o(1).
\]

If \(\mathcal B_N\) is any \(\sigma\)-algebra independent of the entire auxiliary complex Gaussian array, then the same identities and inequalities hold almost surely after conditioning on \(\mathcal B_N\). In particular, for any \(\eta>0\) satisfying
\[
\kappa_+ + 2\eta + \epsilon_+ < \frac13,
\qquad
\kappa_- + \eta + \epsilon_- < \frac13,
\]
the choices
\[
\theta_+=2\eta,
\qquad
\theta_-=\eta
\]
are admissible.
\end{cor}

\begin{proof}
Fix one sign \(\mathfrak s\). Keep the amplitude belonging to this sign equal to the amplitude in the statement. Choose the unused opposite-sign amplitude positive and small enough that its rarity exponent is at most
\(\kappa_{\mathfrak s}\). This is possible because
\[
\frac43a^{3/2}
\quad\text{and}\quad
\frac{b^3}{12}
\]
both tend to zero with their respective positive amplitudes. Thus, in this application of Theorem~\ref{thm:nested-wishart-decorrelation}, the maximum of the two rarity exponents is \(\kappa_{\mathfrak s}\). The hypothesis
\[
\kappa_{\mathfrak s}
+
\theta_{\mathfrak s}
+
\epsilon_{\mathfrak s}
<
\frac13
\]
implies
\[
\epsilon_{\mathfrak s}
<
\frac13-\kappa_{\mathfrak s},
\]
so the theorem applies at the active sign.

Apply this argument separately to the two signs, and let \(\eta_0>0\) be the smaller of the two resulting short-window constants. The project-normalization clause of Theorem~\ref{thm:nested-wishart-decorrelation} gives, uniformly for the active event at every integer
\(j\in[N,(1+\eta_0)N]\),
\[
\Prob(H_{N,j}^{\mathfrak s})
=
N^{-\kappa_{\mathfrak s}+o(1)},
\]
and, uniformly for distinct \(j,k\) in this window satisfying
\[
\abs{j-k}
\geq
N^{2/3+\epsilon_{\mathfrak s}},
\]
\[
\Prob(H_{N,j}^{\mathfrak s}\cap H_{N,k}^{\mathfrak s})
\leq
(1+o(1))
\Prob(H_{N,j}^{\mathfrak s})
\Prob(H_{N,k}^{\mathfrak s}).
\]
Here \(H_{N,j}^+\) is the upper event with threshold
\[
a(\log j)^{2/3}+h_{j,+},
\]
and \(H_{N,j}^-\) is the lower event with threshold
\[
-b(\log j)^{1/3}+h_{j,-}.
\]
Both error terms are deterministic and uniform over the endpoint labels and every allowed deterministic bounded shift family.

For this sign, define
\[
h_N^{\mathfrak s}
=
\left\lceil
N^{2/3+\epsilon_{\mathfrak s}}
\right\rceil,
\]
\[
L_N^{\mathfrak s}
=
\left\lfloor
N^{\kappa_{\mathfrak s}+\theta_{\mathfrak s}}
\right\rfloor,
\]
and
\[
\mathcal T_N^{\mathfrak s}
=
\left\{
N+rh_N^{\mathfrak s}
\mid
0\leq r<L_N^{\mathfrak s}
\right\}.
\]
The grid spacing is at least
\(N^{2/3+\epsilon_{\mathfrak s}}\), and
\[
\abs{\mathcal T_N^{\mathfrak s}}
=
L_N^{\mathfrak s}
=
N^{\kappa_{\mathfrak s}+\theta_{\mathfrak s}+o(1)}.
\]
Its largest displacement from \(N\) is at most
\begin{align*}
L_N^{\mathfrak s}h_N^{\mathfrak s}
&=
N^{
\kappa_{\mathfrak s}
+
\theta_{\mathfrak s}
+
2/3
+
\epsilon_{\mathfrak s}
+
o(1)
}
\\
&=
o(N),
\end{align*}
where the final equality is exactly the strict exponent hypothesis. Hence, for every sufficiently large \(N\),
\[
\mathcal T_N^{\mathfrak s}
\subseteq
[N,(1+\eta_0)N]
\subseteq
[N,2N).
\]

Put
\[
u_{N,j}^{\mathfrak s}
=
\Prob(H_{N,j}^{\mathfrak s})
\]
and
\[
m_N^{\mathfrak s}
=
\sum_{j\in\mathcal T_N^{\mathfrak s}}
u_{N,j}^{\mathfrak s}
=
\E S_N^{\mathfrak s}.
\]
The uniform one-point estimate and the grid cardinality give
\[
m_N^{\mathfrak s}
=
N^{\theta_{\mathfrak s}+o(1)}.
\]

For the second moment, every summand of \(S_N^{\mathfrak s}\) is an indicator. Using the single uniform pair-error term gives
\begin{align*}
\E[(S_N^{\mathfrak s})^2]
&=
m_N^{\mathfrak s}
+
\sum_{\substack{
j,k\in\mathcal T_N^{\mathfrak s}\\
j\neq k
}}
\Prob(H_{N,j}^{\mathfrak s}\cap H_{N,k}^{\mathfrak s})
\\
&\leq
m_N^{\mathfrak s}
+
(1+o(1))
\sum_{\substack{
j,k\in\mathcal T_N^{\mathfrak s}\\
j\neq k
}}
u_{N,j}^{\mathfrak s}u_{N,k}^{\mathfrak s}
\\
&\leq
m_N^{\mathfrak s}
+
(1+o(1))(m_N^{\mathfrak s})^2.
\end{align*}
Because \(\theta_{\mathfrak s}>0\) and
\[
m_N^{\mathfrak s}
=
N^{\theta_{\mathfrak s}+o(1)},
\]
one has
\[
m_N^{\mathfrak s}
\geq
N^{\theta_{\mathfrak s}/2}
\]
for all sufficiently large \(N\), and
\[
\frac1{m_N^{\mathfrak s}}=o(1).
\]
Dividing the preceding second-moment bound by
\((m_N^{\mathfrak s})^2\) proves
\[
\frac{
\E[(S_N^{\mathfrak s})^2]
}{
(\E S_N^{\mathfrak s})^2
}
\leq
1+o(1).
\]
The construction applies separately to both signs.

Now let \(\mathcal B_N\) be independent of the entire auxiliary complex Gaussian array. Each \(S_N^{\mathfrak s}\) is a measurable function of that array and the deterministic grid and shift family. Independence gives, almost surely,
\[
\E[S_N^{\mathfrak s}\mid\mathcal B_N]
=
\E S_N^{\mathfrak s}
\]
and
\[
\E[(S_N^{\mathfrak s})^2\mid\mathcal B_N]
=
\E[(S_N^{\mathfrak s})^2].
\]
Thus every first- and second-moment identity and inequality remains valid after conditioning. Finally, the construction allowed every positive
\(\theta_{\mathfrak s}\) satisfying the strict exponent inequality. Substituting
\[
\theta_+=2\eta,
\qquad
\theta_-=\eta
\]
gives the stated recurrence choices whenever their corresponding inequalities hold.
\end{proof}

Thus Corollary~\ref{cor:gaussian-grid-moments} is an intrinsic sparse-grid consequence of the one-sided pair theorem: the rarity cost fixes the number of tests, the supercritical spacing fixes their separation, and the strict exponent budget keeps the construction within the available macroscopic window. The companion paper uses the final conditioning clause after exposing an external sigma-algebra; that application does not enlarge the ranges or strengthen the probability conclusion proved here.

\appendix

\section{Finite-particle determinant algebra}\label{app:gaussian-finite-particle-algebra}

\subsection{Two-level determinantal reductions}\label{app:proof-two-level-reductions}

\begin{proof}[Proof of Lemma~\ref{lem:two-level-reductions}]
For $z=(s,x)$ and $w=(t,y)$, the determinantal correlation identity gives
\[
 \rho_2(z,w)
 =
 K_{ss}(x,x)K_{tt}(y,y)
 -
 K_{st}(x,y)K_{ts}(y,x).
\]
Both terms are integrable under the hypotheses. The two-point Campbell identity therefore yields
\[
 \E[N_sN_t]=\mu_s\mu_t-C_{st}.
\]

At one level, Hermitianity gives
\[
 \rho_2(x,y)
 =
 K_{uu}(x,x)K_{uu}(y,y)-\abs{K_{uu}(x,y)}^2
 \leq
 K_{uu}(x,x)K_{uu}(y,y).
\]
Hence
\[
 \E[N_u(N_u-1)]\leq\mu_u^2.
\]
For every nonnegative integer $n$,
\[
 \ind_{\{n>0\}}\leq n,
 \qquad
 \ind_{\{n>0\}}\geq n-\frac{n(n-1)}2.
\]
The second inequality is an equality for $n=0,1,2$ and has nonpositive right side for $n\geq3$. Taking expectations gives
\[
 \mu_u-\frac{\mu_u^2}{2}
 \leq
 \Prob(N_u>0)
 \leq
 \mu_u.
\]
If $\mu_u\leq1$, then $\mu_u\leq2\Prob(N_u>0)$. Moreover,
\[
 \ind_{\{N_s>0\}}\ind_{\{N_t>0\}}\leq N_sN_t.
\]
Therefore
\[
 \begin{split}
 \Prob(N_s>0,N_t>0)
 &\leq
 \E[N_sN_t]\\
 &\leq
 \mu_s\mu_t+\abs{C_{st}}\\
 &\leq
 (1+c)\mu_s\mu_t\\
 &\leq
 4(1+c)\Prob(N_s>0)\Prob(N_t>0).
 \end{split}
\]

For the gap formula, the trace-class hypothesis permits the Fredholm inclusion-exclusion expansion:
\[
 \Prob(G_s)=\det A,
 \qquad
 \Prob(G_t)=\det D,
\]
and
\[
 \Prob(G_s\cap G_t)
 =
 \det
 \begin{bmatrix}
 A&-K_{st}\\
 -K_{ts}&D
 \end{bmatrix}.
\]
Every block of a trace-class operator on a direct sum is trace class. Block Gaussian elimination and multiplicativity of Fredholm determinants give
\[
 \begin{split}
 \det
 \begin{bmatrix}
 A&-K_{st}\\
 -K_{ts}&D
 \end{bmatrix}
 &=
 \det(A)\det(D-K_{ts}A^{-1}K_{st})\\
 &=
 \det(A)\det(D)
 \det(I-D^{-1}K_{ts}A^{-1}K_{st}).
 \end{split}
\]
The bounded inverses and the trace-ideal property show that $R$ is trace class. This proves the exact gap identity. The exterior-power Fredholm series gives
\[
 \abs{\det(I-R)}
 \leq
 \sum_{k=0}^\infty\frac{\norm{R}_1^k}{k!}
 =
 \ee^{\norm{R}_1}.
\]
Since the joint gap probability is nonnegative, the stated upper bound follows.

For a regular conditional determinantal law, the same argument is a deterministic calculation with its conditional kernel outside the null set where the description fails. The last assertion follows from Theorem~\ref{thm:conditional-growth-kernel} by pushing forward under the map from the past matrix to its positive spectrum.
\end{proof}

\subsection{Fixed-particle Eynard--Mehta lemma}\label{app:proof-fixed-particle-em}

\begin{proof}[Proof of Lemma~\ref{lem:fixed-particle-em}]
We first prove the Andr\'eief identity in the form needed here. If $f_1,\ldots,f_P$ and $g_1,\ldots,g_P$ are integrable functions on one measure space and the products below are absolutely integrable, expansion of both determinants into permutations, followed by Fubini's theorem and a relabeling of the integration variables, gives
\begin{equation}\label{eq:andreief}
 \int
 \det[f_i(x_j)]\det[g_i(x_j)]
 \prod_{j=1}^P\mu(\dd x_j)
 =
 P!\det\left[\int f_i(x)g_j(x)\,\mu(\dd x)\right].
\end{equation}
Repeatedly apply \eqref{eq:andreief}, first at level zero and then successively through level $T$. Integrating a transfer determinant against the determinant immediately to its left convolves their entries. Hence
\begin{equation}\label{eq:em-normalization}
 Z=(P!)^{T+1}\det M.
\end{equation}
Performing the same integrations on the two sides of any fixed intermediate level $t$ gives
\begin{equation}\label{eq:em-moment}
 M_{ij}
 =
 \int_{E_t}F_i^{(t)}(x)G_j^{(t)}(x)\,\mu_t(\dd x),
\end{equation}
so $M$ is independent of $t$. Since $Z>0$, \eqref{eq:em-normalization} gives $\det M\neq0$.

Let $X$ be the disjoint union of the $E_t$, equipped with the direct-sum measure. Let $h$ be a function on $X$ whose restriction to $E_t$ is bounded and equal to $h_t$, and assume that the marked integrals converge absolutely. Insert
\[
 \prod_{t=0}^T\prod_{a=1}^P
 \bigl(1+h_t(x_a^{(t)})\bigr)
\]
into the joint density. At each level, place the factor $1+h_t$ in the columns of one adjacent determinant and repeat \eqref{eq:andreief}. The result is
\begin{equation}\label{eq:em-marked-boundary}
 \E\left[
 \prod_{t=0}^T\prod_{a=1}^P
 \bigl(1+h_t(x_a^{(t)})\bigr)
 \right]
 =
 \frac{\det M(h)}{\det M},
\end{equation}
where $M(h)_{ij}$ is the one-particle path integral from $\phi_i$ to $\psi_j$ with a factor $1+h_t$ at every level.

We next prove the coefficient identity used to rewrite the finite matrix on the right of \eqref{eq:em-marked-boundary}. This argument does not use an operator Fredholm determinant. Let $\mathcal T$ be the kernel on $X$ whose block from $E_s$ to $E_t$ is
\[
 \ind_{\{s<t\}}W_{s,t}.
\]
Let $H$ denote insertion of $h$ in an integral. Define
\[
 U((s,x),j)=G_j^{(s)}(x),
 \qquad
 V(i,(t,y))=F_i^{(t)}(y),
\]
where $V$ is followed by integration against $\mu_t$. Put
\begin{equation}\label{eq:em-calK}
 \mathcal K=-\mathcal T+UM^{-1}V.
\end{equation}
For any kernel $L$ on $X$, define its coefficientwise determinant series by
\begin{equation}\label{eq:em-coefficient-series}
 \mathfrak D_L(h)
 =
 \sum_{n\geq0}\frac1{n!}
 \int_{X^n}
 \det[L(q_a,q_b)]_{a,b=1}^n
 \prod_{a=1}^nh(q_a)
 \prod_{a=1}^n\mu(\dd q_a),
\end{equation}
where the term with $n=0$ is one. Formula \eqref{eq:em-coefficient-series} denotes only the explicitly displayed multiple-integral coefficients.

We claim that
\begin{equation}\label{eq:em-algebraic-identity}
 \mathfrak D_{\mathcal K}(h)
 =
 \det\left[
 I_P+M^{-1}VH(I-\mathcal TH)^{-1}U
 \right],
\end{equation}
where
\begin{equation}\label{eq:em-nilpotent-series}
 (I-\mathcal TH)^{-1}
 =
 \sum_{r=0}^T(\mathcal TH)^r.
\end{equation}
The sum is finite because every $\mathcal T$ factor strictly increases the time level.

To prove \eqref{eq:em-algebraic-identity} coefficientwise, first replace each $E_t$ by an arbitrary finite set, every integral by a finite weighted sum, and the kernels by the corresponding finite arrays. Ordinary finite-dimensional determinant algebra gives
\begin{align}
 \det[I+(-\mathcal T+UM^{-1}V)H]
 &=
 \det[I-\mathcal TH]\,
 \det[I+(I-\mathcal TH)^{-1}UM^{-1}VH] \notag\\
 &=
 \det[I_P+M^{-1}VH(I-\mathcal TH)^{-1}U].
 \label{eq:em-finite-algebra}
\end{align}
The first determinant in the middle expression is one because $I-\mathcal TH$ is block upper triangular with identity diagonal. The last equality is the rectangular identity $\det(I+AB)=\det(I+BA)$, which follows by expanding both finite determinants into minors. On the left of \eqref{eq:em-finite-algebra}, the coefficient of total degree $n$ in the point weights is the sum of all $n$-by-$n$ principal minors. A determinant with a repeated point index vanishes, so this coefficient is exactly the finite-sum version of the $n$th term in \eqref{eq:em-coefficient-series}.

For the measure spaces in the statement, expand both sides of \eqref{eq:em-finite-algebra} by the Leibniz formula, expand \eqref{eq:em-nilpotent-series}, and compare one fixed coefficient of total degree $n$. Each side is a finite sum indexed by permutations, boundary labels in $\{1,\ldots,P\}$, and strictly increasing time chains of length at most $T+1$. The finite-set proof matches these summands by the same permutation and chain reindexing. Replacing each finite sum over a point variable by its integral against the corresponding $\mu_t$ preserves the identity. Absolute convergence permits Fubini's theorem and all finite rearrangements. This proves \eqref{eq:em-algebraic-identity} coefficientwise. Its right side is a finite polynomial, so all coefficients in \eqref{eq:em-coefficient-series} beyond its degree vanish.

Expanding the marked one-particle path integral by choosing the nonempty ordered set $t_1<\cdots<t_k$ of levels at which $h$ is selected gives
\begin{equation}\label{eq:em-marked-matrix}
 M(h)
 =
 M+\sum_{k=1}^{T+1}VH(\mathcal TH)^{k-1}U
 =
 M+VH(I-\mathcal TH)^{-1}U.
\end{equation}
For example, the summand corresponding to $t_1<\cdots<t_k$ is the integral of
\[
 \begin{split}
 &F_i^{(t_1)}(x_1)h_{t_1}(x_1)
 W_{t_1,t_2}(x_1,x_2)h_{t_2}(x_2)\cdots\\
 &\qquad\cdots
 W_{t_{k-1},t_k}(x_{k-1},x_k)h_{t_k}(x_k)
 G_j^{(t_k)}(x_k)
 \end{split}
\]
over $x_1,\ldots,x_k$. Therefore
\begin{equation}\label{eq:em-generating-equality}
 \frac{\det M(h)}{\det M}
 =
 \det[I_P+M^{-1}VH(I-\mathcal TH)^{-1}U]
 =
 \mathfrak D_{\mathcal K}(h).
\end{equation}
Combining \eqref{eq:em-marked-boundary} and \eqref{eq:em-generating-equality} identifies the marked generating functional with the explicit coefficient series \eqref{eq:em-coefficient-series}.

It remains to identify the correlation densities. The process has exactly $P(T+1)$ points on $X$, so its marked generating functional has the finite factorial-moment expansion
\begin{equation}\label{eq:em-factorial-moments}
 \begin{aligned}
 \E\left[\prod_{q\text{ in the process}}(1+h(q))\right]
 &=\sum_{n=0}^{P(T+1)}\frac1{n!}\int_{X^n}
 \rho_n(q_1,\ldots,q_n)\\
 &\qquad{}\times\prod_{a=1}^nh(q_a)
 \prod_{a=1}^n\mu(\dd q_a).
 \end{aligned}
\end{equation}
Compare \eqref{eq:em-factorial-moments} with \eqref{eq:em-coefficient-series} and \eqref{eq:em-generating-equality}. For bounded measurable $f_1,\ldots,f_n$ supported on prescribed levels and independent scalars $z_1,\ldots,z_n$, substitute
\[
 h=\sum_{r=1}^nz_rf_r
\]
and compare the coefficient of $z_1\cdots z_n$. Both $\rho_n$ and $\det[\mathcal K(q_a,q_b)]$ are symmetric in the variables $q_a$. Consequently,
\[
 \begin{split}
 &\int_{X^n}\rho_n(q_1,\ldots,q_n)
 \prod_{r=1}^nf_r(q_r)\prod_r\mu(\dd q_r)\\
 &\qquad=
 \int_{X^n}\det[\mathcal K(q_a,q_b)]_{a,b=1}^n
 \prod_{r=1}^nf_r(q_r)\prod_r\mu(\dd q_r).
 \end{split}
\]
Taking the $f_r$ to be indicators of measurable finite-measure sets on chosen levels and applying a monotone-class argument on the sigma-finite product space yields
\[
 \rho_n(q_1,\ldots,q_n)
 =
 \det[\mathcal K(q_a,q_b)]_{a,b=1}^n
\]
for almost every $n$-tuple and every $n$. By \eqref{eq:em-calK}, $\mathcal K$ is the kernel stated in the lemma.
\end{proof}

\subsection{Boundary moment evaluation}\label{app:proof-boundary-moment}

\begin{proof}[Proof of Lemma~\ref{lem:boundary-moment}]
A nonzero one-particle path from boundary row $i$ is real from time $\tau_i$ onward. If $i=1$, it begins at the unique real initial state. If $i>1$, it remains at $v_i$ until the step at which $p$ increases from $i-1$ to $i$. Every subsequent real-real transfer has entry $\ind_{\{y>x\}}$. Consequently,
\begin{equation}\label{eq:ordered-moment-integral}
 M_{ij}
 =
 \int_{0<x_{\tau_i}<\cdots<x_T}
 \left(\prod_{t=\tau_i}^{T-1}x_t^{\kappa_t}\right)
 \ee^{-x_T}x_T^{\kappa_T+P-j}
 \,\dd x_{\tau_i}\cdots\dd x_T.
\end{equation}

For $\tau_i\leq k\leq T$, define
\[
 A_{i,k}
 =
 \sum_{t=\tau_i}^k\kappa_t+k-\tau_i+1.
\]
For $\tau_i\leq k<T$, put
\[
 b=\sum_{u=\tau_i}^{k-1}\epsilon_u,
 \qquad
 L=k-\tau_i.
\]
At a birth step $p$ increases by one and $r$ decreases by one. At a same-size step $p$ is unchanged and $r$ increases by one. Therefore
\begin{equation}\label{eq:r-evolution}
 r_k=\rho_i+L-2b.
\end{equation}
If $\tau_i>0$, the step immediately before $\tau_i$ is a birth step, so
\[
 \epsilon_{\tau_i-1}=1,
 \qquad
 r_{\tau_i-1}-1=r_{\tau_i}=\rho_i.
\]
If $\tau_i=0$, then $\rho_i=r_0=0$. Telescoping the definition of $\kappa_t$ gives
\begin{equation}\label{eq:kappa-prefix}
 \sum_{t=\tau_i}^k\kappa_t
 =
 \rho_i-b-\epsilon_kr_k.
\end{equation}
If $\epsilon_k=0$, then
\begin{equation}\label{eq:A-same}
 A_{i,k}
 =
 \rho_i-b+L+1
 =
 q_k-i+1.
\end{equation}
If $\epsilon_k=1$, then
\begin{equation}\label{eq:A-birth}
 A_{i,k}=b+1=p_k-i+1.
\end{equation}
Both values are positive integers.

At the terminal time,
\[
 \sum_{t=\tau_i}^T\kappa_t
 =
 \rho_i-\sum_{u=\tau_i}^{T-1}\epsilon_u
 =
 \rho_i-(P-i).
\]
There are $P-i$ birth steps and $Q-a_i$ same-size steps after $\tau_i$, so
\[
 T-\tau_i=(P-i)+(Q-a_i).
\]
Since $a_i=i+\rho_i$,
\begin{equation}\label{eq:A-terminal}
 A_{i,T}=Q-i+1.
\end{equation}

Successive integration in \eqref{eq:ordered-moment-integral} gives
\begin{equation}\label{eq:moment-gamma-product}
 M_{ij}
 =
 \frac{\Gamma(A_{i,T}+P-j)}
      {\prod_{k=\tau_i}^{T-1}A_{i,k}}.
\end{equation}
At the birth steps after $\tau_i$, the values of $p_k$ immediately before the increments are $i,\ldots,P-1$. Hence
\[
 \prod_{\substack{\tau_i\leq k<T\\\epsilon_k=1}}A_{i,k}
 =
 (P-i)!.
\]
At the same-size steps, the values of $q_k$ immediately before the increments are $a_i,\ldots,Q-1$. Hence
\[
 \prod_{\substack{\tau_i\leq k<T\\\epsilon_k=0}}A_{i,k}
 =
 \frac{(Q-i)!}{(a_i-i)!}
 =
 \frac{(Q-i)!}{\rho_i!}.
\]
Finally,
\[
 \Gamma(A_{i,T}+P-j)
 =
 \Gamma(P+Q-i-j+1)
 =
 (P+Q-i-j)!.
\]
Substitution into \eqref{eq:moment-gamma-product} proves \eqref{eq:closed-moment-matrix}.
\end{proof}

\subsection{Laguerre Hankel inversion}\label{app:proof-laguerre-hankel}

\begin{proof}[Proof of Lemma~\ref{lem:laguerre-hankel}]
Let $J$ reverse the coordinates of $\R^P$, let $G=JHJ$, and number the coordinates of $G$ from zero to $P-1$. Then
\[
 G_{ab}
 =
 (r+a+b)!
 =
 \int_0^\infty u^{a+b+r}\ee^{-u}\,\dd u.
\]
Thus $G$ is the Gram matrix of
\[
 m(u)=(1,u,\ldots,u^{P-1})^{\mathsf T}
\]
for the inner product
\[
 \langle f,g\rangle_r
 =
 \int_0^\infty f(u)g(u)\ee^{-u}u^r\,\dd u.
\]
The weight is positive on $(0,\infty)$, so every nonzero coefficient vector has positive quadratic form. Hence $G$, and therefore $H$, is invertible.

The finite definition of $L_k^{(r)}$ and the Leibniz formula give
\begin{equation}\label{eq:laguerre-rodrigues}
 \ee^{-u}u^rL_k^{(r)}(u)
 =
 \frac1{k!}
 \left(\frac{\dd}{\dd u}\right)^k
 [\ee^{-u}u^{k+r}].
\end{equation}
Indeed, the coefficient of $\ee^{-u}u^{r+\ell}$ on the right is
\[
 (-1)^\ell
 \frac{(k+r)!}{(k-\ell)!(r+\ell)!\ell!}.
\]
If $f$ has degree less than $k$, insert \eqref{eq:laguerre-rodrigues} into the inner product and integrate by parts $k$ times. All boundary terms vanish at infinity because of exponential decay. At zero, before the final integration the differentiated power has exponent at least $r+1$, so those boundary terms vanish as well. Thus $L_k^{(r)}$ is orthogonal to every polynomial of degree less than $k$. Since its leading coefficient is $(-1)^k/k!$, the same calculation gives
\begin{equation}\label{eq:laguerre-orthogonality}
 \langle L_k^{(r)},L_l^{(r)}\rangle_r
 =
 \ind_{\{k=l\}}\frac{(k+r)!}{k!}.
\end{equation}

The reproducing kernel of polynomials of degree less than $P$ is therefore both
\[
 R_P(x,y)=m(x)^{\mathsf T}G^{-1}m(y)
\]
and
\[
 R_P(x,y)
 =
 \sum_{k=0}^{P-1}
 \frac{k!}{(k+r)!}
 L_k^{(r)}(x)L_k^{(r)}(y).
\]
Since $H^{-1}=JG^{-1}J$,
\[
 m(x)^{\mathsf T}G^{-1}m(y)
 =
 \sum_{i,j=1}^P
 x^{P-j}(H^{-1})_{ji}y^{P-i}.
\]
Multiplying by $y^r$ proves \eqref{eq:hankel-laguerre-inversion}.

Coefficient comparison in the defining finite sum gives, with $L_{-1}^{(r)}=0$,
\begin{equation}\label{eq:laguerre-recurrence}
 (k+1)L_{k+1}^{(r)}(u)
 =
 (2k+r+1-u)L_k^{(r)}(u)
 -(k+r)L_{k-1}^{(r)}(u).
\end{equation}
Set
\[
 C_k
 =
 L_k^{(r)}(x)L_{k+1}^{(r)}(y)
 -
 L_{k+1}^{(r)}(x)L_k^{(r)}(y).
\]
Applying \eqref{eq:laguerre-recurrence} at $x$ and $y$ gives
\[
 (x-y)L_k^{(r)}(x)L_k^{(r)}(y)
 =
 (k+1)C_k-(k+r)C_{k-1}.
\]
Multiply by $k!/(k+r)!$ and sum over $0\leq k<P$. The adjacent terms cancel, leaving
\[
 (x-y)R_P(x,y)
 =
 \frac{P!}{(P+r-1)!}C_{P-1},
\]
which proves \eqref{eq:christoffel-darboux} off the diagonal and gives its continuous diagonal extension.

To identify the terminal block, use $\tau_i$ and $\rho_i$ from Lemma~\ref{lem:boundary-moment}. A terminal forward path is real after $\tau_i$, so
\[
 F_i^{(T)}(y)
 =
 \ee^{-y}
 \int_{0<x_{\tau_i}<\cdots<x_{T-1}<y}
 \prod_{u=\tau_i}^T x_u^{\kappa_u}
 \,\dd x_{\tau_i}\cdots\dd x_{T-1}.
\]
The accumulated-exponent calculation in the proof of Lemma~\ref{lem:boundary-moment} yields
\[
 F_i^{(T)}(y)
 =
 \frac{\rho_i!}{(P-i)!(Q-i)!}
 \ee^{-y}y^{Q-i}.
\]
That lemma gives $M=DH$, where
\[
 D_{ii}=\frac{\rho_i!}{(P-i)!(Q-i)!}.
\]
Thus $M^{-1}=H^{-1}D^{-1}$. Since
\[
 G_j^{(T)}(x)=x^{P-j},
\]
the terminal finite-rank block is
\[
 \sum_{i,j=1}^P
 x^{P-j}(M^{-1})_{ji}F_i^{(T)}(y)
 =
 \ee^{-y}
 \sum_{i,j=1}^P
 x^{P-j}(H^{-1})_{ji}y^{Q-i}.
\]
This proves the last assertion.
\end{proof}

\subsection{Forward gauge cancellation}\label{app:proof-forward-gauge}

\begin{proof}[Proof of Lemma~\ref{lem:forward-gauge-cancellation}]
Fix $t$ and $1\leq i\leq p_t$, and write $\tau=\tau_i$. A nonzero path from boundary label $i$ is real from time $\tau$ onward. With $x_t=x$,
\begin{equation}\label{eq:forward-ordered-integral}
 F_i^{(t)}(x)
 =
 h_t(x)
 \int_{0<x_\tau<\cdots<x_{t-1}<x_t}
 \prod_{u=\tau}^t x_u^{\kappa_u}
 \,\dd x_\tau\cdots\dd x_{t-1}.
\end{equation}
Define
\[
 A_{i,k}
 =
 \sum_{u=\tau}^k\kappa_u+k-\tau+1.
\]
The telescoping calculation in Lemma~\ref{lem:boundary-moment} gives
\[
 A_{i,k}=
 \begin{cases}
 q_k-i+1,&k<T,\ \epsilon_k=0,\\
 p_k-i+1,&k<T,\ \epsilon_k=1,\\
 Q-i+1,&k=T.
 \end{cases}
\]
Successive integration in \eqref{eq:forward-ordered-integral} gives
\[
 F_i^{(t)}(x)
 =
 \frac{h_t(x)x^{A_{i,t}-1}}
      {\prod_{k=\tau}^{t-1}A_{i,k}}.
\]
At the birth steps before $t$, the successive values of $p_k$ are
\[
 i,i+1,\ldots,p_t-1,
\]
so
\[
 \prod_{\substack{\tau\leq k<t\\\epsilon_k=1}}A_{i,k}
 =
 (p_t-i)!.
\]
At the same-size steps, the successive values of $q_k$ are
\[
 a_i,a_i+1,\ldots,q_t-1,
\]
so
\[
 \prod_{\substack{\tau\leq k<t\\\epsilon_k=0}}A_{i,k}
 =
 \frac{(q_t-i)!}{\rho_i!}.
\]
This proves \eqref{eq:explicit-forward}.

If $i>p_t$, the boundary label begins at $v_i$ and has not yet reached its birth step. Every preceding augmented transfer preserves $v_i$, sends it to no other virtual state, and has zero real-virtual cross terms. Hence
\[
 F_i^{(t)}(x)=0\quad(x>0),
 \qquad
 F_i^{(t)}(v_k)=\ind_{\{i=k\}}.
\]

Lemma~\ref{lem:boundary-moment} gives
\[
 M=DH,
 \qquad
 D_{ii}=\frac{\rho_i!}{(P-i)!(Q-i)!}.
\]
Since $M$ is invertible, $D$ and $H$ are invertible and
\[
 M^{-1}=H^{-1}D^{-1}.
\]
Multiplying \eqref{eq:explicit-forward} by this identity proves \eqref{eq:gauge-cancellation}.
\end{proof}

\section{Standard Airy envelope estimates}\label{app:gaussian-airy-envelopes}

\subsection{Airy envelope estimates}\label{app:proof-airy-envelopes}

\begin{proof}[Proof of Lemma~\ref{lem:airy-envelopes}]
Write \(f=\Ai\). We first record the zero-data uniqueness argument used below. Let \(y''(x)=xy(x)\), and suppose that
\[
y(x_0)=y'(x_0)=0.
\]
On any compact interval \([r,R]\) containing \(x_0\), put
\[
E_0(x)=y(x)^2+y'(x)^2,
\qquad
L_0=\max_{r\leq x\leq R}\abs{1+x}.
\]
Then
\[
\abs{E_0'(x)}
=\abs{2(1+x)y(x)y'(x)}
\leq L_0E_0(x).
\]
For \(x\geq x_0\), the derivative of
\[
\exp\bigl(-L_0(x-x_0)\bigr)E_0(x)
\]
is nonpositive. Its initial value is zero, while \(E_0\) is nonnegative, so \(E_0(x)=0\). After the change of variables \(x=x_0-s\), the same calculation gives \(E_0(x)=0\) for \(x\leq x_0\). Since \([r,R]\) was arbitrary, zero initial data force \(y\) to vanish everywhere.

We next consider \(x\geq1\). A nonzero solution \(f\) of \(f''=xf\) that tends to zero cannot have two zeros in \([1,\infty)\). Indeed, if \(1\leq a<b\) and \(f(a)=f(b)=0\), then multiplication of the equation by \(f\) and integration give
\[
-\int_a^b f'(x)^2\dd x
=
\int_a^b xf(x)^2\dd x.
\]
Equality is possible only if \(f\) and \(f'\) vanish throughout \([a,b]\). The zero-data argument would then force \(f\) to vanish everywhere, which is a contradiction.

If \(f\) had one zero \(a\geq1\), its derivative there would be nonzero by zero-data uniqueness. Beyond \(a\), the function would have a fixed sign. If \(f\) were positive just after \(a\), then \(f'(a)>0\) and \(f''=xf>0\) while \(f\) remains positive. Thus \(f'\) would stay positive, and \(f\) could not tend to zero. If \(f\) were negative just after \(a\), then \(f'(a)<0\) and \(f''<0\), which gives the same contradiction after reversing signs. Therefore \(f\) has a fixed sign on \([1,\infty)\). Multiplying by \(-1\) if necessary, assume that \(f>0\) there.

It follows that \(f''>0\), so \(f'\) is increasing. Convergence of \(f\) to zero forces \(f'<0\) and \(f'(x)\to0\). For \(x\geq1\),
\[
\frac{\dd}{\dd x}
\left[\exp(-x)\bigl(f'(x)+f(x)\bigr)\right]
=
\exp(-x)(x-1)f(x)\geq0.
\]
The expression on the left tends to zero at infinity. Hence
\[
f'(x)+f(x)\leq0.
\]
Therefore
\[
f(x)\leq f(1)\exp(-(x-1)).
\]
Since
\[
f'(x)=-\int_x^\infty tf(t)\dd t,
\]
the same bound gives
\[
\abs{f'(x)}\leq C(1+x)\exp(-x).
\]

Put
\[
\xi=\frac23x^{3/2},\qquad a_{\Ai}=\frac5{36},
\qquad h_+(\xi)=x^{1/4}f(x).
\]
Direct differentiation, using \(\dd\xi/\dd x=\sqrt{x}\), gives
\begin{equation}\label{eq:airy-positive-ode}
h_+''(\xi)=\left(1-a_{\Ai}\xi^{-2}\right)h_+(\xi).
\end{equation}
The preceding exponential bounds show that \(h_+\) and \(h_+'\) tend to zero. For \(x\geq1\), one has \(\xi\geq2/3\) and
\[
1-a_{\Ai}\xi^{-2}\geq\frac{11}{16}=m_{\Ai}^2,
\qquad
m_{\Ai}=\frac{\sqrt{11}}4.
\]
Consequently,
\[
\begin{aligned}
&\frac{\dd}{\dd\xi}
\left[\exp(-m_{\Ai}\xi)
\bigl(h_+'+m_{\Ai}h_+\bigr)\right]\\
&\qquad=\exp(-m_{\Ai}\xi)
\left[\left(1-a_{\Ai}\xi^{-2}\right)-m_{\Ai}^2\right]h_+\geq0.
\end{aligned}
\]
Its limit at infinity is zero, so
\[
h_+'+m_{\Ai}h_+\leq0.
\]
Hence, for \(t\geq\xi\),
\begin{equation}\label{eq:airy-positive-decay}
h_+(t)\leq h_+(\xi)\exp\bigl(-m_{\Ai}(t-\xi)\bigr).
\end{equation}

Equation~\eqref{eq:airy-positive-ode} also gives
\[
\frac{\dd}{\dd\xi}
\left[
\exp(-\xi)(h_+'+h_+)
\right]
=
-a_{\Ai}\xi^{-2}\exp(-\xi)h_+.
\]
Integrating from \(\xi\) to infinity yields
\begin{equation}\label{eq:airy-positive-integral}
\exp(-\xi)\bigl(h_+'(\xi)+h_+(\xi)\bigr)
=
a_{\Ai}\int_\xi^\infty
t^{-2}\exp(-t)h_+(t)\dd t.
\end{equation}
Let
\[
q_{\Ai}(\xi)=\exp(\xi)h_+(\xi).
\]
Equations~\eqref{eq:airy-positive-decay} and~\eqref{eq:airy-positive-integral} imply
\begin{equation}\label{eq:airy-q-bound}
0\leq\frac{q_{\Ai}'(\xi)}{q_{\Ai}(\xi)}
\leq
a_{\Ai}\exp((1+m_{\Ai})\xi)
\int_\xi^\infty
t^{-2}\exp(-(1+m_{\Ai})t)\dd t
\leq
\frac{a_{\Ai}}{1+m_{\Ai}}\xi^{-2}.
\end{equation}
The right side is integrable at infinity. Thus \(q_{\Ai}\) is positive and nondecreasing, and \eqref{eq:airy-q-bound} bounds \(\log q_{\Ai}\) uniformly from above. It follows that \(q_{\Ai}\) has a finite positive limit. Integrating \eqref{eq:airy-q-bound} from \(\xi\) to infinity also shows that the ratio of this limit to \(q_{\Ai}(\xi)\) is at most
\[
\exp\left(\frac{a_{\Ai}}{(1+m_{\Ai})\xi}\right).
\]
Therefore \(q_{\Ai}\) is bounded above and below by positive constants on \([2/3,\infty)\). Since
\[
f(x)=x^{-1/4}\exp(-\xi)q_{\Ai}(\xi),
\]
this proves the two-sided positive-half-line estimate.

For the negative half-line, put \(g(x)=f(-x)\), so that
\[
g''(x)+xg(x)=0.
\]
For \(x\geq1\), use the same \(\xi\) and put
\[
h_-(\xi)=x^{1/4}g(x).
\]
Direct differentiation gives
\begin{equation}\label{eq:airy-negative-ode}
h_-''(\xi)+\left(1+a_{\Ai}\xi^{-2}\right)h_-(\xi)=0.
\end{equation}
For
\[
E_{\Ai}(\xi)=h_-'(\xi)^2+h_-(\xi)^2,
\]
equation~\eqref{eq:airy-negative-ode} gives
\[
\abs{E_{\Ai}'(\xi)}
=
\abs{2a_{\Ai}\xi^{-2}h_-(\xi)h_-'(\xi)}
\leq a_{\Ai}\xi^{-2}E_{\Ai}(\xi).
\]
Put \(\xi_0=2/3\). Since
\[
E_{\Ai}'(\xi)\leq a_{\Ai}\xi^{-2}E_{\Ai}(\xi),
\]
direct differentiation gives
\[
\frac{\dd}{\dd\xi}
\left\{
\exp\left[-\int_{\xi_0}^{\xi}a_{\Ai}t^{-2}\dd t\right]
E_{\Ai}(\xi)
\right\}
\leq0.
\]
Therefore
\[
E_{\Ai}(\xi)
\leq
E_{\Ai}(\xi_0)
\exp\left[\int_{\xi_0}^{\xi}a_{\Ai}t^{-2}\dd t\right]
\leq
E_{\Ai}(\xi_0)\exp(a_{\Ai}/\xi_0).
\]
Thus \(E_{\Ai}\) is bounded on \([2/3,\infty)\). Hence \(\abs{h_-(\xi)}\) is bounded and
\[
\abs{f(-x)}\leq Cx^{-1/4},
\qquad x\geq1.
\]
Continuity of \(f\) on \([-1,0]\) enlarges the constant to give
\[
\abs{f(-x)}\leq C(1+x)^{-1/4},
\qquad x\geq0.
\]
Restoring the original sign of \(f\) does not change any absolute-value estimate.
\end{proof}

\end{document}